\documentclass[reqno,11pt]{amsart}
\usepackage{amscd,amssymb,verbatim,array}
\usepackage{hyperref}

\numberwithin{equation}{section} \theoremstyle{plain}

\newcommand{\Complex}{\mathbb C}
\newcommand{\Real}{\mathbb R}
\newcommand{\N}{\mathbb N}
\newcommand{\ddbar}{\overline\partial}
\newcommand{\pr}{\partial}
\newcommand{\ol}{\overline}
\newcommand{\Td}{\widetilde}
\newcommand{\norm}[1]{\left\Vert#1\right\Vert}
\newcommand{\set}[1]{\left\{#1\right\}}
\newcommand{\To}{\rightarrow}

\newtheorem{theorem}{Theorem}[section]
\newtheorem{lemma}[theorem]{Lemma}

\newtheorem{definition}[theorem]{Definition}
\newtheorem{ass}[theorem]{Assumption}

\theoremstyle{definition}

\theoremstyle{remark}
\newtheorem{remark}[theorem]{Remark}

\numberwithin{equation}{section}

\newcommand{\abs}[1]{\lvert#1\rvert}

\begin{document}

\title[$G$-invariant Szeg\"o kernel asymptotics  and CR reduction]
{$G$-invariant Szeg\"o kernel asymptotics and CR reduction}

\author{Chin-Yu Hsiao}

\address{Institute of Mathematics, Academia Sinica and National Center for Theoretical Sciences, Astronomy-Mathematics Building, No. 1, Sec. 4, Roosevelt Road, Taipei 10617, Taiwan}
\thanks{The first author was partially supported by Taiwan Ministry of Science of Technology project 104-2628-M-001-003-MY2 , the Golden-Jade fellowship of Kenda Foundation and Academia Sinica Career Development Award. This work was initiated when the second author was visiting the Institute of Mathematics at Academia Sinica in the summer of 2016. The second author would like to thank the Institute of Mathematics at Academia Sinica for its hospitality and financial support during his stay. The second author was also supported by Taiwan Ministry of Science of Technology project 105-2115-M-008-008-MY2}

\email{chsiao@math.sinica.edu.tw or chinyu.hsiao@gmail.com}

\author{Rung-Tzung Huang}

\address{Department of Mathematics, National Central University, Chung-Li 320, Taiwan}

\email{rthuang@math.ncu.edu.tw}


\begin{abstract}
Let $(X, T^{1,0}X)$ be a compact connected orientable CR manifold of dimension $2n+1$ with non-degenerate Levi curvature. Assume that $X$ admits a connected compact Lie group action $G$. Under certain natural assumptions about the group action $G$, we show that the $G$-invariant Szeg\"o kernel for $(0,q)$ forms is a complex Fourier integral operator, smoothing away $\mu^{-1}(0)$ and there is a precise description of the singularity near $\mu^{-1}(0)$, where $\mu$ denotes the CR moment map. We apply our result  to the case when $X$ admits a transversal CR $S^1$ action and 
deduce an asymptotic expansion for the $m$-th Fourier component of the $G$-invariant Szeg\"o kernel for $(0,q)$ forms as $m\To+\infty$. 
As an application, we show that if $m$ large enough, quantization commutes with reduction. 
\end{abstract}

\maketitle


\section{Introduction and statement of the main results}\label{s-gue170124}

Let $(X, T^{1,0}X)$ be a CR manifold of dimension $2n+1$, $n\geq1$.
Let $\Box^{(q)}_b$ be the Kohn Lalpacian acting on $(0,q)$ forms. 
The orthogonal projection $S^{(q)}:L^2_{(0,q)}(X)\To {\rm Ker\,}\Box^{(q)}_b$ onto ${\rm Ker\,}\Box^{(q)}_b$
is called the Szeg\"{o} projection, while its distribution kernel $S^{(q)}(x,y)$ is called the Szeg\"{o} kernel.
The study of the Szeg\"{o} projection and kernel is a classical subject in several complex variables and CR geometry.
A very important case is when $X$ is a compact strictly pseudoconvex
CR manifold. Assume first that $X$ is the boundary of a strictly pseudoconvex domain.
Boutet de Monvel-Sj\"ostrand~\cite{BouSj76} showed that $S^{(0)}(x,y)$
is a complex Fourier integral operator. 

The Boutet de Monvel-Sj\"ostrand description of the Szeg\"{o} kernel had a profound impact
in many research areas, especially through  \cite{BG81}: several complex variables, 
symplectic and contact geometry, geometric quantization, K\"ahler geometry, semiclassical analysis,
quantum chaos, etc. cf.\ \cite{Engl:02,Gu89,Ma10,Pa05, ShZ99,Zelditch98},
to quote just a few. These ideas also partly motivated the introduction of alternative approaches, 
see \cite{Ma10,MM06,MM08a,MM}.

Now, we consider a connected compact group action $G$ acting on $X$. The study of $G$-invariant Szeg\"o kernel is closely related to Mathematical physics and geometric quantization of CR manifolds. It is a fundamental problem  to establish $G$-invariant Boutet de Monvel-Sj\"ostrand type theorems for $G$-invariant Szeg\"o kernels and study the influence of the $G$-invariant Szeg\"o kernel. This is the motivation of this work. In this paper, we consider $G$-invariant Szeg\"o kernel for $(0,q)$ forms and we  show that the $G$-invariant Szeg\"o kernel for $(0,q)$ forms is a complex Fourier integral operator.   In particular,
$S^{(q)}(x,y)$ is smoothing outside $\mu^{-1}(0)$ and there is a precise description of the singularity near $\mu^{-1}(0)$, where $\mu$ denotes the CR moment map. We apply our result  to the case when $X$ admits a transversal CR $S^1$ action and 
deduce an asymptotic expansion for the $m$-th Fourier component of the Szeg\"o kernel for $(0,q)$ forms as $m\To+\infty$. 
As an application, we show that, if $m$ large enough, quantization commutes with reduction.  

In \cite{MZ}, Ma and Zhang have studied the asymptotic expansion of the invariant Bergman kernel of the $\operatorname{spin}^c$ Dirac operator associated with high tensor powers of a positive line bundle on a symplectic manifold admitting a Hamiltonian action of a compact connected Lie group and its relation to the asymptotic expansion of Bergman kernel on symplectic reduced space. Their approach is inspired by the analytic localization techniques developed by Bismut and Lebeau. 
About the "quantization commutes with
reduction problem in symplectic geometry, we refer the readers to~\cite{Ma10}.  In the second part of~\cite{Ma10}, Ma described how the $G$-invariant Bergman kernel €œconcentrates€ on the Bergman kernel of the reduced space. In particular, this is used to
address M. Vergne€™s conjecture~\cite{MZI}. 

In~\cite{Pa03}, ~\cite{Pa05},~\cite{Pa08},~\cite{Pa12}, ~\cite{Pa13} , Paoletti studied equivariant Szeg\"o kernel on complex manifold and the relation between the Szeg\"o kernel of an ample
line bundle on a complex projective manifold $M$ and the Szeg\"o kernel of
the induced polarization on the quotient of $M$ by the holomorphic action of
a compact Lie group $G$. The approach of Paolletti is based on Microlocal technique, especially from Boutet de Monvel-Sj\"ostrand~\cite{BouSj76}, Boutet de Monvel-Guillemin~\cite{BG81}. It should be also mentioned that Paolletti  obtained in~\cite{Pa05a}, the asymptotic growth of equivariant sections of positive and big line bundles.

We now formulate the main results. We refer to Section~\ref{s:prelim} for some notations and terminology used here.
Let $(X, T^{1,0}X)$ be a compact connected orientable CR manifold of dimension $2n+1$, $n\geq1$, where $T^{1,0}X$ denotes the CR structure of $X$. 
Fix a global non-vanishing real $1$-form $\omega_0\in C^\infty(X,T^*X)$ such that $\langle\,\omega_0\,,\,u\,\rangle=0$, for every $u\in T^{1,0}X\oplus T^{0,1}X$. The Levi form of $X$ at $x\in X$ is the Hermitian quadratic form on $T^{1,0}_xX$ given by $\mathcal{L}_x(U,\ol V)=-\frac{1}{2i}\langle\,d\omega_0(x)\,,\,U\wedge\ol V\,\rangle$, $\forall U, V\in T^{1,0}_xX$. In this work, we assume that 

\begin{ass}\label{a-gue170123}
The Levi form is non-degenerate of constant signature $(n_-,n_+)$ on $X$. That is, the Levi form  has exactly $n_-$ negative and $n_+$ positive eigenvalues at each point of $X$, where $n_-+n_+=n$. 
\end{ass}

Let $HX=\set{{\rm Re\,}u;\, u\in T^{1,0}X}$ and let $J:HX\To HX$ be the complex structure map given by $J(u+\ol u)=iu-i\ol u$, for every $u\in T^{1,0}X$. 
In this work, we assume that $X$ admits a $d$-dimensional locally free connected compact Lie group action $G$.  We assume throughout that

\begin{ass}\label{a-gue170123I}
The Lie group action $G$ preserves $\omega_0$ and $J$. That is, $g^\ast\omega_0=\omega_0$ on $X$ and $g_\ast J=Jg_\ast$ on $HX$, for every $g\in G$, where $g^*$ and $g_*$ denote  the pull-back map and push-forward map of $G$, respectively. 
\end{ass}

Let $\mathfrak{g}$ denote the Lie algebra of $G$. For any $\xi \in \mathfrak{g}$, we write $\xi_X$ to denote the vector field on $X$ induced by $\xi$. That is, $(\xi_X u)(x)=\frac{\partial}{\partial t}\left(u(\exp(t\xi)\circ x)\right)|_{t=0}$, for any $u\in C^\infty(X)$.

\begin{definition}\label{d-gue170124}
The moment map associated to the form $\omega_0$ is the map $\mu:X \to \mathfrak{g}^*$ such that, for all $x \in X$ and $\xi \in \mathfrak{g}$, we have 
\begin{equation}\label{E:cmpm}
\langle \mu(x), \xi \rangle = \omega_0(\xi_X(x)).
\end{equation}
\end{definition}

In this work, we assume that 

\begin{ass}\label{a-gue170123II}
$0$ is a regular value of $\mu$ and the action $G$ on $\mu^{-1}(0)$ is globally free. 
\end{ass}

By Assumption~\ref{a-gue170123II}, $\mu^{-1}(0)$ is a $d$-dimensional submanifold of $X$. Let $Y:=\mu^{-1}(0)$ and let $HY:=HX\bigcap TY$. In Section~\ref{s-gue170301}, we will show that ${\rm dim\,}(HY\bigcap JHY)=2n-2d$ at every point of $Y$. Moreover, we will show that $\mu^{-}(0)/G=:Y_G$ is a CR manifold with natural CR 
structure induced by $T^{1,0}X$ of dimension $2n-2d+1$ and we can identify $HY_G$ with $HY\bigcap JHY$. 

Let $\underline{\mathfrak{g}}={\rm Span\,}(\xi_X;\, \xi\in\mathfrak{g})$. For $x\in\mu^{-1}(0)$, $\underline{\mathfrak{g}}_x\subset H_x$. 
Fix a $G$-invariant smooth Hermitian metric $\langle \cdot \mid \cdot \rangle$ on $\mathbb{C}TX$ so that $T^{1,0}X$ is orthogonal to $T^{0,1}X$, $\underline{\mathfrak{g}}$ is orthogonal to $HY\bigcap JHY$ at every point of $Y$, $\langle u \mid v \rangle$ is real if $u, v$ are real tangent vectors, $\langle\,T\,|\,T\,\rangle=1$ and $T$ is orthogonal to $T^{1,0}X\oplus T^{0,1}X$, where $T$ is given by \eqref{e-gue170111ry}. The Hermitian metric $\langle \cdot | \cdot \rangle$ on $\mathbb{C}TX$ induces, by duality, a Hermitian metric on $\mathbb{C}T^*X$ and also on the bundles of $(0,q)$ forms $T^{*0,q}X, q=0, 1, \cdots, n$. We shall also denote all these induced metrics by $\langle \cdot | \cdot \rangle$. Fix $g\in G$. Let $g^*:\Lambda^r_x(\Complex T^*X)\To\Lambda^r_{g^{-1}\circ x}(\Complex T^*X)$ be the pull-back map. Since $G$ preserves $J$, we have 
\[g^*:T^{*0,q}_xX\To T^{*0,q}_{g^{-1}\circ x}X,\  \ \forall x\in X.\]
Thus, for $u\in\Omega^{0,q}(X)$, we have $g^*u\in\Omega^{0,q}(X)$. Put 
\[\Omega^{0,q}(X)^G:=\set{u\in\Omega^{0,q}(X);\, g^*u=u,\ \ \forall g\in G}.\]
Since the Hermitian metric $\langle\,\cdot\,|\,\cdot\,\rangle$ on $\Complex TX$ is $G$-invariant, the $L^2$ inner product $(\,\cdot\,|\,\cdot\,)$ on $\Omega^{0,q}(X)$ 
induced by $\langle\,\cdot\,|\,\cdot\,\rangle$ is $G$-invariant. Let $u\in L^2_{(0,q)}(X)$ and $g\in G$, we can also define $g^*u$ in the standard way (see the discussion in the beginning of Section~\ref{s-gue161109}). Put 
\[L^2_{(0,q)}(X)^G:=\set{u\in L^2_{(0,q)}(X);\, g^*u=u,\ \ \forall g\in G}.\]
Let $\Box^{(q)}_b : {\rm Dom\,}\Box^{(q)}_b\To L^2_{(0,q)}(X)$ be the Gaffney extension of Kohn Laplacian (see \eqref{e-suIX}). Put $({\rm Ker\,}\Box^{(q)}_b)^G:={\rm Ker\,}\Box^{(q)}_b\bigcap L^2_{(0,q)}(X)^G$. The $G$-invariant Szeg\"o projection is the orthogonal projection 
\[S^{(q)}_G:L^2_{(0,q)}(X)\To ({\rm Ker\,}\Box^{(q)}_b)^G\]
with respect to $(\,\cdot\,|\,\cdot\,)$. Let $S^{(q)}_G(x,y)\in D'(X\times X,T^{*0,q}X\boxtimes(T^{*0,q}X)^*)$ be the distribution kernel of $S^{(q)}_G$. The first main result of this work is the following 

\begin{theorem}\label{t-gue170124}
With the assumptions and notations above, suppose that $\Box^{(q)}_b : {\rm Dom\,}\Box^{(q)}_b\To L^2_{(0,q)}(X)$ has closed range. If $q\notin\set{n_-, n_+}$, then $S^{(q)}_G\equiv 0$ on $X$. 

Suppose  $q\in\set{n_-, n_+}$. Let $D$ be an open set of $X$ with $D\bigcap\mu^{-1}(0)=\emptyset$. Then, 
\[S^{(q)}_G\equiv0\ \ \mbox{on $D$}.\]
Let $p\in\mu^{-1}(0)$ and let $U$ be an open set of  $p$ and let $x=(x_1,\ldots,x_{2n+1})$ be local coordinates defined in $U$. 
Then, there exist continuous operators
\[S^G_-, S^G_+:\Omega^{0,q}_0(U)\To\Omega^{0,q}(U)\]
such that 
\begin{equation}\label{e-gue170108wrm}
S^{(q)}_G\equiv S^G_-+S^G_+\ \ \mbox{on $U$},
\end{equation}
\begin{equation}\label{e-gue170108wrIm}
\begin{split}
&S^G_-=0\ \ \mbox{if $q\neq n_-$},\\
&S^G_+=0\ \ \mbox{if $q\neq n_+$},\\
\end{split}
\end{equation}
and if $q=n_-$, $S^G_-(x,y)$ satisfies
\begin{equation}\label{e-gue170108wrIIm}
S^G_-(x, y)\equiv\int^{\infty}_{0}e^{i\Phi_-(x, y)t}a_-(x, y, t)dt\ \ \mbox{on $U$}
\end{equation}
with 
\begin{equation}  \label{e-gue170108wrIIIm}\begin{split}
&a_-(x, y, t)\in S^{n-\frac{d}{2}}_{1,0}(U\times U\times\mathbb{R}_+,T^{*0,q}X\boxtimes(T^{*0,q}X)^*), \\
&a_-(x, y, t)\sim\sum^\infty_{j=0}a^j_-(x, y)t^{n-\frac{d}{2}-j}\quad\text{ in }S^{n-\frac{d}{2}}_{1, 0}(U\times U\times\mathbb{R}_+,T^{*0,q}X\boxtimes(T^{*0,q}X)^*),\\
&a^j_-(x, y)\in C^\infty(U\times U,T^{*0,q}X\boxtimes(T^{*0,q}X)^*),\ \ j=0,1,2,3,\ldots,\\
&a^0_-(x,x)\neq0,\ \ \forall x\in U,
\end{split}\end{equation}
$a^0_-(x,x)$, $x\in\mu^{-1}(0)\bigcap U$, is given by \eqref{e-gue170128} below, $\Phi_-(x,y)\in C^\infty(U\times U)$, 
\begin{equation}\label{e-gue170125}
\begin{split}
&{\rm Im\,}\Phi_-(x,y)\geq0,\\
&d_x\Phi_-(x,x)=-d_y\Phi_-(x,x)=-\omega_0(x),\ \ \forall x\in U\bigcap\mu^{-1}(0),\\
\end{split}
\end{equation}
there is a constant $C\geq 1$ such that, for all $(x,y)\in U\times U$, 
\begin{equation}\label{e-gue170125I}
\begin{split}
&\abs{\Phi_-(x,y)}+{\rm Im\,}\Phi_-(x,y)\leq C \left( \inf\set{d^2(g\circ x,y);\, g\in G}+d^2(x,\mu^{-1}(0))+d^2(y,\mu^{-1}(0)) \right), \\
&\abs{\Phi_-(x,y)}+{\rm Im\,}\Phi_-(x,y)\geq\frac{1}{C} \left( \inf\set{d^2(g\circ x,y);\, g\in G}+d^2(x,\mu^{-1}(0))+d^2(y,\mu^{-1}(0)) \right),\\
&Cd^2(x,\mu^{-1}(0))\geq {\rm Im\,}\Phi_-(x,x)\geq\frac{1}{C}d^2(x,\mu^{-1}(0)),\ \ \forall x\in U, 
\end{split}
\end{equation}
and $\Phi_-(x,y)$ satisfies \eqref{e-gue170126}  below and \eqref{e-gue170126I} below. 

If $q=n_+$, then $S^G_+(x,y)$ satisfies 
\begin{equation}\label{e-gue170108waIm}
S^G_+(x, y)\equiv\int^{\infty}_{0}e^{i\Phi_+(x, y)t}a_+(x, y, t)dt\ \ \mbox{on $U$}
\end{equation}
with
\begin{equation}  \label{e-gue161110rIm}
\begin{split}
&a_+(x, y, t)\in S^{n-\frac{d}{2}}_{1,0}(U\times U\times\mathbb{R}_+,T^{*0,q}X\boxtimes(T^{*0,q}X)^*), \\
&a_+(x, y, t)\sim\sum^\infty_{j=0}a^j_+(x, y)t^{n-\frac{d}{2}-j}\quad\text{ in }S^{n-\frac{d}{2}}_{1, 0}(U\times U\times\mathbb{R}_+,T^{*0,q}X\boxtimes(T^{*0,q}X)^*),\\
&a^j_+(x, y)\in C^\infty(U\times U,T^{*0,q}X\boxtimes(T^{*0,q}X)^*),\ \ j=0,1,2,3,\ldots,\\
&a^0_+(x,x)\neq0,\ \ \forall x\in U,
\end{split}\end{equation}
$a^0_+(x,x)$, $x\in\mu^{-1}(0)\bigcap U$, is given by \eqref{e-gue170128} below, and $\Phi_+(x,y)\in C^\infty(U\times U)$, $-\ol\Phi_+(x,y)$ satisfies \eqref{e-gue170125}, \eqref{e-gue170125I}, \eqref{e-gue170126}  below and \eqref{e-gue170126I} below . 
\end{theorem}

We refer the reader to the discussion before \eqref{e-gue160507f} and Definition~\ref{d-gue140221a} for the precise meanings of $A\equiv B$ and the symbol space $S^{n-\frac{d}{2}}_{1,0}$, respectively. 


Let $\Phi\in C^\infty(U\times U)$.
We assume that $\Phi$ satisfies \eqref{e-gue170125}, \eqref{e-gue170125I}, \eqref{e-gue170126}, \eqref{e-gue170126I}. We will show in Theorem~\ref{t-gue140305II} that 
the functions
$\Phi$ and $\Phi_-$ are equivalent on $U$ in the sense of Definition~\ref{d-gue140305}
if and only if there is a function $f\in C^\infty(U\times U)$ with $f(x,x)=1$, for every $x\in\mu^{-1}(0)$, such that $\Phi(x,y)-f(x,y)\Phi_-(x,y)$
vanishes to infinite order at ${\rm diag\,}\Bigr((\mu^{-1}(0)\bigcap U)\times(\mu^{-1}(0)\bigcap U)\Bigr)$. From this observation, we see that the leading term $a^0_-(x,x)$, $x\in\mu^{-1}(0)$, is well-defined. To state the formula for $a^0_-(x,x)$, we introduce some notations. For a given point $x_0\in X$, let $\{W_j\}_{j=1}^{n}$ be an
orthonormal frame of $(T^{1,0}X,\langle\,\cdot\,|\,\cdot\,\rangle)$ near $x_0$, for which the Levi form
is diagonal at $x_0$. Put
\begin{equation}\label{levi140530}
\mathcal{L}_{x_0}(W_j,\ol W_\ell)=\mu_j(x_0)\delta_{j\ell}\,,\;\; j,\ell=1,\ldots,n\,.
\end{equation}
We will denote by
\begin{equation}\label{det140530}
\det\mathcal{L}_{x_0}=\prod_{j=1}^{n}\mu_j(x_0)\,.
\end{equation}
Let $\{T_j\}_{j=1}^{n}$ denote the basis of $T^{*0,1}X$, dual to $\{\ol W_j\}^{n}_{j=1}$. We assume that
$\mu_j(x_0)<0$ if\, $1\leq j\leq n_-$ and $\mu_j(x_0)>0$ if\, $n_-+1\leq j\leq n$. Put
\begin{equation}\label{n140530}
\begin{split}
&\mathcal{N}(x_0,n_-):=\set{cT_1(x_0)\wedge\ldots\wedge T_{n_-}(x_0);\, c\in\Complex},\\
&\mathcal{N}(x_0,n_+):=\set{cT_{n_-+1}(x_0)\wedge\ldots\wedge T_{n}(x_0);\, c\in\Complex}\end{split}
\end{equation}
and let
\begin{equation}\label{tau140530}
\begin{split}
\tau_{n_-}=\tau_{x_0,n_-}:T^{*0,q}_{x_0}X\To\mathcal{N}(x_0,n_-)\,,\quad
\tau_{n_+}=\tau_{x_0,n_+}:T^{*0,q}_{x_0}X\To\mathcal{N}(x_0,n_+)\,,\end{split}
\end{equation}
be the orthogonal projections onto $\mathcal{N}(x_0,n_-)$ and $\mathcal{N}(x_0,n_+)$
with respect to $\langle\,\cdot\,|\,\cdot\,\rangle$, respectively.

Fix $x\in\mu^{-1}(0)$, consider the linear map 
\[\begin{split}
R_x:\underline{\mathfrak{g}}_x&\To\underline{\mathfrak{g}}_x,\\
u&\To R_xu,\ \ \langle\,R_xu\,|\,v\,\rangle=\langle\,d\omega_0(x)\,,\,Ju\wedge v\,\rangle.
\end{split}\]
Let $\det R_x=\lambda_1(x)\cdots\lambda_d(x)$, where $\lambda_j(x)$, $j=1,2,\ldots,d$, are the eigenvalues of $R_x$. 

Fix $x\in\mu^{-1}(0)$, put $Y_x=\set{g\circ x;\, g\in G}$. $Y_x$ is a $d$-dimensional submanifold of $X$. The $G$-invariant Hermitian metric $\langle\,\cdot\,|\,\cdot\,\rangle$ induces a volume form $dv_{Y_x}$ on $Y_x$. Put 
\begin{equation}\label{e-gue170108em}
V_{{\rm eff\,}}(x):=\int_{Y_x}dv_{Y_x}.
\end{equation}

\begin{theorem}\label{t-gue170128}
With the notations used above, for $a^0_-(x,y)$ and $a^0_+(x,y)$ in \eqref{e-gue170108wrIIIm} and \eqref{e-gue161110rIm}, we have 
\begin{equation}\label{e-gue170128}
\begin{split}
&a^0_-(x,x)=2^{d-1}\frac{1}{V_{{\rm eff\,}}(x)}\pi^{-n-1+\frac{d}{2}}\abs{\det R_x}^{-\frac{1}{2}}\abs{\det\mathcal{L}_{x}}\tau_{x,n_-},\ \ \forall x\in\mu^{-1}(0),\\
&a^0_+(x,x)=2^{d-1}\frac{1}{V_{{\rm eff\,}}(x)}\pi^{-n-1+\frac{d}{2}}\abs{\det R_x}^{-\frac{1}{2}}\abs{\det\mathcal{L}_{x}}\tau_{x,n_+},\ \ \forall x\in\mu^{-1}(0).
\end{split}
\end{equation}
\end{theorem}

We now assume that $X$ admits an $S^1$ action: $S^1\times X\rightarrow X$. We write $e^{i\theta}$ to denote the $S^1$ action. Let $T\in C^\infty(X, TX)$ be the global real vector field induced by the $S^1$ action given by
$(Tu)(x)=\frac{\partial}{\partial\theta}\left(u(e^{i\theta}\circ x)\right)|_{\theta=0}$, $u\in C^\infty(X)$. We assume that  the $S^1$ action $e^{i\theta}$ is CR and transversal (see Definition~\ref{d-gue160502}). We take $\omega_0\in C^\infty(X,T^*X)$ to be the global real one form determined by $\langle\,\omega_0\,,\,u\,\rangle=0$, for every $u\in T^{1,0}X\oplus T^{0,1}X$ and $\langle\,\omega_0\,,\,T\,\rangle=-1$. In this paper, we assume that 

\begin{ass}\label{a-gue170128}
\begin{equation}\label{e-gue170111ryI}
\mbox{$T$ is transversal to the space $\underline{\mathfrak{g}}$ at every point $p\in\mu^{-1}(0)$},
\end{equation}
\begin{equation}\label{e-gue170111ryII}
e^{i\theta}\circ g\circ x=g\circ e^{i\theta}\circ x,\  \ \forall x\in X,\ \ \forall\theta\in[0,2\pi[,\ \ \forall g\in G, 
\end{equation}
and 
\begin{equation}\label{e-gue170117t}
\mbox{$G\times S^1$ acts globally free near $\mu^{-1}(0)$}. 
\end{equation}
\end{ass}

Fix $\theta_0\in]-\pi, \pi[$, $\theta_0$ small. 
Let $(e^{i\theta_0})^*:\Lambda^r(\Complex T^*X)\To\Lambda^r(\Complex T^*X)$ be the pull-back map by $e^{i\theta_0}$, $r=0,1,\ldots,2n+1$. It is easy to see that, for every $q=0,1,\ldots,n$, one has
\begin{equation}\label{e-gue150508faIm}
(e^{i\theta_0})^*:T^{*0,q}_{e^{i\theta_0}x}X\To T^{*0,q}_{x}X.
\end{equation}
Let $u\in\Omega^{0,q}(X)$ be arbitrary. Define
\begin{equation}\label{e-gue150508faIIm}
Tu:=\frac{\pr}{\pr\theta}\bigr((e^{i\theta})^*u\bigr)|_{\theta=0}\in\Omega^{0,q}(X).
\end{equation}
For every $m\in\mathbb Z$, let
\begin{equation}\label{e-gue150508dIm}
\begin{split}
&\Omega^{0,q}_m(X):=\set{u\in\Omega^{0,q}(X);\, Tu=imu},\ \ q=0,1,2,\ldots,n,\\
&\Omega^{0,q}_{m}(X)^G=\set{u\in\Omega^{0,q}(X)^G;\, Tu=imu},\ \ q=0,1,2,\ldots,n.
\end{split}
\end{equation}
We denote $C^\infty_m(X):=\Omega^{0,0}_m(X)$, $C^\infty_m(X)^G:=\Omega^{0,0}_m(X)^G$. From the CR property of the $S^1$ action and \eqref{e-gue170111ryII}, it is not difficult to see that 
\[Tg^*\ddbar_b=g^*T\ddbar_b=\ddbar_bg^*T=\ddbar_bTg^*\ \ \mbox{on $\Omega^{0,q}(X)$},\ \ \forall g\in G.\]
Hence,
\begin{equation}\label{e-gue160527m}
\ddbar_b:\Omega^{0,q}_m(X)^G\To\Omega^{0,q+1}_m(X)^G,\ \ \forall m\in\mathbb Z.
\end{equation}
We now assume that the Hermitian metric $\langle\,\cdot\,|\,\cdot\,\rangle$ on $\Complex TX$ is $G\times S^1$ invariant.  Then the $L^2$ inner product $(\,\cdot\,|\,\cdot\,)$ on $\Omega^{0,q}(X)$ 
induced by $\langle\,\cdot\,|\,\cdot\,\rangle$ is $G\times S^1$-invariant. We then have 
\[\begin{split}
&Tg^*\ol{\pr}^*_b=g^*T\ol{\pr}^*_b=\ol{\pr}^*_bg^*T=\ol{\pr}^*_bTg^*\ \ \mbox{on $\Omega^{0,q}(X)$},\ \ \forall g\in G,\\
&Tg^*\Box^{(q)}_b=g^*T\Box^{(q)}_b=\Box^{(q)}_bg^*T=\Box^{(q)}_bTg^*\ \ \mbox{on $\Omega^{0,q}(X)$},\ \ \forall g\in G,
\end{split}\]
where $\ol{\pr}^*_b$ is the $L^2$ adjoint of $\ddbar_b$ with respect to $(\,\cdot\,|\,\cdot\,)$. 

Let $L^2_{(0,q), m}(X)^G$ be
the completion of $\Omega_m^{0,q}(X)^G$ with respect to $(\,\cdot\,|\,\cdot\,)$. 
We write $L^2_m(X)^G:=L^2_{(0,0),m}(X)^G$. Put 
\[H^q_{b,m}(X)^G:=({\rm Ker\,}\Box^{(q)}_b)^G_m:=({\rm Ker\,}\Box^{(q)}_b)^G\bigcap L^2_{(0,q),m}(X)^G.\]
It is not difficult to see that, for every $m\in\mathbb Z$, $({\rm Ker\,}\Box^{(q)}_b)^G_m\subset\Omega^{0,q}_m(X)^G$ and ${\rm dim\,}({\rm Ker\,}\Box^{(q)}_b)^G_m<\infty$.
The $m$-th $G$-invariant Szeg\"o projection is the orthogonal projection 
\[S^{(q)}_{G,m}:L^2_{(0,q)}(X)\To ({\rm Ker\,}\Box^{(q)}_b)^G_m\]
with respect to $(\,\cdot\,|\,\cdot\,)$. Let $S^{(q)}_{G,m}(x,y)\in C^\infty(X\times X,T^{*0,q}X\boxtimes(T^{*0,q}X)^*)$ be the distribution kernel of $S^{(q)}_{G,m}$. The second main result of this work is the following 

\begin{theorem}\label{t-gue170128I}
With the assumptions and notations used above, if $q\notin n_-$, then, as $m\To+\infty$, $S^{(q)}_{G,m}=O(m^{-\infty})$ on $X$. 

Suppose  $q=n_-$. Let $D$ be an open set of $X$ with $D\bigcap\mu^{-1}(0)=\emptyset$. Then, as $m\To+\infty$, 
\[S^{(q)}_{G,m}=O(m^{-\infty})\ \ \mbox{on $D$}.\]
Let $p\in\mu^{-1}(0)$ and let $U$ be an open set of  $p$ and let $x=(x_1,\ldots,x_{2n+1})$ be local coordinates defined in $U$. 
Then, as $m\To+\infty$, 
\begin{equation}\label{e-gue170117pVIIIm}
\begin{split}
&S^{(q)}_{G,m}(x,y)=e^{im\Psi(x,y)}b(x,y,m)+O(m^{-\infty}),\\
&b(x,y,m)\in S^{n-\frac{d}{2}}_{{\rm loc\,}}(1; U\times U, T^{*0,q}X\boxtimes(T^{*0,q}X)^*),\\
&\mbox{$b(x,y,m)\sim\sum^\infty_{j=0}m^{n-\frac{d}{2}-j}b_j(x,y)$ in $S^{n-\frac{d}{2}}_{{\rm loc\,}}(1; U\times U, T^{*0,q}X\boxtimes(T^{*0,q}X)^*)$},\\
&b_j(x,y)\in C^\infty(U\times U, T^{*0,q}X\boxtimes(T^{*0,q}X)^*),\ \ j=0,1,2,\ldots,
\end{split}
\end{equation}
\begin{equation}\label{e-gue170117pVIIIam}
b_0(x,x)=2^{d-1}\frac{1}{V_{{\rm eff\,}}(x)}\pi^{-n-1+\frac{d}{2}}\abs{\det R_x}^{-\frac{1}{2}}\abs{\det\mathcal{L}_{x}}\tau_{x,n_-},\ \ \forall x\in\mu^{-1}(0),
\end{equation}
$\Psi(x,y)\in C^\infty(U\times U)$, $d_x\Psi(x,x)=-d_y\Psi(x,x)=-\omega_0(x)$, for every $x\in\mu^{-1}(0)$, $\Psi(x,y)=0$ if and only if $x=y\in\mu^{-1}(0)$ and there is a constant $C\geq1$ such that, for all $(x,y)\in U\times U$, 
\begin{equation}\label{e-gue170117pVIIIbm}
\begin{split}
&{\rm Im\,}\Psi(x,y)\geq\frac{1}{C}\Bigr(d(x,\mu^{-1}(0))^2+d(y,\mu^{-1}(0))^2+\inf_{g\in G,\theta\in S^1}d(e^{i\theta}\circ g\circ x,y)^2\Bigr),\\
&{\rm Im\,}\Psi(x,y)\leq C\Bigr(d(x,\mu^{-1}(0))^2+d(y,\mu^{-1}(0))^2+\inf_{g\in G,\theta\in S^1}d(e^{i\theta}\circ g\circ x,y)^2\Bigr).
\end{split}
\end{equation}
(We refer the reader to Theorem~\ref{t-gue170128a} for more properties of the phase $\Psi(x,y)$.)
\end{theorem} 

We refer the reader to the discussion in the beginning of Section~\ref{s-gue170111w} and Definition~\ref{d-gue140826} for the precise meanings of $A=B+O(m^{-\infty})$ and the symbol space $S^{n-\frac{d}{2}}_{{\rm loc\,}}$, respectively. 

It is well-known that when $X$ admits a transversal and CR $S^1$ action and the Levi form is non-degenerate of constant signature on $X$, then  $\Box^{(q)}_b$ has $L^2$ closed range (see Theorem 1.12 in~\cite{HM14}). 

Let $Y_G:=\mu^{-1}(0)/G$. In Theorem~\ref{t-gue170128b}, we will show that $Y_G$ is a CR manifold with natural CR structure induced by $T^{1,0}X$ of dimension $2n-2d+1$. Let $\mathcal{L}_{Y_G,x}$ be the Levi form on $Y_G$  at $x\in Y_G$ induced naturally from $\mathcal{L}$.
Note that the bilinear form $b$ is non-degenerate on $\mu^{-1}(0)$, where $b$ is given by \eqref{E:biform}. Hence, on $(\underline{\mathfrak{g}}, \underline{\mathfrak{g}})$, $b$ has constant signature on $\mu^{-1}(0)$. Assume that on $(\underline{\mathfrak{g}}, \underline{\mathfrak{g}})$, $b$ has $r$ negative eigenvalues and $d-r$ positive eigenvalues on $\mu^{-1}(0)$. Hence $\mathcal{L}_{Y_G}$ has $q-r$ negative and $n-d-q+r$ positive eigenvalues at each point of $Y_G$. Let $\Box^{(q-r)}_{b,Y_G}$ be the Kohn Laplacian for $(0,q-r)$ forms on $Y_G$. Fix $m\in\mathbb N$. Let 
\[H^{q-r}_{b,m}(Y_G):=\set{u\in\Omega^{0,q-r}(Y_G);\, \Box^{(q-r)}_{b,Y_G}u=0,\ \ Tu=imu}.\]
We will apply Theorem~\ref{t-gue170128I} to establish an isomorphism between $H^{q}_{b,m}(X)^G$ and $H^{q-r}_{b,m}(Y_G)$. We introduce some notations. 

Since $\underline{\mathfrak{g}}_x$ is orthogonal to $H_xY\bigcap JH_xY$ and $H_xY\bigcap JH_xY\subset\underline{\mathfrak{g}}^{\perp_b}_x$ (see Lemma~\ref{l-gue170120} and \eqref{e-gue170311} for the meaning of $\underline{\mathfrak{g}}^{\perp_b}_x$), for every $x\in Y$, we can find a $G$-invariant orthonormal basis $\set{Z_1,\ldots,Z_n}$ of $T^{1,0}X$ on $Y$ such that 
\[\mathcal{L}_x(Z_j(x),\ol Z_k(x))=\delta_{j,k}\lambda_j(x),\ \ j,k=1,\ldots,n,\ \ x\in Y,\]
and 
\[\begin{split}
&Z_j(x)\in\underline{\mathfrak{g}}_x+iJ\underline{\mathfrak{g}}_x,\ \ \forall x\in Y,\ \ j=1,2,\ldots,d,\\
&Z_j(x)\in\Complex H_xY\bigcap J(\Complex H_xY),\ \ \forall x\in Y,\ \ j=d+1,\ldots,n.
\end{split}\]
Let $\set{e_1,\ldots,e_n}$ denote the orthonormal basis of $T^{*0,1}X$ on $Y$, dual to $\set{\ol Z_1,\ldots,\ol Z_n}$. Fix $s=0,1,2,\ldots,n-d$. For $x\in Y$, put 
\begin{equation}\label{e-gue170303m}
B^{*0,s}_xX=\set{\sum_{d+1\leq j_1<\cdots<j_s\leq n}a_{j_1,\ldots,j_s}e_{j_1}\wedge\cdots\wedge e_{j_s};\, a_{j_1,\ldots,j_s}\in\Complex,\ \forall d+1\leq j_1<\cdots<j_s\leq n}
\end{equation}
and let $B^{*0,s}X$ be the vector bundle of $Y$ with fiber $B^{*0,s}_xX$, $x\in Y$. Let $C^\infty(Y,B^{*0,s}X)^G$ denote the set of all $G$-invariant sections of $Y$ with values in $B^{*0,s}X$. Let 
\begin{equation}\label{e-gue170303cwm}
\iota_G:C^\infty(Y,B^{*0,s}X)^G\To\Omega^{0,s}(Y_G)
\end{equation}
be the natural identification. 

Assume that $\lambda_1<0,\ldots,\lambda_r<0$, and $\lambda_{d+1}<0,\ldots,\lambda_{n_--r+d}<0$. For $x\in Y$, put 
\begin{equation}\label{e-gue170303Im}
\hat{\mathcal{N}}(x,n_-)=\set{ce_{d+1}\wedge\cdots\wedge e_{n_--r+d};\, c\in\Complex},
\end{equation}
and let 
\begin{equation}\label{e-gue170303cwIm}
\begin{split}
&\hat p=\hat p_x:\mathcal{N}(x,n_-)\To \hat{\mathcal{N}}(x,n_-),\\
&u=ce_1\wedge\cdots\wedge e_r\wedge e_{d+1}\wedge\cdots\wedge e_{n_--r+d}\To ce_{d+1}\wedge\cdots\wedge e_{n_--r+d}.
\end{split}
\end{equation}
Let $\iota:Y\To X$ be the natural inclusion and let $\iota^*:\Omega^{0,q}(X)\To\Omega^{0,q}(Y)$ be the pull-back of $\iota$. Let $q=n_-$. 
Let $S^{(q-r)}_{Y_G,m}:L^2_{(0,q-r)}(Y_G)\To H^{q-r}_{b,m}(Y_G)$ be the orthogonal projection and let 
\begin{equation}\label{e-gue170303cwam}
f(x)=\sqrt{V_{{\rm eff\,}}(x)}\abs{\det\,R_x}^{-\frac{1}{4}}\in C^\infty(Y)^G.
\end{equation}
Let
\begin{equation}\label{e-gue170303cwIIz}
\begin{split}
\sigma_m:H^q_{b,m}(X)^G&\To H^{q-r}_{b,m}(Y_G),\\
u&\To m^{-\frac{d}{4}}S^{(q-r)}_{Y_G,m}\circ\iota_G\circ \hat p\circ\tau_{x,n_-}\circ f\circ \iota^*\circ u.
\end{split}
\end{equation}

In Section~\ref{s-gue170303}, we will show that 

\begin{theorem}\label{t-gue170122}
With the notations and assumptions above, suppose that $q=n_-$. If $m$ is large, then 
\[\sigma_m:H^q_{b,m}(X)^G\To H^{q-r}_{b,m}(Y_G)\]
is an isomorphism. In particular, if $m$ large enough, then
\begin{equation}\label{e-gue170122r}
{\rm dim\,}H^q_{b,m}(X)^G={\rm dim\,}H^{q-r}_{b,m}(Y_G). 
\end{equation}
\end{theorem}

\begin{remark}\label{r-gue170216}
Let's sketch the idea of the proof of Theorem~\ref{t-gue170122}. W can consider $\sigma_m$ as a map from $\Omega^{0,q}(X)\To H^{q-r}_{b,m}(Y_G)$: 
\begin{equation}\label{e-gue170303cwIIy}
\begin{split}
\sigma_m:\Omega^{0,q}(X)&\To H^{q-r}_{b,m}(Y_G)\subset\Omega^{0,q-r}(Y_G),\\
u&\To m^{-\frac{d}{4}}S^{(q-r)}_{Y_G,m}\circ\iota_G\circ \hat p\circ\tau_{x,n_-}\circ f\circ \iota^*\circ S^{(q)}_{G,m}u.
\end{split}
\end{equation}
Let $\sigma^*_m:\Omega^{0,q-r}(Y_G)\To\Omega^{0,q}(X)$ be the adjoint of $\sigma_m$. From Theorem~\ref{t-gue170128I} and some calculation of complex Fourier integral operators, we will show in Section~\ref{s-gue170303} that $F_m=\sigma^*_m\sigma_m:\Omega^{0,q}(X)\To\Omega^{0,q}(X)$ is the same type of operator as $S^{(q)}_{G,m}$ and 
\begin{equation}\label{e-gue170306ng}
\frac{1}{C_0}F_m=(I+R_m)S^{(q)}_{G,m},
\end{equation}
 where $C_0>0$ is a constant and $R_m$ is also the same type of operator as $S^{(q)}_{G,m}$, but the leading symbol of $R_m$ vanishes at ${\rm diag\,}(Y\times Y)$. By using the fact that the leading symbol of $R_m$ vanishes at ${\rm diag\,}(Y\times Y)$, we will show in Lemma~\ref{l-gue170306} that \[
\norm{R_mu}\leq \varepsilon_m\norm{u},\ \ \forall u\in\Omega^{0,q}(X),\ \ \forall m\in\mathbb N,\]
where $\varepsilon_m$ is a sequence with $\lim_{m\To\infty}\varepsilon_m=0$. In particular, if $m$ is large enough, then the map 
\begin{equation}\label{e-gue170306ngI}
I+R_m:\Omega^{0,q}(X)\To \Omega^{0,q}(X)
\end{equation}
is injective. From \eqref{e-gue170306ng} and \eqref{e-gue170306ngI}, we deduce that, if $m$ is large enough, then
\[F_m:H^q_{b,m}(X)^G\To H^q_{b,m}(X)^G\]
is injective. Hence $\sigma_m:H^q_{b,m}(X)^G\To H^{q-r}_{b,m}(Y_G)$ is injective. 

Similarly, we can repeat the argument above with minor change and deduce that if $m$ is large enough, then the map
\[\hat F_m=\sigma_m\sigma^*_m: H^{q-r}_{b,m}(Y_G)\To H^{q-r}_{b,m}(Y_G)\]
is injective. Hence, if $m$ is large enough, then the map
\[\sigma^*_m:H^{q-r}_{b,m}(Y_G)\To H^{q}_{b,m}(X)^G\]
is injective. Thus,  ${\rm dim\,}H^q_{b,m}(X)^G={\rm dim\,}H^{q-r}_{b,m}(Y_G)$ and $\sigma_m$ is an isomorphism if $m$ large enough. 
\end{remark}

Let's apply Theorem~\ref{t-gue170122} to complex case. 
Let $(L,h^L)$ be a holomorphic line bundle over a connected compact complex manifold $(M,J)$ with ${\rm dim\,}_{\Complex}M=n$, where $J$ denotes the complex structure map of $M$ and $h^L$ is a Hermitian fiber metric of $L$. Let $R^L$ be the curvature of $L$ induced by $h^L$. Assume that $R^L$ is non-degenerate of constant signature $(n_-,n_+)$ on $M$. Let $K$ be a connected compact Lie group with Lie algebra $\mathfrak{k}$. We assume that ${\rm dim\,}_{\Real}K=d$ and $K$ acts holomorphically on $(M,J)$, and that the action lifts to a holomorphic action on $L$. We assume further that $h^L$ is preserved by the $K$-action. Then $R^L$ is a $K$-invariant form. Let $\omega=\frac{i}{2\pi}R^L$ and let $\tilde{\mu}: M \to \mathfrak{k}^*$ be the moment map induced by $\omega$.  Assume that $0 \in \mathfrak{k}^*$ is regular and the action of $K$ on $\tilde{\mu}^{-1}(0)$ is globally free.  The  analogue of the Marsden-Weinstein reduction holds (see~\cite{GS}). More precisely, the complex structure $J$ on $M$ induces a complex structure $J_K$ on $M_0:=\tilde{\mu}^{-1}(0)/K$, for which the line bundle $L_0:=L/K$ is a holomorphic line bundle over $M_0$. 

For any $\xi \in \mathfrak{k}$, we write $\xi_M$ to denote the vector field on $M$ induced by $\xi$. Let $\underline{\mathfrak{k}}={\rm Span\,}(\xi_M;\, \xi\in\mathfrak{k})$. On $\tilde{\mu}^{-1}(0)$, let $b^L$ be the bilinear form on $\underline{\mathfrak{k}}\times\underline{\mathfrak{k}}$ given by $b^L(\,\cdot\,,\,\cdot\,)=\omega(\,\cdot\,,\,J\cdot\,)$.  Assume that $b^L$ has $r$ negative eigenvalues and $d-r$ positive eigenvalues on $\tilde{\mu}^{-1}(0)$. Let $q=n_-$. For $m\in\mathbb N$,  let $H^q(M,L^m)^K$ denote the $K$-invariant $q$-th Dolbeault cohomology group with values in $L^m$ and let $H^{q-r}(M_0,L^m_0)$ denote the $(q-r)$-th Dolbeault cohomology group with values in $L^m_0$. Theorem~\ref{t-gue170122} implies that, if $m$ is large enough, then there is an isomorphism map:
\[\Td\sigma_m:H^q(M,L^m)^K\To H^{q-r}(M_0,L^m_0).\]
In particular, if $m$ is large enough,then 
\begin{equation}\label{e-gue170216}
\dim H^q(M,L^m)^K = \dim H^{q-r}(M_0,L_0^m).
\end{equation}
Note that when $m=1$ and $q=0$, the equality \eqref{e-gue170216}
was first proved in \cite[\S 5]{GS}. For $m=1$, the equality \eqref{e-gue170216} was established in \cite{Br, T, Z} when $L$ is positive. 
The proof of the equality \eqref{e-gue170216} in \cite{T} is completely algebraic, while the the proofs of the equality \eqref{e-gue170216} in \cite{Br, Z} are purely analytic where different quasi-homomorphisms between Dolbeault complexes under considerations were constructed to prove the equality \eqref{e-gue170216}. If $m$ is large enough and $q=0$, an isomorphism map in \eqref{e-gue170216} was also constructed in \cite[(0.27), Corollary 4.13]{MZ}.

\subsection{The phase functions $\Phi_-(x,y)$ and $\Psi(x,y)$}\label{s-gue170126}

In this section, we collect some properties of the phase functions $\Phi_-(x,y)$, $\Psi(x,y)$ in Theorem~\ref{t-gue170124} and Theorem~\ref{t-gue170128I}. 

Let $v=(v_1,\ldots,v_d)$ be local coordinates of $G$ defined  in a neighborhood $V$ of $e_0$ with $v(e_0)=(0,\ldots,0)$. From now on, we will identify the element $e\in V$ with $v(e)$.  Fix $p\in\mu^{-1}(0)$. In Theorem~\ref{t-gue161202}, we will show that 
there exist local coordinates $v=(v_1,\ldots,v_d)$ of $G$ defined in  a neighborhood $V$ of $e_0$ with $v(e_0)=(0,\ldots,0)$, local coordinates $x=(x_1,\ldots,x_{2n+1})$ of $X$ defined in a neighborhood $U=U_1\times U_2$ of $p$ with $0\leftrightarrow p$, where $U_1\subset\Real^d$ is an open set of $0\in\Real^d$,  $U_2\subset\Real^{2n+1-d}$ is an open set of $0\in\Real^{2n+1-d} $ and a smooth function $\gamma=(\gamma_1,\ldots,\gamma_d)\in C^\infty(U_2,U_1)$ with $\gamma(0)=0\in\Real^d$  such that
\begin{equation}\label{e-gue161202m}
\begin{split}
&(v_1,\ldots,v_d)\circ (\gamma(x_{d+1},\ldots,x_{2n+1}),x_{d+1},\ldots,x_{2n+1})\\
&=(v_1+\gamma_1(x_{d+1},\ldots,x_{2n+1}),\ldots,v_d+\gamma_d(x_{d+1},\ldots,x_{2n+1}),x_{d+1},\ldots,x_{2n+1}),\\
&\forall (v_1,\ldots,v_d)\in V,\ \ \forall (x_{d+1},\ldots,x_{2n+1})\in U_2,
\end{split}
\end{equation}
\begin{equation}\label{e-gue161206m}
\begin{split}
&\underline{\mathfrak{g}}={\rm span\,}\set{\frac{\pr}{\pr x_1},\ldots,\frac{\pr}{\pr x_d}},\\
&\mu^{-1}(0)\bigcap U=\set{x_{d+1}=\cdots=x_{2d}=0},\\
&\mbox{On $\mu^{-1}(0)\bigcap U$, we have $J(\frac{\pr}{\pr x_j})=\frac{\pr}{\pr x_{d+j}}+a_j(x)\frac{\pr}{\pr x_{2n+1}}$, $j=1,2,\ldots,d$}, 
\end{split}
\end{equation}
where $a_j(x)$ is a smooth function on $\mu^{-1}(0)\bigcap U$, independent of $x_1,\ldots,x_{2d}$, $x_{2n+1}$ and $a_j(0)=0$, $j=1,\ldots,d$, 
\begin{equation}\label{e-gue161202Im}
\begin{split}
&T^{1,0}_pX={\rm span\,}\set{Z_1,\ldots,Z_n},\\
&Z_j=\frac{1}{2}(\frac{\pr}{\pr x_j}-i\frac{\pr}{\pr x_{d+j}})(p),\ \ j=1,\ldots,d,\\
&Z_j=\frac{1}{2}(\frac{\pr}{\pr x_{2j-1}}-i\frac{\pr}{\pr x_{2j}})(p),\ \ j=d+1,\ldots,n,\\
&\langle\,Z_j\,|\,Z_k\,\rangle=\delta_{j,k},\ \ j,k=1,2,\ldots,n,\\
&\mathcal{L}_p(Z_j, \ol Z_k)=\mu_j\delta_{j,k},\ \ j,k=1,2,\ldots,n
\end{split}
\end{equation}
and 
\begin{equation}\label{e-gue161219m}
\begin{split}
\omega_0(x)&=(1+O(\abs{x}))dx_{2n+1}+\sum^d_{j=1}4\mu_jx_{d+j}dx_j\\
&\quad+\sum^n_{j=d+1}2\mu_jx_{2j}dx_{2j-1}-\sum^n_{j=d+1}2\mu_jx_{2j-1}dx_{2j}+\sum^{2n}_{j=d+1}b_jx_{2n+1}dx_j+O(\abs{x}^2),
\end{split}
\end{equation}
where $b_{d+1}\in\Real,\ldots,b_{2n}\in\Real$.  Put $x''=(x_{d+1},\ldots,x_{2n+1})$, $\hat x''=(x_{d+1}, x_{d+2},\ldots,x_{2d})$, $\mathring{x}''=(x_{d+1},\ldots,x_{2n})$. We have the following (see Theorem~\ref{t-gue170105I} and Theorem~\ref{t-gue170107})

\begin{theorem}\label{t-gue170126}
With the notations above, the phase function $\Phi_-(x,y)\in C^\infty(U\times U)$ is independent of $(x_1,\ldots,x_d)$ and $(y_1,\ldots,y_d)$. Hence, $\Phi_-(x,y)=\Phi_-((0,x''),(0,y'')):=\Phi_-(x'',y'')$. Moreover, there is a constant $c>0$ such that 
\begin{equation}\label{e-gue170126}
{\rm Im\,}\Phi_-(x'',y'')\geq c\Bigr(\abs{\hat x''}^2+\abs{\hat y''}^2+\abs{\mathring{x}''-\mathring{y}''}^2\Bigr),\ \ \forall ((0,x''),(0,y''))\in U\times U.
\end{equation}

Furthermore, 
\begin{equation}\label{e-gue170126I}
\begin{split}
\Phi_-(x'', y'')&=-x_{2n+1}+y_{2n+1}+2i\sum^d_{j=1}\abs{\mu_j}y^2_{d+j}+2i\sum^d_{j=1}\abs{\mu_j}x^2_{d+j}\\
&+i\sum^{n}_{j=d+1}\abs{\mu_j}\abs{z_j-w_j}^2 +\sum^{n}_{j=d+1}i\mu_j(\ol z_jw_j-z_j\ol w_j)\\
&+\sum^d_{j=1}(-b_{d+j}x_{d+j}x_{2n+1}+b_{d+j}y_{d+j}y_{2n+1})\\
&+\sum^n_{j=d+1}\frac{1}{2}(b_{2j-1}-ib_{2j})(-z_jx_{2n+1}+w_jy_{2n+1})\\
&+\sum^n_{j=d+1}\frac{1}{2}(b_{2j-1}+ib_{2j})(-\ol z_jx_{2n+1}+\ol w_jy_{2n+1})\\
&+(x_{2n+1}-y_{2n+1})f(x, y) +O(\abs{(x, y)}^3),
\end{split}
\end{equation}
where $z_j=x_{2j-1}+ix_{2j}$, $w_j=y_{2j-1}+iy_{2j}$, $j=d+1,\ldots,n$, $\mu_j$, $j=1,\ldots,n$, and $b_{d+1}\in\Real,\ldots,b_{2n}\in\Real$ are as in \eqref{e-gue161219m} and $f$ is smooth and satisfies $f(0,0)=0$, $f(x, y)=\ol f(y, x)$. 
\end{theorem}

We now assume that $X$ admits an $S^1$ action: $S^1\times X\rightarrow X$. We will use the same notations as in Theorem~\ref{t-gue170128I}. Recall that we work with Assumption~\ref{a-gue170128}. Let $p\in\mu^{-1}(0)$. We can repeat the proof of Theorem~\ref{t-gue161202} with minor change and show that there exist local coordinates $v=(v_1,\ldots,v_d)$ of $G$ defined in  a neighborhood $V$ of $e_0$ with $v(e_0)=(0,\ldots,0)$, local coordinates $x=(x_1,\ldots,x_{2n+1})$ of $X$ defined in a neighborhood $U=U_1\times(\hat U_2\times]-2\delta,2\delta[)$ of $p$ with $0\leftrightarrow p$, where $U_1\subset\Real^d$ is an open set of $0\in\Real^d$,  $\hat U_2\subset\Real^{2n-d}$ is an open set of $0\in\Real^{2n-d} $, $\delta>0$, and a smooth function $\gamma=(\gamma_1,\ldots,\gamma_d)\in C^\infty(\hat U_2\times ]-2\delta,2\delta[,U_1)$ with $\gamma(0)=0\in\Real^d$  such that $T=-\frac{\pr}{\pr x_{2n+1}}$ and \eqref{e-gue161202m}, \eqref{e-gue161206m}, \eqref{e-gue161202Im}, \eqref{e-gue161219m} hold. We have the following (see the proof of Theorem~\ref{t-gue170118})

\begin{theorem}\label{t-gue170128a}
With the notations above, the phase function $\Psi$ satisfies $\Psi(x,y)=-x_{2n+1}+y_{2n+1}+\hat\Psi(\mathring{x}'',\mathring{y}'')$, where $\hat\Psi(\mathring{x}'',\mathring{y}'')\in C^\infty(U\times U)$, $\mathring{x}''=(x_{d+1},\ldots,x_{2n})$,  $\mathring{y}''=(y_{d+1},\ldots,y_{2n})$, and $\Psi$ satisfies \eqref{e-gue170126} and \eqref{e-gue170126I}. 
\end{theorem}

\section{Preliminaries}\label{s:prelim}
\subsection{Standard notations} \label{s-ssna}
We shall use the following notations: $\mathbb N=\set{1,2,\ldots}$, $\mathbb N_0=\mathbb N\cup\set{0}$, $\Real$ 
is the set of real numbers, $\ol\Real_+:=\set{x\in\Real;\, x\geq0}$. 
For a multi-index $\alpha=(\alpha_1,\ldots,\alpha_n)\in\mathbb N_0^n$,
we denote by $\abs{\alpha}=\alpha_1+\ldots+\alpha_n$ its norm and by $l(\alpha)=n$ its length.
For $m\in\mathbb N$, write $\alpha\in\set{1,\ldots,m}^n$ if $\alpha_j\in\set{1,\ldots,m}$, 
$j=1,\ldots,n$. $\alpha$ is strictly increasing if $\alpha_1<\alpha_2<\ldots<\alpha_n$. For $x=(x_1,\ldots,x_n)$, we write
\[
\begin{split}
&x^\alpha=x_1^{\alpha_1}\ldots x^{\alpha_n}_n,\\
& \pr_{x_j}=\frac{\pr}{\pr x_j}\,,\quad
\pr^\alpha_x=\pr^{\alpha_1}_{x_1}\ldots\pr^{\alpha_n}_{x_n}=\frac{\pr^{\abs{\alpha}}}{\pr x^\alpha}\,,\\
&D_{x_j}=\frac{1}{i}\pr_{x_j}\,,\quad D^\alpha_x=D^{\alpha_1}_{x_1}\ldots D^{\alpha_n}_{x_n}\,,
\quad D_x=\frac{1}{i}\pr_x\,.
\end{split}
\]
Let $z=(z_1,\ldots,z_n)$, $z_j=x_{2j-1}+ix_{2j}$, $j=1,\ldots,n$, be coordinates of $\Complex^n$.
We write
\[
\begin{split}
&z^\alpha=z_1^{\alpha_1}\ldots z^{\alpha_n}_n\,,\quad\ol z^\alpha=\ol z_1^{\alpha_1}\ldots\ol z^{\alpha_n}_n\,,\\
&\pr_{z_j}=\frac{\pr}{\pr z_j}=
\frac{1}{2}\Big(\frac{\pr}{\pr x_{2j-1}}-i\frac{\pr}{\pr x_{2j}}\Big)\,,\quad\pr_{\ol z_j}=
\frac{\pr}{\pr\ol z_j}=\frac{1}{2}\Big(\frac{\pr}{\pr x_{2j-1}}+i\frac{\pr}{\pr x_{2j}}\Big),\\
&\pr^\alpha_z=\pr^{\alpha_1}_{z_1}\ldots\pr^{\alpha_n}_{z_n}=\frac{\pr^{\abs{\alpha}}}{\pr z^\alpha}\,,\quad
\pr^\alpha_{\ol z}=\pr^{\alpha_1}_{\ol z_1}\ldots\pr^{\alpha_n}_{\ol z_n}=
\frac{\pr^{\abs{\alpha}}}{\pr\ol z^\alpha}\,.
\end{split}
\]
For $j, s\in\mathbb Z$, set $\delta_{j,s}=1$ if $j=s$, $\delta_{j,s}=0$ if $j\neq s$.

Let $M$ be a $C^\infty$ paracompact manifold.
We let $TM$ and $T^*M$ denote the tangent bundle of $M$
and the cotangent bundle of $M$, respectively.
The complexified tangent bundle of $M$ and the complexified cotangent bundle of $M$ will be denoted by $\Complex TM$
and $\Complex T^*M$, respectively. Write $\langle\,\cdot\,,\cdot\,\rangle$ to denote the pointwise
duality between $TM$ and $T^*M$.
We extend $\langle\,\cdot\,,\cdot\,\rangle$ bilinearly to $\Complex TM\times\Complex T^*M$.
Let $G$ be a $C^\infty$ vector bundle over $M$. The fiber of $G$ at $x\in M$ will be denoted by $G_x$.
Let $E$ be a vector bundle over a $C^\infty$ paracompact manifold $M_1$. We write
$G\boxtimes E^*$ to denote the vector bundle over $M\times M_1$ with fiber over $(x, y)\in M\times M_1$
consisting of the linear maps from $E_y$ to $G_x$.  Let $Y\subset M$ be an open set. 
From now on, the spaces of distribution sections of $G$ over $Y$ and
smooth sections of $G$ over $Y$ will be denoted by $D'(Y, G)$ and $C^\infty(Y, G)$, respectively.
Let $E'(Y, G)$ be the subspace of $D'(Y, G)$ whose elements have compact support in $Y$.
For $m\in\Real$, let $H^m(Y, G)$ denote the Sobolev space
of order $m$ of sections of $G$ over $Y$. Put
\begin{gather*}
H^m_{\rm loc\,}(Y, G)=\big\{u\in D'(Y, G);\, \varphi u\in H^m(Y, G),
      \, \forall\varphi\in C^\infty_0(Y)\big\}\,,\\
      H^m_{\rm comp\,}(Y, G)=H^m_{\rm loc}(Y, G)\cap E'(Y, G)\,.
\end{gather*}
We recall the Schwartz kernel theorem \cite[Theorems\,5.2.1, 5.2.6]{Hor03}, \cite[Thorem\,B.2.7]{MM}.
Let $G$ and $E$ be $C^\infty$ vector
bundles over paracompact orientable $C^\infty$ manifolds $M$ and $M_1$, respectively, equipped with smooth densities of integration. If
$A: C^\infty_0(M_1,E)\To D'(M,G)$
is continuous, we write $K_A(x, y)$ or $A(x, y)$ to denote the distribution kernel of $A$.
The following two statements are equivalent
\begin{enumerate}
\item $A$ is continuous: $E'(M_1,E)\To C^\infty(M,G)$,
\item $K_A\in C^\infty(M\times M_1,G\boxtimes E^*)$.
\end{enumerate}
If $A$ satisfies (1) or (2), we say that $A$ is smoothing on $M \times M_1$. Let
$A,B: C^\infty_0(M_1,E)\To D'(M,G)$ be continuous operators.
We write 
\begin{equation} \label{e-gue160507f}
\mbox{$A\equiv B$ (on $M\times M_1$)} 
\end{equation}
if $A-B$ is a smoothing operator. If $M=M_1$, we simply write "on $M$". 

We say that $A$ is properly supported if the restrictions of the two projections 
$(x,y)\mapsto x$, $(x,y)\mapsto y$ to ${\rm Supp\,}K_A$
are proper.

Let $H(x,y)\in D'(M\times M_1,G\boxtimes E^*)$. We write $H$ to denote the unique 
continuous operator $C^\infty_0(M_1,E)\To D'(M,G)$ with distribution kernel $H(x,y)$. 
In this work, we identify $H$ with $H(x,y)$.

\subsection{Some standard notations in semi-classical analysis}\label{s-gue170111w}

Let $W_1$ be an open set in $\Real^{N_1}$ and let $W_2$ be an open set in $\Real^{N_2}$. Let $E$ and $F$ be vector bundles over $W_1$ and $W_2$, respectively. 
An $m$-dependent continuous operator
$A_m:C^\infty_0(W_2,F)\To D'(W_1,E)$ is called $m$-negligible on $W$
if, for $m$ large enough, $A_m$ is smoothing and, for any $K\Subset W_1\times W_2$, any
multi-indices $\alpha$, $\beta$ and any $N\in\mathbb N$, there exists $C_{K,\alpha,\beta,N}>0$
such that
\begin{equation}\label{e-gue13628III}
\abs{\pr^\alpha_x\pr^\beta_yA_m(x, y)}\leq C_{K,\alpha,\beta,N}m^{-N}\:\: \text{on $K$},\ \ \forall m\gg1.
\end{equation}
In that case we write
\[A_m(x,y)=O(m^{-\infty})\:\:\text{on $W_1\times W_2$,}\]
or
\[A_m=O(m^{-\infty})\:\:\text{on $W_1\times W_2$.}\]
If $A_m, B_m:C^\infty_0(W_2, F)\To D'(W_1, E)$ are $m$-dependent continuous operators,
we write $A_m= B_m+O(m^{-\infty})$ on $W_1\times W_2$ or $A_m(x,y)=B_m(x,y)+O(m^{-\infty})$ on $W_1\times W_2$ if $A_m-B_m=O(m^{-\infty})$ on $W_1\times W_2$. 
When $W=W_1=W_2$, we sometime write "on $W$".

Let $X$ and $M$ be smooth manifolds and let $E$ and $F$ be vector bundles over $X$ and $M$, respectively. Let $A_m, B_m:C^\infty(M,F)\To C^\infty(X,E)$ be $m$-dependent smoothing operators. We write $A_m=B_m+O(m^{-\infty})$ on $X\times M$ if on every local coordinate patch $D$ of $X$ and local coordinate patch $D_1$ of $M$, $A_m=B_m+O(m^{-\infty})$ on $D\times D_1$.
When $X=M$, we sometime write on $X$.

We recall the definition of the semi-classical symbol spaces

\begin{definition} \label{d-gue140826}
Let $W$ be an open set in $\Real^N$. Let
\begin{gather*}
S(1;W):=\Big\{a\in C^\infty(W)\,|\, \forall\alpha\in\mathbb N^N_0:
\sup_{x\in W}\abs{\pr^\alpha a(x)}<\infty\Big\},\\
S^0_{{\rm loc\,}}(1;W):=\Big\{(a(\cdot,m))_{m\in\Real}\,|\, \forall\alpha\in\mathbb N^N_0,
\forall \chi\in C^\infty_0(W)\,:\:\sup_{m\in\Real, m\geq1}\sup_{x\in W}\abs{\pr^\alpha(\chi a(x,m))}<\infty\Big\}\,.
\end{gather*}
For $k\in\Real$, let
\[
S^k_{{\rm loc}}(1):=S^k_{{\rm loc}}(1;W)=\Big\{(a(\cdot,m))_{m\in\Real}\,|\,(m^{-k}a(\cdot,m))\in S^0_{{\rm loc\,}}(1;W)\Big\}\,.
\]
Hence $a(\cdot,m)\in S^k_{{\rm loc}}(1;W)$ if for every $\alpha\in\mathbb N^N_0$ and $\chi\in C^\infty_0(W)$, there
exists $C_\alpha>0$ independent of $m$, such that $\abs{\pr^\alpha (\chi a(\cdot,m))}\leq C_\alpha m^{k}$ holds on $W$.

Consider a sequence $a_j\in S^{k_j}_{{\rm loc\,}}(1)$, $j\in\N_0$, where $k_j\searrow-\infty$,
and let $a\in S^{k_0}_{{\rm loc\,}}(1)$. We say
\[
a(\cdot,m)\sim
\sum\limits^\infty_{j=0}a_j(\cdot,m)\:\:\text{in $S^{k_0}_{{\rm loc\,}}(1)$},
\]
if, for every
$\ell\in\N_0$, we have $a-\sum^{\ell}_{j=0}a_j\in S^{k_{\ell+1}}_{{\rm loc\,}}(1)$ .
For a given sequence $a_j$ as above, we can always find such an asymptotic sum
$a$, which is unique up to an element in
$S^{-\infty}_{{\rm loc\,}}(1)=S^{-\infty}_{{\rm loc\,}}(1;W):=\cap _kS^k_{{\rm loc\,}}(1)$.

Similarly, we can define $S^k_{{\rm loc\,}}(1;Y,E)$ in the standard way, where $Y$ is a smooth manifold and $E$ is a vector bundle over $Y$. 
\end{definition}

\subsection{CR manifolds and bundles} 

Let $(X, T^{1,0}X)$ be a compact, connected and orientable CR manifold of dimension $2n+1$, $n\geq 1$, where $T^{1,0}X$ is a CR structure of $X$, that is, $T^{1,0}X$ is a subbundle of rank $n$ of the complexified tangent bundle $\mathbb{C}TX$, satisfying $T^{1,0}X\cap T^{0,1}X=\{0\}$, where $T^{0,1}X=\overline{T^{1,0}X}$, and $[\mathcal V,\mathcal V]\subset\mathcal V$, where $\mathcal V=C^\infty(X, T^{1,0}X)$. There is a unique subbundle $HX$ of $TX$ such that $\mathbb{C}HX=T^{1,0}X \oplus T^{0,1}X$, i.e. $HX$ is the real part of $T^{1,0}X \oplus T^{0,1}X$. Let $J:HX\To HX$ be the complex structure map given by $J(u+\ol u)=iu-i\ol u$, for every $u\in T^{1,0}X$. 
By complex linear extension of $J$ to $\mathbb{C}TX$, the $i$-eigenspace of $J$ is $T^{1,0}X \, = \, \left\{ V \in \mathbb{C}HX \, : \, JV \, =  \,  \sqrt{-1}V  \right\}.$ We shall also write $(X, HX, J)$ to denote a compact CR manifold.

We  fix a real non-vanishing $1$ form $\omega_0\in C(X,T^*X)$ so that $\langle\,\omega_0(x)\,,\,u\,\rangle=0$, for every $u\in H_xX$, for every $x\in X$. 
For each $x \in X$, we define a quadratic form on $HX$ by
\begin{equation}\label{E:levi}
\mathcal{L}_x(U,V) =\frac{1}{2}d\omega_0(JU, V), \forall \ U, V \in H_xX.
\end{equation}
We extend $\mathcal{L}$ to $\mathbb{C}HX$ by complex linear extension. Then, for $U, V \in T^{1,0}_xX$,
\begin{equation}
\mathcal{L}_x(U,\overline{V}) = \frac{1}{2}d\omega_0(JU, \overline{V}) = -\frac{1}{2i}d\omega_0(U,\overline{V}).
\end{equation}
The Hermitian quadratic form $\mathcal{L}_x$ on $T^{1,0}_xX$ is called Levi form at $x$. We recall that in this paper, we always assume that the Levi form $\mathcal{L}$ on $T^{1,0}X$ is non-degenerate of constant signature $(n_-,n_+)$ on $X$, where $n_-$ denotes the number of negative eigenvalues of the Levi form and $n_+$ denotes the number of positive eigenvalues of the Levi form. Let $T\in C^\infty(X,TX)$ be the non-vanishing vector field determined by 
\begin{equation}\label{e-gue170111ry}\begin{split}
&\omega_0(T)=-1,\\
&d\omega_0(T,\cdot)\equiv0\ \ \mbox{on $TX$}.
\end{split}\end{equation}
Note that $X$ is a contact manifold with contact form $\omega_0$, contact plane $HX$ and $T$ is the Reeb vector field.

Fix a smooth Hermitian metric $\langle \cdot \mid \cdot \rangle$ on $\mathbb{C}TX$ so that $T^{1,0}X$ is orthogonal to $T^{0,1}X$, $\langle u \mid v \rangle$ is real if $u, v$ are real tangent vectors, $\langle\,T\,|\,T\,\rangle=1$ and $T$ is orthogonal to $T^{1,0}X\oplus T^{0,1}X$. For $u \in \mathbb{C}TX$, we write $|u|^2 := \langle u | u \rangle$. Denote by $T^{*1,0}X$ and $T^{*0,1}X$ the dual bundles $T^{1,0}X$ and $T^{0,1}X$, respectively. They can be identified with subbundles of the complexified cotangent bundle $\mathbb{C}T^*X$. Define the vector bundle of $(0,q)$-forms by $T^{*0,q}X := \wedge^qT^{*0,1}X$. The Hermitian metric $\langle \cdot | \cdot \rangle$ on $\mathbb{C}TX$ induces, by duality, a Hermitian metric on $\mathbb{C}T^*X$ and also on the bundles of $(0,q)$ forms $T^{*0,q}X, q=0, 1, \cdots, n$. We shall also denote all these induced metrics by $\langle \cdot | \cdot \rangle$. Note that we have the pointwise orthogonal decompositions:
\begin{equation}
\begin{array}{c}
\mathbb{C}T^*X = T^{*1,0}X \oplus T^{*0,1}X \oplus \left\{ \lambda \omega_0: \lambda \in \mathbb{C} \right\}, \\
\mathbb{C}TX = T^{1,0}X \oplus T^{0,1}X \oplus \left\{ \lambda T: \lambda \in \mathbb{C} \right\}.
\end{array}
\end{equation}

For $x, y\in X$, let $d(x,y)$ denote the distance between $x$ and $y$ induced by the Hermitian metric $\langle \cdot \mid \cdot \rangle$. Let $A$ be a subset of $X$. For every $x\in X$, let $d(x,A):=\inf\set{d(x,y);\, y\in A}$.

Let $\Omega^{0,q}(D)$ denote the space of smooth sections of $T^{*0,q}X$ over $D$ and let $\Omega^{0,q}_0(D)$ be the subspace of $\Omega^{0,q}(D)$ whose elements have compact support in $D$.




\subsection{Contact reduction}
Let $G$ be a connected compact Lie group with Lie algebra $\mathfrak{g}$ such that $\dim_{\mathbb{R}}G = d$. We assume that the Lie group $G$ acts on $X$ preserving $\omega_0$, i.e. $g^*\omega_0 =\omega_0$, for any $g \in G$. For any $\xi \in \mathfrak{g}$, there is an induced vector field $\xi_X$ on $X$ given by $(\xi_X u)(x)=\frac{\partial}{\partial t}\left(u( \exp(t\xi)\circ x)\right)|_{t=0}$, for any $u\in C^\infty(X)$.

\begin{definition}
The contact moment map associated to the form $\omega_0$ is the map $\mu:X \to \mathfrak{g}^*$ such that, for all $x \in X$ and $\xi \in \mathfrak{g}$, we have 
\begin{equation}\label{E:cmp}
\langle \mu(x), \xi \rangle = \omega_0(\xi_X(x)).
\end{equation}
\end{definition}

We now recall the contact reduction from \cite{A,Ge}. It was shown in \cite{A, Ge} that the contact moment map is $G$-equivariant, so $G$ acts on $Y:=\mu^{-1}(0)$, where $G$ acts on $\mathfrak{g}^*$ through co-adjoint represent. Since we assume that the action of $G$ on $Y$ is globally free, $Y_G:=\mu^{-1}(0)/G$ is a smooth manifold. Let $\pi: Y \to Y_G$ and $\iota:Y \hookrightarrow X$ be the natural quotient and inclusion, respectively, then there is a unique induced contact form $\widetilde{\omega}_0$ on $Y_G$ such $\pi^*\widetilde{\omega}_0 = \iota^* \omega_0$. We denote by
\begin{equation}
HY:= \operatorname{Ker} \omega_0 \cap T(\mu^{-1}(0)) = HX \cap TY,
\end{equation}
then the induced contact plane on $Y_G$ is
\begin{equation}
HY_G := \pi_\ast HY.
\end{equation}
In particular, $\dim HY = 2n-d$ and $\dim HY_G=2n-2d.$

\subsection{CR reduction}\label{s-gue170301}
In this subsection we study the reduction of CR manifolds with non-degenerate Levi curvature which is a CR analogue of the reduction on complex manifolds considered in \cite[\S 2.1]{P}. For the case of strictly pseudoconvex CR manifolds, the CR reduction was also studied in \cite{L00}. 

Recall that we work with Assumption~\ref{a-gue170123I}. 
Let $b$ be the nondegenerate bilinear form on $HX$ such that 
\begin{equation}\label{E:biform}
b(\cdot , \cdot) = d\omega_0(\cdot , J\cdot).
\end{equation}

We denote by $\underline{\mathfrak{g}} := \operatorname{Span} (\xi_X, \xi \in \mathfrak{g})$ the tangent bundle of the orbits in $X$. Let 
\begin{equation}\label{e-gue170311}
\underline{\mathfrak{g}}^{\perp_b}=\set{v\in HX;\, b(\xi_X,v)=0,\ \ \forall \xi_X\in\underline{\mathfrak{g}}}.
\end{equation}
Since the Levi form is non-degenerate, when restricted to $\underline{\mathfrak{g}} \times  \underline{\mathfrak{g}}$, the bilinear form $b$ is nondegenerate on $Y$. In the following lemma, we don't assume that  the Levi form is non-degenerate

\begin{lemma}\label{L:ndgbar}
When restricted to $ \underline{\mathfrak{g}} \times  \underline{\mathfrak{g}}$, the bilinear form $b$ is nondegenerate on $Y$.
\end{lemma}
\begin{proof}
For $x \in Y, V \in H_xX$ and $\xi \in \mathfrak{g}$, by \eqref{E:cmp} and \eqref{E:biform}, we have
\begin{equation}
b_x(\xi_X,JV) =-d\omega_0(x)(\xi_X, V) = -\left( d\mu(x)(V) \right)(\xi).
\end{equation}
Therefore,
\begin{equation}\label{E:iff1}
JV \in \underline{\mathfrak{g}}^{\perp_b} |_Y  \iff d\mu(x)(V) =0.
\end{equation}
Since $Y=\mu^{-1}(0)$, we have
\begin{equation}\label{E:iff2}
d\mu(x)(V) =0 \iff V \in T_xY.
\end{equation}
In particular, for $x \in Y$, $\dim \underline{\mathfrak{g}}^{\perp_b}_x = \dim (H_xX \cap T_xY) = \dim H_xY =2n-d.$ Moreover, $\dim \underline{\mathfrak{g}}=d$ and $\underline{\mathfrak{g}}+\underline{\mathfrak{g}}^{\perp_b}=HX.$ We conclude that, when restricted to $Y$, this sum is a direct sum, hence $\underline{\mathfrak{g}}_x \cap \underline{\mathfrak{g}}^{\perp_b}_x = \left\{ 0 \right\}.$ The lemma is proved.
\end{proof}

Let $U$ be a small open $G$-invariant neighborhood of $Y$. Since $G$ acts freely on $Y$, we can thus also assume that $G$ acts freely on $\overline{U}$. Since $\underline{\mathfrak{g}}_x \cap \underline{\mathfrak{g}}^{\perp_b}_x = \left\{ 0 \right\}$, for $x\in Y$, we have, for $x \in Y$,
\begin{equation}\label{E:bdsum}
H_xU = \underline{\mathfrak{g}}_x \oplus \underline{\mathfrak{g}}^{\perp_b}_x.
\end{equation}
Then, by \eqref{E:bdsum}, we can choose the horizontal bundles of the fibrations $U \to U_G:=U/G$ and $Y \to Y_G$ to be
\begin{equation}\label{E:hhx}
H^HU = \underline{\mathfrak{g}}^{\perp_b}|_U
\end{equation}
and 
\begin{equation}\label{E:hhy}
H^HY := H^HU|_Y \cap HY.
\end{equation}
Hence
\begin{equation}
 HY = \underline{\mathfrak{g}}|_Y \oplus H^HY.
\end{equation}

\begin{lemma}\label{l-gue170120}
\begin{equation}\label{E:jhy}
\underline{\mathfrak{g}}^{\perp_b} |_Y =JHY.
\end{equation}
\begin{equation}\label{E:split}
HU|_Y = J\underline{\mathfrak{g}}|_Y \oplus HY = \underline{\mathfrak{g}}|_Y \oplus J\underline{\mathfrak{g}}|_Y \oplus H^HY.
\end{equation}
\end{lemma}
\begin{proof}
The identity \eqref{E:jhy} follows from \eqref{E:iff1} and \eqref{E:iff2}. For $x \in Y, V \in H_xY$ and $\xi \in \mathfrak{g}$,
\begin{equation}\label{E:iff3}
b_x(J\xi_X, V) = d\omega_0(x)(\xi_X, V) = \left( d\mu(x)(V) \right)(\xi)=0.
\end{equation}
Using \eqref{E:iff3}, $\dim H_xU = \dim H_xY + \dim J\underline{\mathfrak{g}}_x$, and the fact that $b$ is nondegenerate on $JHY$, we obtain \eqref{E:split}.
\end{proof}

By \eqref{E:hhx}, \eqref{E:hhy} and \eqref{E:jhy}, we have
\begin{equation}
H^HY = JHY \cap HY.
\end{equation}
In particular, $H^HY$ is preserved by $J$, so we can define the homomorphism $J_G$ on $HY_G$ in the following way:
For $V \in HY_G$, we denote by $V^H$ its lift in $H^HY$, and we define $J_G$ on $Y_G$ by
\begin{equation}\label{E:jg}
(J_GV)^H = J(V^H).
\end{equation}
Hence, we have $J_G: HY_G \to HY_G$ such that $J_G^2 = -\operatorname{id}$, where $\operatorname{id}$ denotes the identity map $\operatorname{id}  \, : \, HY_G \to HY_G.$ By complex linear extension of $J_G$ to $\mathbb{C}TY_G$, we can define the $i$-eigenspace of $J_G$ is given by $T^{1,0}Y_G \, = \, \left\{ V \in \mathbb{C}HY_G \, : \, J_GV \, =  \,  \sqrt{-1} V  \right\}.$
  
\begin{theorem}\label{t-gue170128b}
The subbundle $T^{1,0}Y_G$ is a CR structure of $Y_G$.
\end{theorem}
\begin{proof}
Let $u, v \in C^\infty(Y_G, T^{1,0}Y_G)$, then we can find $U, V \in C^\infty(Y_G, TY_G)$ such that 
\[
u= U - \sqrt{-1}J_GU, \qquad v=V-\sqrt{-1}J_GV.
\]
By \eqref{E:jg}, we have
\[
u^H=U^H-\sqrt{-1}JU^H, \quad v = V^H - \sqrt{-1}JV^H \in T^{1,0}X \cap \mathbb{C}HY.
\]
Since $T^{1,0}X$ is a CR structure and it is clearly that $[u^H, v^H] \in \mathbb{C}HY,$ we have $[u^H, v^H] \in T^{1,0}X \cap \mathbb{C}HY.$ Hence, there is a $W \in C^\infty(X, HX)$ such that
\[
[u^H, v^H] = W-\sqrt{-1}JW.
\]  
In particular, $W, JW \in HY$. Thus, $W \in HY \cap JHY = H^HY$. Let $X^H \in H^HY$ be a lift of $X \in TY_G$ such that $X^H=W$. Then we have
\[
[u, v] = \pi_*[u^H, v^H] = \pi_*(X^H -\sqrt{-1}JX^H) = X - \sqrt{-1}J_GX \in T^{1,0}Y_G,
\]
i.e. we have $[C^\infty(Y_G, T^{1,0}Y_G), C^\infty(Y_G, T^{1,0}Y_G)] \subset C^\infty(Y_G, T^{1,0}Y_G).$ Therefore, $T^{1,0}Y_G$ is a CR structure of $Y_G.$
\end{proof}




\section{$G$-invariant Szeg\"o kernel asymptotics}\label{s-gue161109a}

In this section, we will establish asymptotic expansion for the $G$-invariant Szeg\"o kernel. We first review some well-known results for Szeg\"o kernel. 

\subsection{Szeg\"o kernel asymptotics}\label{s-gue161109I}

In this subsection, we don't assume that our CR manifold admits a compact Lie group action but we still assume that the Levi form is non-degenerate of constant signature $(n_-,n_+)$. The Hermitian metric $\langle\,\cdot\,|\,\cdot\,\rangle$ on $\Complex TX$ induces,
by duality, a Hermitian metric on $\Complex T^*X$ and also on the bundles of $(0,q)$ 
forms $T^{*0,q}X$, $q=0,1,\ldots,n$. We shall also denote all these induced metrics 
by $\langle\,\cdot\,|\,\cdot\,\rangle$. For $u\in T^{*0,q}X$, we write $\abs{u}^2:=\langle\,u\,|\,u\,\rangle$.  
Let $D\subset X$ be an open set. Let $\Omega^{0,q}(D)$ denote the space of smooth sections 
of $T^{*0,q}X$ over $D$ and let $\Omega^{0,q}_0(D)$ be the subspace of
$\Omega^{0,q}(D)$ whose elements have compact support in $D$. 

Let 
\begin{equation} \label{e-suIV}
\ddbar_b:\Omega^{0,q}(X)\To\Omega^{0,q+1}(X)
\end{equation}
be the tangential Cauchy-Riemann operator. Let $dv(x)$ be the volume form induced by the Hermitian metric $\langle\,\cdot\,|\,\cdot\,\rangle$.
The natural global $L^2$ inner product $(\,\cdot\,|\,\cdot\,)$ on $\Omega^{0,q}(X)$ 
induced by $dv(x)$ and $\langle\,\cdot\,|\,\cdot\,\rangle$ is given by
\begin{equation}\label{e:l2}
(\,u\,|\,v\,):=\int_X\langle\,u(x)\,|\,v(x)\,\rangle\, dv(x)\,,\quad u,v\in\Omega^{0,q}(X)\,.
\end{equation}
We denote by $L^2_{(0,q)}(X)$ 
the completion of $\Omega^{0,q}(X)$ with respect to $(\,\cdot\,|\,\cdot\,)$. 
We write $L^2(X):=L^2_{(0,0)}(X)$. We extend $(\,\cdot\,|\,\cdot\,)$ to $L^2_{(0,q)}(X)$ 
in the standard way. For $f\in L^2_{(0,q)}(X)$, we denote $\norm{f}^2:=(\,f\,|\,f\,)$.
We extend
$\ddbar_{b}$ to $L^2_{(0,r)}(X)$, $r=0,1,\ldots,n$, by
\begin{equation}\label{e-suVII}
\ddbar_{b}:{\rm Dom\,}\ddbar_{b}\subset L^2_{(0,r)}(X)\To L^2_{(0,r+1)}(X)\,,
\end{equation}
where ${\rm Dom\,}\ddbar_{b}:=\{u\in L^2_{(0,r)}(X);\, \ddbar_{b}u\in L^2_{(0,r+1)}(X)\}$ and, for any $u\in L^2_{(0,r)}(X)$, $\ddbar_{b} u$ is defined in the sense of distributions.
We also write
\begin{equation}\label{e-suVIII}
\ol{\pr}^{*}_{b}:{\rm Dom\,}\ol{\pr}^{*}_{b}\subset L^2_{(0,r+1)}(X)\To L^2_{(0,r)}(X)
\end{equation}
to denote the Hilbert space adjoint of $\ddbar_{b}$ in the $L^2$ space with respect to $(\,\cdot\,|\,\cdot\, )$.
Let $\Box^{(q)}_{b}$ denote the (Gaffney extension) of the Kohn Laplacian given by
\begin{equation}\label{e-suIX}
\begin{split}
{\rm Dom\,}\Box^{(q)}_{b}=\Big\{s\in L^2_{(0,q)}(X);&\, 
s\in{\rm Dom\,}\ddbar_{b}\cap{\rm Dom\,}\ol{\pr}^{*}_{b},\,
\ddbar_{b}s\in{\rm Dom\,}\ol{\pr}^{*}_{b},\\
&\quad\ol{\pr}^{*}_{b}s\in{\rm Dom\,}\ddbar_{b}\Big\}\,,\\
\Box^{(q)}_{b}s&=\ddbar_{b}\ol{\pr}^{*}_{b}s+\ol{\pr}^{*}_{b}\ddbar_{b}s
\:\:\text{for $s\in {\rm Dom\,}\Box^{(q)}_{b}$}\,.
 \end{split}
\end{equation}
By a result of Gaffney, for every $q=0,1,\ldots,n$, $\Box^{(q)}_{b}$ is a positive self-adjoint operator 
(see \cite[Proposition\,3.1.2]{MM}). That is, $\Box^{(q)}_{b}$ is self-adjoint and 
the spectrum of $\Box^{(q)}_{b}$ is contained in $\ol\Real_+$, $q=0,1,\ldots,n$. Let
\begin{equation}\label{e-suXI-I}
S^{(q)}:L^2_{(0,q)}(X)\To{\rm Ker\,}\Box^{(q)}_b
\end{equation}
be the orthogonal projections with respect to the $L^2$ inner product $(\,\cdot\,|\,\cdot\,)$ and let
\begin{equation}\label{e-suXI-II}
S^{(q)}(x,y)\in D'(X\times X,T^{*0,q}X\boxtimes(T^{*0,q}X)^*)
\end{equation}
denote the distribution kernel of $S^{(q)}$. 

We recall H\"ormander symbol space. Let $D\subset X$ be a local coordinate patch with local coordinates $x=(x_1,\ldots,x_{2n+1})$. 

\begin{definition}\label{d-gue140221a}
For $m\in\Real$, $S^m_{1,0}(D\times D\times\mathbb{R}_+,T^{*0,q}X\boxtimes(T^{*0,q}X)^*)$ 
is the space of all $a(x,y,t)\in C^\infty(D\times D\times\mathbb{R}_+,T^{*0,q}X\boxtimes(T^{*0,q}X)^*)$ 
such that, for all compact $K\Subset D\times D$ and all $\alpha, \beta\in\mathbb N^{2n+1}_0$, $\gamma\in\mathbb N_0$, 
there is a constant $C_{\alpha,\beta,\gamma}>0$ such that 
\[\abs{\pr^\alpha_x\pr^\beta_y\pr^\gamma_t a(x,y,t)}\leq C_{\alpha,\beta,\gamma}(1+\abs{t})^{m-\gamma},\ \ 
\forall (x,y,t)\in K\times\Real_+,\ \ t\geq1.\]
Put 
\[\begin{split}
&S^{-\infty}(D\times D\times\mathbb{R}_+,T^{*0,q}X\boxtimes(T^{*0,q}X)^*)\\
&:=\bigcap_{m\in\Real}S^m_{1,0}(D\times D\times\mathbb{R}_+,T^{*0,q}X\boxtimes(T^{*0,q}X)^*).\end{split}\]
Let $a_j\in S^{m_j}_{1,0}(D\times D\times\mathbb{R}_+,T^{*0,q}X\boxtimes(T^{*0,q}X)^*)$, 
$j=0,1,2,\ldots$ with $m_j\To-\infty$, as $j\To\infty$. 
Then there exists $a\in S^{m_0}_{1,0}(D\times D\times\mathbb{R}_+,T^{*0,q}X\boxtimes(T^{*0,q}X)^*)$ 
unique modulo $S^{-\infty}$, such that 
$a-\sum^{k-1}_{j=0}a_j\in S^{m_k}_{1,0}(D\times D\times\mathbb{R}_+,T^{*0,q}X\boxtimes(T^{*0,q}X)^*\big)$ 
for $k=0,1,2,\ldots$. 

If $a$ and $a_j$ have the properties above, we write $a\sim\sum^{\infty}_{j=0}a_j$ in 
$S^{m_0}_{1,0}\big(D\times D\times\mathbb{R}_+,T^{*0,q}X\boxtimes(T^{*0,q}X)^*\big)$. 
We write
\begin{equation}  \label{e-gue140205III}
s(x, y, t)\in S^{m}_{{\rm cl\,}}\big(D\times D\times\mathbb{R}_+,T^{*0,q}X\boxtimes(T^{*0,q}X)^*\big)
\end{equation}
if $s(x, y, t)\in S^{m}_{1,0}\big(D\times D\times\mathbb{R}_+,T^{*0,q}X\boxtimes(T^{*0,q}X)^*\big)$ and 
\begin{equation}\label{e-fal}\begin{split}
&s(x, y, t)\sim\sum^\infty_{j=0}s^j(x, y)t^{m-j}\text{ in }S^{m}_{1, 0}
\big(D\times D\times\mathbb{R}_+\,,T^{*0,q}X\boxtimes(T^{*0,q}X)^*\big)\,,\\
&s^j(x, y)\in C^\infty\big(D\times D,T^{*0,q}X\boxtimes(T^{*0,q}X)^*\big),\ j\in\N_0.\end{split}\end{equation}
\end{definition}

The following is well-known (see Theorem 4.8 in~\cite{HM14})

\begin{theorem}\label{t-gue161109}
Given $q=0,1,2,\ldots,n$. Assume that 
$q\notin\set{n_-,n_+}$. Then, $S^{(q)}\equiv0$ on $X$. 
\end{theorem}

The following is also well-known (see Theorem 1.2 in~\cite{Hsiao08}, Theorem 4.7 in~\cite{HM14} and see also~\cite{BouSj76} for $q=0$)

\begin{theorem}\label{t-gue161109I}
We recall that we work with the assumption that the Levi form is non-degenerate of constant signature $(n_-,n_+)$ on $X$. 
Let $q=n_-$ or $n_+$. Suppose that $\Box^{(q)}_b$ has $L^2$ closed range. Then, 
\[S^{(q)}(x,y)\in C^\infty(X\times X\setminus{{\rm diag\,}(X\times X)},T^{*0,q}X\boxtimes(T^{*0,q}X)^*).\]
Let $D\subset X$ be any local coordinate patch with local coordinates $x=(x_1,\ldots,x_{2n+1})$. Then, there exist continuous operators
\[S_-, S_+:\Omega^{0,q}_0(D)\To D'(D,T^{*0,q}X)\]
such that 
\begin{equation}\label{e-gue161110}
S^{(q)}\equiv S_-+S_+\ \ \mbox{on $D$},
\end{equation}
\begin{equation}\label{e-gue161110I}
\begin{split}
&S_-=0\ \ \mbox{if $q\neq n_-$},\\
&S_+=0\ \ \mbox{if $q\neq n_+$},\\
\end{split}
\end{equation}
and if $q=n_-$, $S_-(x,y)$ satisfies
\[S_-(x, y)\equiv\int^{\infty}_{0}e^{i\varphi_-(x, y)t}s_-(x, y, t)dt\ \ \mbox{on $D$}\]
with 
\begin{equation}  \label{e-gue161110r}\begin{split}
&s_-(x, y, t)\in S^{n}_{1,0}(D\times D\times\mathbb{R}_+,T^{*0,q}X\boxtimes(T^{*0,q}X)^*), \\
&s_-(x, y, t)\sim\sum^\infty_{j=0}s^j_-(x, y)t^{n-j}\quad\text{ in }S^{n}_{1, 0}(D\times D\times\mathbb{R}_+,T^{*0,q}X\boxtimes(T^{*0,q}X)^*),\\
&s^j_-(x, y)\in C^\infty(D\times D,T^{*0,q}X\boxtimes(T^{*0,q}X)^*),\ \ j=0,1,2,3,\ldots,\\
&s^0_-(x,x)\neq0,\ \ \forall x\in D,
\end{split}\end{equation}
and the phase function $\varphi_-$  satisfies 
\begin{equation}\label{e-gue140205IV}
\begin{split}
&\varphi_-\in C^\infty(D\times D),\ \ {\rm Im\,}\varphi_-(x, y)\geq0,\\
&\varphi_-(x, x)=0,\ \ \varphi_-(x, y)\neq0\ \ \mbox{if}\ \ x\neq y,\\
&d_x\varphi_-(x, y)\big|_{x=y}=-\omega_0(x), \ \ d_y\varphi_-(x, y)\big|_{x=y}=\omega_0(x), \\
&\varphi_-(x, y)=-\ol\varphi_-(y, x), 
\end{split}
\end{equation}
if $q=n_+$, then $S_+(x,y)$ satisfies 
\[S_+(x, y)\equiv\int^{\infty}_{0}\!\! e^{i\varphi_+(x, y)t}s_+(x, y, t)dt\ \ \mbox{on $D$}\]
with
\begin{equation}  \label{e-gue161110rI}\begin{split}
&s_+(x, y, t)\in S^{n}_{1,0}(D\times D\times\mathbb{R}_+,T^{*0,q}X\boxtimes(T^{*0,q}X)^*), \\
&s_+(x, y, t)\sim\sum^\infty_{j=0}s^j_+(x, y)t^{n-j}\quad\text{ in }S^{n}_{1, 0}(D\times D\times\mathbb{R}_+,T^{*0,q}X\boxtimes(T^{*0,q}X)^*),\\
&s^j_+(x, y)\in C^\infty(D\times D,T^{*0,q}X\boxtimes(T^{*0,q}X)^*),\ \ j=0,1,2,3,\ldots,\\
&s^0_+(x,x)\neq0,\ \ \forall x\in D,
\end{split}\end{equation}
and $-\ol\varphi_+(x, y)=\varphi_-(x,y)$. 
\end{theorem}

The following result describes the phase function in local coordinates (see chapter 8 of part I in \cite{Hsiao08})
\begin{theorem} \label{t-gue161110g}
For a given point $p\in X$, let $\{W_j\}_{j=1}^{n}$
be an orthonormal frame of $T^{1, 0}X$ in a neighborhood of $p$
such that
the Levi form is diagonal at $p$, i.e.\ $\mathcal{L}_{x_{0}}(W_{j},\overline{W}_{s})=\delta_{j,s}\mu_{j}$, $j,s=1,\ldots,n$.
We take local coordinates
$x=(x_1,\ldots,x_{2n+1})$, $z_j=x_j+ix_{d+j}$, $j=1,\ldots,d$, $z_j=x_{2j-1}+ix_{2j}$, $j=d+1,\ldots,n$,
defined on some neighborhood of $p$ such that $\omega_0(p)=dx_{2n+1}$, $x(p)=0$, 
and, for some $c_j\in\Complex$, $j=1,\ldots,n$\,,
\begin{equation}\label{e-gue161219a}
W_j=\frac{\pr}{\pr z_j}-i\mu_j\ol z_j\frac{\pr}{\pr x_{2n+1}}-
c_jx_{2n+1}\frac{\pr}{\pr x_{2n+1}}+\sum^{2n}_{k=1}a_{j,k}(x)\frac{\pr}{\pr x_k}+O(\abs{x}^2),\ j=1,\ldots,n\,,\end{equation}
where $a_{j,k}(x)\in C^\infty$, $a_{j,k}(x)=O(\abs{x})$, for every $j=1,\ldots,n$, $k=1,\ldots,2n$. 
Set
$y=(y_1,\ldots,y_{2n+1})$, $w_j=y_j+iy_{d+j}$, $j=1,\ldots,d$, $w_j=y_{2j-1}+iy_{2j}$, $j=d+1,\ldots,n$.
Then, for $\varphi_-$ in Theorem~\ref{t-gue161109I}, we have
\begin{equation} \label{e-gue140205VI}
{\rm Im\,}\varphi_-(x,y)\geq c\sum^{2n}_{j=1}\abs{x_j-y_j}^2,\ \ c>0,
\end{equation}
in some neighbourhood of $(0,0)$ and
\begin{equation} \label{e-gue140205VII}
\begin{split}
&\varphi_-(x, y)=-x_{2n+1}+y_{2n+1}+i\sum^{n}_{j=1}\abs{\mu_j}\abs{z_j-w_j}^2 \\
&\quad+\sum^{n}_{j=1}\Bigr(i\mu_j(\ol z_jw_j-z_j\ol w_j)+c_j(-z_jx_{2n+1}+w_jy_{2n+1})\\
&\quad+\ol c_j(-\ol z_jx_{2n+1}+\ol w_jy_{2n+1})\Bigr)+(x_{2n+1}-y_{2n+1})f(x, y) +O(\abs{(x, y)}^3),
\end{split}
\end{equation}
where $f$ is smooth and satisfies $f(0,0)=0$, $f(x, y)=\ol f(y, x)$.
\end{theorem} 

The following formula for the leading term $s^0_-$ on the diagonal follows from \cite[\S 9]{Hsiao08}.

\begin{theorem} \label{t-gue140205III}
We assume that the Levi form is non-degenerate of constant signature
$(n_-,n_+)$ at each point of $X$. Suppose that $\Box^{(q)}_b$ has $L^2$ closed range.
If $q=n_-$, then, for the leading term $s^0_-(x,y)$ of the expansion \eqref{e-gue161110r} of $s_-(x,y,t)$, we have
\begin{equation}\label{e-gue140205VIII}
s^0_-(x_0, x_0)=\frac{1}{2}\pi^{-n-1}\abs{\det\mathcal{L}_{x_0}}\tau_{x_0,n_-}\,,\:\:x_0\in D, 
\end{equation}
where $\det\mathcal{L}_{x_0}$ is given by \eqref{det140530} and $\tau_{x_0,n_-}$ is given by \eqref{tau140530}. 

Similarly, if $q=n_+$, then, for the leading term $s^0_+(x,y)$ of the expansion \eqref{e-gue161110rI} of $s_+(x,y,t)$, we have
\begin{equation}\label{e-gue140205VIIIb}
s^0_+(x_0, x_0)=\frac{1}{2}\pi^{-n-1}\abs{\det\mathcal{L}_{x_0}}\tau_{x_0,n_+}\,,\:\:x_0\in D, 
\end{equation}
where $\tau_{x_0,n_+}$ is given by \eqref{tau140530}. 
\end{theorem}

\subsection{$G$-invariant Szeg\"o kernel}\label{s-gue161109}

Fix $g\in G$. Let $g^*:\Lambda^r_x(\Complex T^*X)\To\Lambda^r_{g^{-1}\circ x}(\Complex T^*X)$ be the pull-back map. Since $G$ preserves $J$, we have 
\[g^*:T^{*0,q}_xX\To T^{*0,q}_{g^{-1}\circ x}X,\  \ \forall x\in X.\]
Thus, for $u\in\Omega^{0,q}(X)$, we have $g^*u\in\Omega^{0,q}(X)$ and we write $(g^*u)(x) := u(g\circ x)$. Put 
\[\Omega^{0,q}(X)^G:=\set{u\in\Omega^{0,q}(X);\, g^*u=u,\ \ \forall g\in G}.\]
Now, we assume that the Hermitian metric $\langle\,\cdot\,|\,\cdot\,\rangle$ on $\Complex TX$ is $G$-invariant and $\underline{\mathfrak{g}}$ is orthogonal to $HY\bigcap JHY$ at every point of $Y$.  The Hermitian metric is $G$-invariant means that, for any $G$-invariant vector fields $U$ and $V$,  $\langle\,U\,|\,V\,\rangle$ is $G$-invariant. Then the $L^2$ inner product $(\,\cdot\,|\,\cdot\,)$ on $\Omega^{0,q}(X)$ 
induced by $\langle\,\cdot\,|\,\cdot\,\rangle$ is $G$-invariant, that is, $(\,u\,|\,v\,)=(\,g^*u\,|\,g^*v\,)$, for all $u, v\in\Omega^{0,q}(X)$, $g\in G$. Let $u\in L^2_{(0,q)}(X)$ and let $g\in G$.  Take $u_j\in\Omega^{0,q}(X)$, $j=1,2,\ldots$, with $u_j\To u$ in $L^2_{(0,q)}(X)$ as $j\To\infty$. Since $(\,\cdot\,|\,\cdot\,)$ is $G$-invariant, there is a $v\in L^2_{(0,q)}(X)$ such that $v=\lim_{j\To\infty}g^*u_j$. We define $g^*u:=v$. It is clear that the definition is well-defined. We have $g^*:L^2_{(0,q)}(X)\To L^2_{(0,q)}(X)$. Put 
\[L^2_{(0,q)}(X)^G:=\set{u\in L^2_{(0,q)}(X);\, g^*u=u,\ \ \forall g\in G}.\]
It is not difficult to see that $L^2_{(0,q)}(X)^G$ is
the completion of $\Omega^{0,q}(X)^G$ with respect to $(\,\cdot\,|\,\cdot\,)$. 
We write $L^2(X)^G:=L^2_{(0,0)}(X)^G$. Since $G$ preserves $J$ and $(\,\cdot\,|\,\cdot\,)$ is $G$-invariant, it is straightforward to see that
\begin{equation}\label{e-gue161231}
\begin{split}
&g^*\ddbar_b=\ddbar_bg^*\ \ \mbox{on $\Omega^{0,q}(X)$, for all $g\in G$},\\
&g^*\ol{\pr}^*_b=\ol{\pr}^*_bg^*\ \ \mbox{on $\Omega^{0,q}(X)$, for all $g\in G$},\\
&g^*\Box^{(q)}_b=\Box^{(q)}_bg^*\ \ \mbox{on $\Omega^{0,q}(X)$, for all $g\in G$}.\end{split}
\end{equation}
From \eqref{e-gue161231}, we conclude that, for every $g\in G$, 
\begin{equation}\label{e-gue161231I}
\begin{split}
&g^*: {\rm Dom\,}\ddbar_b\To{\rm Dom\,}\ddbar_b,\\
&g^*: {\rm Dom\,}\ol{\pr}^*_b \To{\rm Dom\,}\ol{\pr}^*_b,\\
&g^*: {\rm Dom\,}\Box^{(q)}_b \To{\rm Dom\,}\Box^{(q)}_b
\end{split}\end{equation}
and
\begin{equation}\label{e-gue161231II}
\begin{split}
&g^*\ddbar_b=\ddbar_bg^*\ \ \mbox{on ${\rm Dom\,}\ddbar_b$},\\
&g^*\ol{\pr}^*_b=\ol{\pr}^*_bg^*\ \ \mbox{on ${\rm Dom\,}\ol{\pr}^*_b$},\\
&g^*\Box^{(q)}_b=\Box^{(q)}_bg^*\ \ \mbox{on ${\rm Dom\,}\Box^{(q)}_b$}.
\end{split}\end{equation}
Put $({\rm Ker\,}\Box^{(q)}_b)^G:={\rm Ker\,}\Box^{(q)}_b\bigcap L^2_{(0,q)}(X)^G$. The $G$-invariant Szeg\"o projection is the orthogonal projection 
\[S^{(q)}_G:L^2_{(0,q)}(X)\To ({\rm Ker\,}\Box^{(q)}_b)^G\]
with respect to $(\,\cdot\,|\,\cdot\,)$. Let $S^{(q)}_G(x,y)\in D'(X\times X,T^{*0,q}X\boxtimes(T^{*0,q}X)^*)$ be the distribution kernel of $S^G$. Let $d\mu$ be a Haar measure on $G$. Then, 
\begin{equation}\label{e-gue161231III}
S^{(q)}_G(x,y)=\frac{1}{\abs{G}_{d\mu}}\int_GS^{(q)}(x,g\circ y)d\mu(g),
\end{equation}
where 
\begin{equation}\label{e-gue161231a}
\abs{G}_{d\mu}:=\int_Gd\mu.
\end{equation}
Note that the integral \eqref{e-gue161231III} is defined in the sense of distribution.

\subsection{$G$-invariant Szeg\"o kernel asymptotics near $\mu^{-1}(0)$}\label{s-gue161110w}

In this section, we will study $G$-invariant Szeg\"o kernel near $\mu^{-1}(0)$.

Let $e_0\in G$ be the identity element. 
Let $v=(v_1,\ldots,v_d)$ be the local coordinates of $G$ defined  in a neighborhood $V$ of $e_0$ with $v(e_0)=(0,\ldots,0)$. From now on, we will identify the element $e\in V$ with $v(e)$.  We first need 

\begin{theorem}\label{t-gue161202}
Let $p\in\mu^{-1}(0)$. There exist local coordinates $v=(v_1,\ldots,v_d)$ of $G$ defined in  a neighborhood $V$ of $e_0$ with $v(e_0)=(0,\ldots,0)$, local coordinates $x=(x_1,\ldots,x_{2n+1})$ of $X$ defined in a neighborhood $U=U_1\times U_2$ of $p$ with $0\leftrightarrow p$, where $U_1\subset\Real^d$ is an open set of $0\in\Real^d$,  $U_2\subset\Real^{2n+1-d}$ is an open set of $0\in\Real^{2n+1-d} $ and a smooth function $\gamma=(\gamma_1,\ldots,\gamma_d)\in C^\infty(U_2,U_1)$ with $\gamma(0)=0\in\Real^d$  such that
\begin{equation}\label{e-gue161202}
\begin{split}
&(v_1,\ldots,v_d)\circ (\gamma(x_{d+1},\ldots,x_{2n+1}),x_{d+1},\ldots,x_{2n+1})\\
&=(v_1+\gamma_1(x_{d+1},\ldots,x_{2n+1}),\ldots,v_d+\gamma_d(x_{d+1},\ldots,x_{2n+1}),x_{d+1},\ldots,x_{2n+1}),\\
&\forall (v_1,\ldots,v_d)\in V,\ \ \forall (x_{d+1},\ldots,x_{2n+1})\in U_2,
\end{split}
\end{equation}
\begin{equation}\label{e-gue161206}
\begin{split}
&\underline{\mathfrak{g}}={\rm span\,}\set{\frac{\pr}{\pr x_1},\ldots,\frac{\pr}{\pr x_d}},\\
&\mu^{-1}(0)\bigcap U=\set{x_{d+1}=\cdots=x_{2d}=0},\\
&\mbox{On $\mu^{-1}(0)\bigcap U$, we have $J(\frac{\pr}{\pr x_j})=\frac{\pr}{\pr x_{d+j}}+a_j(x)\frac{\pr}{\pr x_{2n+1}}$, $j=1,2,\ldots,d$}, 
\end{split}
\end{equation}
where $a_j(x)$ is a smooth function on $\mu^{-1}(0)\bigcap U$, independent of $x_1,\ldots,x_{2d}$, $x_{2n+1}$ and $a_j(0)=0$, $j=1,\ldots,d$, 
\begin{equation}\label{e-gue161202I}
\begin{split}
&T^{1,0}_pX={\rm span\,}\set{Z_1,\ldots,Z_n},\\
&Z_j=\frac{1}{2}(\frac{\pr}{\pr x_j}-i\frac{\pr}{\pr x_{d+j}})(p),\ \ j=1,\ldots,d,\\
&Z_j=\frac{1}{2}(\frac{\pr}{\pr x_{2j-1}}-i\frac{\pr}{\pr x_{2j}})(p),\ \ j=d+1,\ldots,n,\\
&\langle\,Z_j\,|\,Z_k\,\rangle=\delta_{j,k},\ \ j,k=1,2,\ldots,n,\\
&\mathcal{L}_p(Z_j, \ol Z_k)=\mu_j\delta_{j,k},\ \ j,k=1,2,\ldots,n
\end{split}
\end{equation}
and 
\begin{equation}\label{e-gue161219}
\begin{split}
\omega_0(x)&=(1+O(\abs{x}))dx_{2n+1}+\sum^d_{j=1}4\mu_jx_{d+j}dx_j\\
&\quad+\sum^n_{j=d+1}2\mu_jx_{2j}dx_{2j-1}-\sum^n_{j=d+1}2\mu_jx_{2j-1}dx_{2j}+\sum^{2n}_{j=d+1}b_jx_{2n+1}dx_j+O(\abs{x}^2),
\end{split}
\end{equation}
where $b_{d+1}\in\Real,\ldots,b_{2n}\in\Real$. 
\end{theorem}

\begin{proof}
From the standard proof of Frobenius Theorem, it is not difficult to see that there exist local coordinates $v=(v_1,\ldots,v_d)$ of $G$ defined in a neighborhood $V$ of $e_0$ with $v(e_0)=(0,\ldots,0)$ and local coordinates $x=(x_1,\ldots,x_{2n+1})$ of $X$ defined in a neighborhood $U$ of $p$ with $x(p)=0$ such that
\begin{equation}\label{e-gue161203}
\begin{split}
&(v_1,\ldots,v_d)\circ (0,\ldots,0,x_{d+1},\ldots,x_{2n+1})\\
&=(v_1,\ldots,v_d,x_{d+1},\ldots,x_{2n+1}),\ \ \forall (v_1,\ldots,v_d)\in V,\ \ \forall (0,\ldots,0,x_{d+1},\ldots,x_{2n+1})\in U,
\end{split}
\end{equation}
and
\begin{equation}\label{e-gue161103a}
\underline{\mathfrak{g}}={\rm span\,}\set{\frac{\pr}{\pr x_1},\ldots,\frac{\pr}{\pr x_d}}.
\end{equation}
Since $p\in\mu^{-1}(0)$, we have $\omega_0(p)(\frac{\pr}{\pr x_j}(p))=0$, $j=1,2,\ldots,d$,  and hence $ \frac{\pr}{\pr x_j}(p)\in H_pX$, $j=1,2,\ldots,d$. 
Consider the linear map 
\[\begin{split}
R:\underline{\mathfrak{g}}_p&\To\underline{\mathfrak{g}}_p,\\
u&\To Ru,\ \ \langle\,Ru\,|\,v\,\rangle=\langle\,d\omega_0\,,\,Ju\wedge v\,\rangle.
\end{split}\]
Since $R$ is self-adjoint, by using linear transformation in $(x_1,\ldots,x_d)$, we can take $(x_1,\ldots,x_d)$ such that 
\begin{equation}\label{e-gue161102}
\begin{split}
&\langle\,R\frac{\pr}{\pr x_j}(p)\,|\,\frac{\pr}{\pr x_k}(p)\,\rangle=4\mu_j\delta_{j,k},\ \ j,k=1,2,\ldots,d,\\
&\langle\,\frac{\pr}{\pr x_j}(p)\,|\,\frac{\pr}{\pr x_k}(p)\,\rangle=2\delta_{j,k},\ \ j,k=1,2,\ldots,d.
\end{split}
\end{equation}
By taking linear transformation in $(v_1,\ldots,v_d)$, \eqref{e-gue161203} still hold. 

Let $\omega_0(\frac{\pr}{\pr x_j})=a_j(x)\in C^\infty(U)$, $j=1,2,\ldots,d$. Since $a_j(x)$ is $G$-invariant, we have $\frac{\pr a_j(x)}{\pr x_s}=0$, $j,s=1,2,\ldots,d$.  By the defintion of the moment map, we have 
\begin{equation}\label{e-gue161103}
\mu^{-1}(0)\bigcap U=\set{x\in U;\, a_1(x)=\cdots=a_d(x)=0}.
\end{equation}
Since $p$ is a regular value of the moment map $\mu$, the matrix $\left(\frac{\pr a_j}{\pr x_s}(p)\right)_{1\leq j\leq d,d+1\leq s\leq 2n+1}$ is of rank $d$. We may assume that the matrix $\left(\frac{\pr a_j}{\pr x_s}(p)\right)_{1\leq j\leq d,d+1\leq s\leq 2d}$ is non-singular. Thus, $(x_1,\ldots,x_d,a_1,\ldots,a_d,x_{2d+1},\ldots,x_{2n+1})$ are also local coordinates of $X$. Hence, we can take $v=(v_1,\ldots,v_d)$ and $x=(x_1,\ldots,x_{2n+1})$ such that \eqref{e-gue161203}, \eqref{e-gue161103a}, \eqref{e-gue161102} hold and 
\begin{equation}\label{e-gue161103I}
\mu^{-1}(0)\bigcap U=\set{x=(x_1,\ldots,x_{2n+1})\in U;\, x_{d+1}=\cdots=x_{2d}=0}.
\end{equation}

On $\mu^{-1}(0)\bigcap U$, let
\[J(\frac{\pr}{\pr x_j})=b_{j,1}(x)\frac{\pr}{\pr x_1}+\cdots+b_{j,2n+1}(x)\frac{\pr}{\pr x_{2n+1}},\ \ j=1,2,\ldots,d.\]
Since we only work on $\mu^{-1}(0)$, $b_{j,k}(x)$ is independent of $x_{d+1},\ldots,x_{2d}$, for all $j=1,\ldots,d$, $k=1,\ldots,2n+1$. Moreover, it is easy to see that $b_{j,k}(x)$ is also independent of $x_{1},\ldots,x_{d}$, for all $j=1,\ldots,d$, $k=1,\ldots,2n+1$. Let $\Td x''=(x_{2d+1},\ldots,x_{2n+1})$. Hence, $b_{j,k}(x)=b_{j,k}(\Td x'')$, $j=1,\ldots,d$, $k=1,\ldots,2n+1$.
We claim that the matrix $\left(b_{j,k}(\Td x'')\right)_{1\leq j\leq d,d+1\leq k\leq 2d}$ is non-singular near $p$. If not, it is easy to see that there is a non-zero vector $u\in J\underline{\mathfrak{g}}\bigcap HY$, where $Y=\mu^{-1}(0)$. Let $u=Jv$, $v\in\underline{\mathfrak{g}}$. Then, $v\in\underline{\mathfrak{g}}\bigcap JHY=\underline{\mathfrak{g}}\bigcap\underline{\mathfrak{g}}^{\perp_b}$ (see \eqref{E:jhy}). Since $\underline{\mathfrak{g}} \cap \underline{\mathfrak{g}}^{\perp_b} = \left\{ 0 \right\}$ on $\mu^{-1}(0)$, we deduce that $v=0$ and we get a contradiction. The claim follows. From the claim, we can use linear transformation in $(x_{d+1},\ldots,x_{2d})$ (the linear transform depends smoothly on $\Td x''$) and we can take $(x_{d+1},\ldots,x_{2d})$ such that on $\mu^{-1}(0)$, 
\begin{equation}\label{e-gue161203cw}
J(\frac{\pr}{\pr x_j})=b_{j,1}(\Td x'')\frac{\pr}{\pr x_1}+\cdots+b_{j,d}(\Td x'')\frac{\pr}{\pr x_{d}}+\frac{\pr}{\pr x_{d+j}}+b_{j,2d+1}(\Td x'')\frac{\pr}{\pr x_{2d+1}}+\cdots+b_{j,2n+1}(\Td x'')\frac{\pr}{\pr x_{2n+1}},
\end{equation}
where $j=1,2,\ldots,d$. Consider the coordinates change:
\[\begin{split}
&x=(x_1,\ldots,x_{2n+1})\To u=(u_1,\ldots,u_{2n+1}),\\
&(x_1,\ldots,x_{2n+1})\\
&\To(x_1-\sum^d_{j=1}b_{j,1}(\Td x'')x_{d+j},\ldots,x_d-\sum^d_{j=1}b_{j,d}(\Td x'')x_{d+j},x_{d+1},\ldots,x_{2d},\\
&\quad\quad x_{2d+1}-\sum^d_{j=1}b_{j,2d+1}(\Td x'')x_{d+j},\ldots,x_{2n+1}-\sum^d_{j=1}b_{j,2n+1}(\Td x'')x_{d+j}).
\end{split}\]
Then, 
\[\begin{split}
&\frac{\pr}{\pr x_j}\rightarrow\frac{\pr}{\pr u_j},\ \ j=1,\ldots,d,2d+1,\ldots,2n+1,\\
&\frac{\pr}{\pr x_{d+j}}\\
&\rightarrow -b_{j,1}\frac{\pr}{\pr u_1}-\cdots-b_{j,d}\frac{\pr}{\pr u_d}+\frac{\pr}{\pr u_{d+j}}-b_{j,2d+1}\frac{\pr}{\pr u_{2d+1}}-\cdots-b_{j,2n+1}\frac{\pr}{\pr u_{2n+1}},\ \ j=1,\ldots,d.
\end{split}\]
Hence, on $\mu^{-1}(0)\bigcap U$, $J(\frac{\pr}{\pr x_j})\rightarrow\frac{\pr}{\pr u_{d+j}}$, $j=1,\ldots,d$. Thus, we can take $v=(v_1,\ldots,v_d)$ and $x=(x_1,\ldots,x_{2n+1})$ such that \eqref{e-gue161202}, \eqref{e-gue161103a}, \eqref{e-gue161102}, \eqref{e-gue161103I} hold and on $\mu^{-1}(0)\bigcap U$, 
\begin{equation}\label{e-gue161103II}
J(\frac{\pr}{\pr x_j})=\frac{\pr}{\pr x_{d+j}},\ \ j=1,2,\ldots,d.
\end{equation}

Let $Z_j=\frac{1}{2}(\frac{\pr }{\pr x_j}-i\frac{\pr}{\pr x_{d+j}})(p)\in T^{1,0}_pX$, $j=1,\ldots,d$. From \eqref{e-gue161102}, we can check that 
\begin{equation}\label{e-gue161103c}
\begin{split}
&\mathcal{L}_p(Z_j,\ol Z_k)=\mu_j\delta_{j,k},\ \ j,k=1,\ldots,d,\\
&\langle\,Z_j\,|\,Z_k\,\rangle=\delta_{j,k},\ \ j,k=1,\ldots,d.
\end{split}
\end{equation}
Since $\underline{\mathfrak{g}}_p$ is orthogonal to $H_pY\bigcap JH_pY$ and $H_pY\bigcap JH_pY\subset\underline{\mathfrak{g}}^{\perp_b}_p$, we can find an orthonormal frame $\set{Z_1,\ldots,Z_d,V_1,\ldots,V_{n-d}}$ for $T^{1,0}_pX$ such that the Levi form $\mathcal{L}_p$ is diagonalized with respect to $Z_1,\ldots,Z_d,V_1,\ldots,V_{n-d}$, where $V_1\in\Complex H_pY\bigcap J\Complex H_pY ,\ldots,V_{n-d}\in\Complex H_pY\bigcap J\Complex H_pY$. Write 
\[{\rm Re\,}V_j=\sum^{2n+1}_{k=1}\alpha_{j,k}\frac{\pr}{\pr x_k},\ \ {\rm Im\,}V_j=\sum^{2n+1}_{k=1}\beta_{j,k}\frac{\pr}{\pr x_k},\ \ j=1,\ldots,n-d.\]
We claim that $\alpha_{j,k}=\beta_{j,k}=0$, for all $k=d+1,\ldots,2d$, $j=1,\ldots,n-d$. Fix $j=1,\ldots,n-d$. Since ${\rm Re\,}V_j\in\underline{\mathfrak{g}}^{\perp_b}_p$ and ${\rm span\,}\set{\frac{\pr}{\pr x_{d+1}},\ldots,\frac{\pr}{\pr x_{2d}}}\in\underline{\mathfrak{g}}_p^{\perp_b}$, we conclude that 
\begin{equation}\label{e-gue161218}
\sum^d_{k=1}\alpha_{j,k}\frac{\pr}{\pr x_k}+\sum^{2n}_{k=2d+1}\alpha_{j,k}\frac{\pr}{\pr x_k}\in\underline{\mathfrak{g}}^{\perp_b}_p\bigcap H_pY.
\end{equation} 
From \eqref{E:jhy} and \eqref{e-gue161218}, we deduce that 
\[\sum^d_{k=1}\alpha_{j,k}\frac{\pr}{\pr x_k}+\sum^{2n}_{k=2d+1}\alpha_{j,k}\frac{\pr}{\pr x_k}\in JH_pY\bigcap H_pY=\underline{\mathfrak{g}}^{\perp_b}_p\bigcap H_pY\]
and hence 
\begin{equation}\label{e-gue161218I}
J\Bigr(\sum^d_{k=1}\alpha_{j,k}\frac{\pr}{\pr x_k}+\sum^{2n}_{k=2d+1}\alpha_{j,k}\frac{\pr}{\pr x_k}\Bigr)\in\underline{\mathfrak{g}}^{\perp_b}_p\bigcap H_pY.
\end{equation}
From \eqref{e-gue161218I} and notice that $J({\rm Re\,}V_j)\in\underline{\mathfrak{g}}^{\perp_b}_p$, we deduce that 
\[J(\sum^{2d}_{k=d+1}\alpha_{j,k}\frac{\pr}{\pr x_k})\in\underline{\mathfrak{g}}_p\bigcap\underline{\mathfrak{g}}^{\perp_b}_p=\set{0}.\]
Thus, $\alpha_{j,k}=0$, for all $k=d+1,\ldots,2d$, $j=1,\ldots,n-d$.
Similarly, we can repeat the procedure above and deduce that $\beta_{j,k}=0$, for all $k=d+1,\ldots,2d$, $j=1,\ldots,n-d$. 

Since ${\rm span\,}\set{{\rm Re\,}V_j, {\rm Im\,}V_j;\, j=1,\ldots,n-d}$ is transversal to $\underline{\mathfrak{g}}_p\oplus J\underline{\mathfrak{g}}_p$, we can take linear transformation in $(x_{2d+1},\ldots,x_{2n+1})$ so that 
\[\begin{split}
&{\rm Re\,}V_j=\alpha_{j,1}\frac{\pr}{\pr x_1}+\cdots+\alpha_{j,d}\frac{\pr}{\pr x_{d}}+\frac{\pr}{\pr x_{2j-1+2d}},\ \ j=1,2,\ldots,n-d,\\
&{\rm Im\,}V_j=\beta_{j,1}\frac{\pr}{\pr x_1}+\cdots+\beta_{j,d}\frac{\pr}{\pr x_{d}}+\frac{\pr}{\pr x_{2j+2d}},\ \ j=1,2,\ldots,n-d.
\end{split}\]
Consider the coordinates change:
\[\begin{split}
&x=(x_1,\ldots,x_{2n+1})\To u=(u_1,\ldots,u_{2n+1}),\\
&(x_1,\ldots,x_{2n+1})\\
&\To(x_1-\sum^d_{j=1}\alpha_{j,1}x_{2j-1+2d}-\sum^{d}_{j=1}\beta_{j,1}x_{2j+2d},\ldots,x_d-\sum^d_{j=1}\alpha_{j,d}x_{2j-1+2d}-\sum^{d}_{j=1}\beta_{j,d}x_{2j+2d},\\
&\quad\quad\quad x_{d+1},\ldots, x_{2n+1})
\end{split}\]
Then, 
\[\begin{split}
&\frac{\pr}{\pr x_j}\rightarrow\frac{\pr}{\pr u_j},\ \ j=1,\ldots,2d,\\
&\frac{\pr}{\pr x_{2j-1+2d}}\rightarrow -\alpha_{j,1}\frac{\pr}{\pr u_1}-\cdots-\alpha_{j,d}\frac{\pr}{\pr u_d}+\frac{\pr}{\pr u_{2j-1+2d}},\ \ j=1,\ldots,n-d,\\
&\frac{\pr}{\pr x_{2j+2d}}\rightarrow -\beta_{j,1}\frac{\pr}{\pr u_1}-\cdots-\beta_{j,d}\frac{\pr}{\pr u_d}+\frac{\pr}{\pr u_{2j+2d}},\ \ j=1,\ldots,n-d.
\end{split}\]
 Thus, we can take $v=(v_1,\ldots,v_d)$ and $x=(x_1,\ldots,x_{2n+1})$ such that \eqref{e-gue161202}, \eqref{e-gue161206} and \eqref{e-gue161202I} hold. 
 
Now, we can take linear transformation in $x_{2n+1}$ so that $\omega_0(p)=dx_{2n+1}$. Let $W_j$, $j=1,\ldots,n$ be an orthonormal basis of $T^{1,0}X$ such that $W_j(p)=Z_j$, $j=1,\ldots,n$, where $Z_j\in T^{1,0}_pX$, $j=1,\ldots,n$, are as in \eqref{e-gue161202I}. Let $\Td x=(\Td x_1,\ldots,\Td x_{2n+1})$ be the coordinates as in Theorem~\ref{t-gue161110g}. It is easy to see that 
\begin{equation}\label{e-gue161219b}\begin{split}
&\Td x_j=x_j+a_jx_{2n+1}+h_j(x),\ \ h_j(x)=O(\abs{x}^2),\ \ a_j\in\Real,\ \ j=1,2,\ldots,2n,\\
&\Td x_{2n+1}=x_{2n+1}+h_{2n+1}(x),\ \ h_{2n+1}(x)=O(\abs{x}^2).\end{split}\end{equation}
We may change $x_{2n+1}$ be $x_{2n+1}+h_{2n+1}(0,\ldots,0,x_{d+1},\ldots,x_{2n},0)$ and we have 
\begin{equation}\label{e-gue161209}
\frac{\pr^2\Td x_{2n+1}}{\pr x_j\pr x_k}(p)=0,\ \ j, k=\set{d+1,\ldots,2n}.
\end{equation}
Note that when we change $x_{2n+1}$ to $x_{2n+1}+h_{2n+1}(0,\ldots,0,x_{d+1},\ldots,x_{2n},0)$, $\frac{\pr}{\pr x_j}$ will change to $\frac{\pr}{\pr x_j}+\alpha_j(x)\frac{\pr}{\pr x_{2n+1}}$, $j=d+1,\ldots,2n$, where $\alpha_j(x)$ is a smooth function on $\mu^{-1}(0)\bigcap U$, independent of $x_1,\ldots,x_{d}$, $x_{2n+1}$ and $\alpha_j(0)=0$, $j=d+1,\ldots,2n$. Hence, on $\mu^{-1}(0)\bigcap U$, we have $J(\frac{\pr}{\pr x_j})=\frac{\pr}{\pr x_{d+j}}+a_j(x)\frac{\pr}{\pr x_{2n+1}}$, $j=1,2,\ldots,d$,  where $a_j(x)$ is a smooth function on $\mu^{-1}(0)\bigcap U$, independent of $x_1,\ldots,x_{2d}$, $x_{2n+1}$ and $a_j(0)=0$, $j=1,\ldots,d$. 

From \eqref{e-gue161219a} and \eqref{e-gue161219b}, it is straightforward to see that 
\begin{equation}\label{e-gue161209I}
\begin{split}
&\omega_0(\Td x)=(1+O(\abs{\Td x}))d\Td x_{2n+1}+\sum^d_{j=1}2\mu_j\Td x_{d+j}d\Td x_j+\sum^d_{j=1}(-2\mu_j\Td x_j)d\Td x_{d+j}\\
&\quad+\sum^n_{j=d+1}2\mu_jx_{2j}dx_{2j-1}+\sum^n_{j=d+1}(-2\mu_j\Td x_{2j-1})d\Td x_{2j}+\sum^{2n}_{j=1}\hat b_j\Td x_{2n+1}d\Td x_j+O(\abs{x}^2)\\
&=(1+O(\abs{x}))dx_{2n+1}+\sum^d_{j=1}(2\mu_jx_{d+j}+\frac{\pr\Td x_{2n+1}}{\pr x_j})dx_j+\sum^d_{j=1}(-2\mu_jx_j+\frac{\pr\Td x_{2n+1}}{\pr x_{d+j}})dx_{d+j}\\
&\quad+\sum^n_{j=d+1}(2\mu_jx_{2j}+\frac{\pr\Td x_{2n+1}}{\pr x_{2j-1}})dx_{2j-1}+\sum^n_{j=d+1}(-2\mu_jx_{2j-1}+\frac{\pr\Td x_{2n+1}}{\pr x_{2j}})dx_{2j}\\
&\quad+\sum^{2n}_{j=1}\Td b_jx_{2n+1}dx_j+O(\abs{x}^2),
\end{split}
\end{equation}
where $\Td b_j\in\Real, \hat b_j\in\Real$, $j=1,\ldots,2n$. Note that $\omega_0$ is $G$-invariant. From this observation and \eqref{e-gue161209I}, we deduce that 
\begin{equation}\label{e-gue161209II}
\begin{split}
&\frac{\pr^2\Td x_{2n+1}}{\pr x_j\pr x_k}(p)=0,\ \ j\in\set{1,\ldots,d}, k\in\set{1,\ldots,d}\bigcup\set{2d+1,\ldots,2n},\\
&\frac{\pr^2\Td x_{2n+1}}{\pr x_{d+j}\pr x_k}(p)=2\mu_j\delta_{j,k},\ \ j, k\in\set{1,\ldots,d}.
\end{split}
\end{equation}
From \eqref{e-gue161209II}, \eqref{e-gue161209I} and \eqref{e-gue161209}, it is straightforward to see that
\begin{equation}\label{e-gue161219q}
\begin{split}
\omega_0(x)&=(1+O(\abs{x}))dx_{2n+1}+\sum^d_{j=1}4\mu_jx_{d+j}dx_j\\
&\quad+\sum^n_{j=d+1}2\mu_jx_{2j}dx_{2j-1}-\sum^n_{j=d+1}2\mu_jx_{2j-1}dx_{2j}+\sum^{2n}_{j=1}b_jx_{2n+1}dx_j+O(\abs{x}^2),
\end{split}
\end{equation}
where $b_{1}\in\Real,\ldots,b_{2n}\in\Real$. Since $\omega_0(p)(\frac{\pr}{\pr x_j})=0$ on $x_{d+1}=\cdots=x_{2d}=0$, $j=1,2,\ldots,d$, we deduce that $b_1=\cdots=b_d=0$ and we get
\eqref{e-gue161219}. The theorem follows. 
\end{proof}

We need 

\begin{theorem}\label{t-gue161222}
Let $p\in\mu^{-1}(0)$ and take local coordinates $x=(x_1,\ldots,x_{2n+1})$ of $X$ defined in an open set $U$of $p$ with $0\leftrightarrow p$ such that \eqref{e-gue161206}, \eqref{e-gue161202I} and \eqref{e-gue161219} hold.  Let $\varphi_-(x,y)\in C^\infty(U\times U)$ be as in Theorem~\ref{t-gue161109I}. Then, 
\begin{equation} \label{e-gue161222}
\begin{split}
\varphi_-(x, y)&=-x_{2n+1}+y_{2n+1}-2\sum^d_{j=1}\mu_jx_jx_{d+j}+2\sum^d_{j=1}\mu_jy_jy_{d+j}
+i\sum^{n}_{j=1}\abs{\mu_j}\abs{z_j-w_j}^2 \\
&+\sum^{n}_{j=1}i\mu_j(\ol z_jw_j-z_j\ol w_j)+\sum^d_{j=1}(-\frac{i}{2}b_{d+j})(-z_jx_{2n+1}+w_jy_{2n+1})\\
&+\sum^d_{j=1}(\frac{i}{2}b_{d+j})(-\ol z_jx_{2n+1}+\ol w_jy_{2n+1})+\sum^n_{j=d+1}\frac{1}{2}(b_{2j-1}-ib_{2j})(-z_jx_{2n+1}+w_jy_{2n+1})\\
&+\sum^n_{j=d+1}\frac{1}{2}(b_{2j-1}+ib_{2j})(-\ol z_jx_{2n+1}+\ol w_jy_{2n+1})+(x_{2n+1}-y_{2n+1})f(x, y) +O(\abs{(x, y)}^3),
\end{split}
\end{equation}
where $z_j=x_j+ix_{d+j}$, $j=1,\ldots,d$, $z_j=x_{2j-1}+ix_{2j}$, $j=2d+1,\ldots,2n$, $\mu_j$, $j=1,\ldots,n$, and $b_{d+1}\in\Real,\ldots,b_{2n}\in\Real$ are as in \eqref{e-gue161219} and $f$ is smooth and satisfies $f(0,0)=0$, $f(x, y)=\ol f(y, x)$. 
\end{theorem}

\begin{proof}
 Let $\Td x=(\Td x_1,\ldots,\Td x_{2n+1})$ be the coordinates as in Theorem~\ref{t-gue161110g}. It is easy to see that 
\begin{equation}\label{e-gue161222I}\begin{split}
&\Td x_j=x_j+a_jx_{2n+1}+h_j(x),\ \ h_j(x)=O(\abs{x}^2),\ \ a_j\in\Real,\ \ j=1,2,\ldots,2n,\\
&\Td x_{2n+1}=x_{2n+1}+h_{2n+1}(x),\ \ h_{2n+1}(x)=O(\abs{x}^2).\end{split}\end{equation}
From \eqref{e-gue161219a}, it is straightforward to see that 
\begin{equation}\label{e-gue161222II}
\begin{split}
&\omega_0(\Td x)=(1+O(\abs{\Td x}))d\Td x_{2n+1}+\sum^d_{j=1}2\mu_j\Td x_{d+j}d\Td x_j+\sum^d_{j=1}(-2\mu_j\Td x_j)d\Td x_{d+j}\\
&\quad+\sum^n_{j=d+1}2\mu_jx_{2j}dx_{2j-1}+\sum^n_{j=d+1}(-2\mu_j\Td x_{2j-1})d\Td x_{2j}+\sum^{2n}_{j=1}\hat b_j\Td x_{2n+1}d\Td x_j+O(\abs{x}^2), 
\end{split}
\end{equation}
where 
\[\begin{split}
&\hat b_j=c_j+\ol c_j,\ \ j\in\set{1,\ldots,d}\bigcup\set{2d+1,2d+3,\ldots,2n-1},\\
&\hat b_j=ic_j-i\ol c_j,\ \ j\in\set{d+1,\ldots,2d}\bigcup\set{2d+2,\ldots,2n}.\end{split}\]
From \eqref{e-gue161222II} and \eqref{e-gue161219}, it is not difficulty to see that  (see also \eqref{e-gue161209I}) 
\begin{equation}\label{e-gue161222III}
\begin{split}
&\frac{\pr^2\Td x_{2n+1}}{\pr x_j\pr x_k}(p)=0,\ \ j\in\set{1,\ldots,d}, k\in\set{1,\ldots,d},\\
&\frac{\pr^2\Td x_{2n+1}}{\pr x_j\pr x_k}(p)=0,\ \ j\in\set{1,\ldots,2n}, k\in\set{2d+1,\ldots,2n},\\
&\frac{\pr^2\Td x_{2n+1}}{\pr x_{d+j}\pr x_k}(p)=2\mu_j\delta_{j,k},\ \ j, k\in\set{1,\ldots,d}.
\end{split}
\end{equation}
From \eqref{e-gue161222I}, \eqref{e-gue161222III} and \eqref{e-gue140205VII}, it is straightforward to check that 
\begin{equation} \label{e-gue161222a}
\begin{split}
\varphi_-(x, y)&=-x_{2n+1}+y_{2n+1}-2\sum^d_{j=1}\mu_jx_jx_{d+j}+2\sum^d_{j=1}\mu_jy_jy_{d+j}
+i\sum^{n}_{j=1}\abs{\mu_j}\abs{z_j-w_j}^2 \\
&+\sum^{n}_{j=1}i\mu_j(\ol z_jw_j-z_j\ol w_j)+\sum^n_{j=1}\beta_j(-z_jx_{2n+1}+w_jy_{2n+1})\\
&+\sum^n_{j=1}\ol\beta_j(-\ol z_jx_{2n+1}+\ol w_jy_{2n+1})+(x_{2n+1}-y_{2n+1})f(x, y) +O(\abs{(x, y)}^3),
\end{split}
\end{equation}
where $\beta_j\in\Complex$, $j=1,\ldots,n$ and $f$ is smooth and satisfies $f(0,0)=0$, $f(x, y)=\ol f(y, x)$. We now determine $\beta_j$, $j=1,\ldots,n$. We can compute that 
\begin{equation}\label{e-gue161222aI}
\begin{split}
&\frac{\pr\varphi_-}{\pr x_j}(x,x)=-4\mu_jx_{d+j}-(\beta_j+\ol\beta_j)x_{2n+1}+O(\abs{x}^2),\ \ j=1,\ldots,d,\\
&\frac{\pr\varphi_-}{\pr x_{d+j}}(x,x)=-i(\beta_j-\ol\beta_j)x_{2n+1}+O(\abs{x}^2),\ \ j=1,\ldots,d,\\
&\frac{\pr\varphi_-}{\pr x_{2j-1}}(x,x)=-2\mu_jx_{2j}-(\beta_j+\ol\beta_j)x_{2n+1}+O(\abs{x}^2),\ \ j=d+1,\ldots,n,\\
&\frac{\pr\varphi_-}{\pr x_{2j}}(x,x)=2\mu_jx_{2j-1}+(-i\beta_j+i\ol\beta_j)x_{2n+1}+O(\abs{x}^2),\ \ j=d+1,\ldots,n.
\end{split}
\end{equation}
Note that $d_x\varphi_-(x,x)=-\omega_0(x)$. From this observation and \eqref{e-gue161219}, we deduce that 
\begin{equation}\label{e-gue161222aII}
\begin{split}
&\frac{\pr\varphi_-}{\pr x_j}(x,x)=-4\mu_jx_{d+j}+O(\abs{x}^2),\ \ j=1,\ldots,d,\\
&\frac{\pr\varphi_-}{\pr x_{d+j}}(x,x)=-b_{d+j}x_{2n+1}+O(\abs{x}^2),\ \ j=1,\ldots,d,\\
&\frac{\pr\varphi_-}{\pr x_{2j-1}}(x,x)=-2\mu_jx_{2j}-b_{2j-1}x_{2n+1}+O(\abs{x}^2),\ \ j=d+1,\ldots,n,\\
&\frac{\pr\varphi_-}{\pr x_{2j}}(x,x)=2\mu_jx_{2j-1}-b_{2j}x_{2n+1}+O(\abs{x}^2),\ \ j=d+1,\ldots,n.
\end{split}
\end{equation}
From \eqref{e-gue161222aI} and \eqref{e-gue161222aII}, we deduce that 
\begin{equation}\label{e-gue161222aIII}
\begin{split}
&\beta_j=-\frac{i}{2}b_{d+j},\ \ j=1,\ldots,d,\\
&\beta_j=\frac{1}{2}(b_{2j-1}-ib_{2j}),\ \ j=d+1,\ldots,n.
\end{split}
\end{equation}
From \eqref{e-gue161222aIII} and \eqref{e-gue161222a}, we get \eqref{e-gue161222}. 
\end{proof}

We now work with local coordinates as in Theorem~\ref{t-gue161202}. 
From \eqref{e-gue161222}, we see that near $(p,p)\in U\times U$, we have $\frac{\pr\varphi_-}{\pr y_{2n+1}}\neq0$.
Using the Malgrange preparation theorem \cite[Th.\,7.5.7]{Hor03}, we have
\begin{equation} \label{e-gue170105}
\varphi_-(x,y)=g(x,y)(y_{2n+1}+\hat\varphi_-(x,\mathring{y}))
\end{equation}
in some neighborhood of $(p,p)$, where $\mathring{y}=(y_1,\ldots,y_{2n})$, $g, \hat\varphi_-\in C^\infty$. Since ${\rm
Im\,}\varphi_-\geq0$, it is not difficult to see that
${\rm Im\,}\hat\varphi_-\geq0$ in some neighborhood of $(p,p)$. We may take $U$
small enough so that \eqref{e-gue170105} holds and ${\rm
Im\,}\hat\varphi_-\geq0$ on $U\times U$.  From the global theory of Fourier integral
operators \cite[Th.\,4.2]{MS74}, we see that
$\varphi_-(x,y)t$ and $(y_{2n+1}+\hat\varphi_-(x,\mathring{y}))t$ are equivalent in the
sense of Melin-Sj\"{o}strand. We can replace the phase $\varphi_-$
by $y_{2n+1}+\hat\varphi_-(x,\mathring{y})$. From now on, we assume that $\varphi_-(x,y)$ has the form 
\begin{equation}\label{e-gue170131}
\varphi_-(x,y)=y_{2n+1}+\hat\varphi_-(x,\mathring{y}).
\end{equation}
It is easy to check that $\varphi_-(x,y)$ satisfies \eqref{e-gue140205VI} and \eqref{e-gue161222} with $f(x,y)=0$. 

We now study $S^{(q)}_G(x,y)$. From Theorem~\ref{t-gue161109}, we get 

\begin{theorem}\label{t-gue170112}
Assume that $q\notin\set{n_-,n_+}$.Then, $S^{(q)}_G\equiv0$ on $X$. 
\end{theorem}

Assume that $q=n_-$ and $\Box^{(q)}_b$ has $L^2$ closed range. Fix $p\in\mu^{-1}(0)$ and let $v=(v_1,\ldots,v_d)$ and $x=(x_1,\ldots,x_{2n+1})$ be the local coordinates of $G$ and $X$ as in Theorem~\ref{t-gue161202}. Take any Haar measure $d\mu$ on $G$ and assume that $d\mu=m(v)dv=m(v_1,\ldots,v_d)dv_1\cdots dv_d$ on $V$, where $V$ is an open neighborhood of $e_0\in G$ as in Theorem~\ref{t-gue161202}. From \eqref{e-gue161231III}, we have 
\[S^{(q)}_G(x,y)=\frac{1}{\abs{G}_{d\mu}}\int_G\chi(g)S^{(q)}(x,g\circ y)d\mu(g)+\frac{1}{\abs{G}_{d\mu}}\int_G(1-\chi(g))S^{(q)}(x,g\circ y)d\mu(g),\]
where $\chi\in C^\infty_0(V)$, $\chi=1$ near $e_0$. Since $G$ is globally free on $\mu^{-1}(0)$, if $U$ and $V$ are small, there is a constant $c>0$ such that 
\begin{equation}\label{e-gue161231cr}
d(x,g\circ y)\geq c,\ \ \forall x, y\in U, g\in{\rm Supp\,}(1-\chi),
\end{equation}
where $U$ is an open set of $p\in\mu^{-1}(0)$ as in Theorem~\ref{t-gue161202}. From now on, we take $U$ and $V$ small enough so that \eqref{e-gue161231cr} holds. In view of Theorem~\ref{t-gue161109I}, we see that $\frac{1}{\abs{G}_{d\mu}}\int_G(1-\chi(g))S^{(q)}(x,g\circ y)d\mu(g)\equiv0$ on $U$ and hence 
\begin{equation}\label{e-gue170102}
\mbox{$S^{(q)}_G(x,y)\equiv\frac{1}{\abs{G}_{d\mu}}\int_G\chi(g)S^{(q)}(x,g\circ y)d\mu(g)$ on $U$}.
\end{equation}
From Theorem~\ref{t-gue161109I} and \eqref{e-gue170102}, we have 
\begin{equation}\label{e-gue170102I}
\begin{split}
&\mbox{$S^{(q)}_G(x,y)\equiv \hat S_{-}(x,y)+\hat S_{+}(x,y)$ on $U$},\\
&\hat S_{-}(x,y)=\frac{1}{\abs{G}_{d\mu}}\int_G\chi(g)S_-(x,g\circ y)d\mu(g),\\
&\hat S_{+}(x,y)=\frac{1}{\abs{G}_{d\mu}}\int_G\chi(g)S_+(x,g\circ y)d\mu(g).
\end{split}
\end{equation}
Write $x=(x',x'')=(x',\hat x'',\Td x'')$, $y=(y',y'')=(y',\hat y'',\Td y'')$, where $\hat x''=(x_{d+1},\ldots,x_{2d})$, $\hat y''=(y_{d+1},\ldots,y_{2d})$, $\Td x''=(x_{2d+1},\ldots,x_{2n+1})$, $\Td y''=(y_{2d+1},\ldots,y_{2n+1})$. Since $S^{(q)}_G(x,y)$ is $G$-invariant, we have $S^{(q)}_G(x,y)=S^{(q)}_G((0,x''),(\gamma(y''),y''))$, where $\gamma\in C^\infty(U_2,U_1)$ is as in Theorem~\ref{t-gue161202}. From this observation and \eqref{e-gue170102I}, we have 
\begin{equation}\label{e-gue170102Ib}
\mbox{$S^{(q)}_G(x,y)\equiv\hat S_{-}((0,x''),(\gamma(y''),y''))+\hat S_{+}((0,x''),(\gamma(y''),y''))$ on $U$}.
\end{equation}
 Write $\mathring{x}''=(x_{d+1},\ldots,x_{2n})$, $\mathring{y}''=(y_{d+1},\ldots,y_{2n})$
From \eqref{e-gue170131}, \eqref{e-gue170102Ib}, Theorem~\ref{t-gue161202} and Theorem~\ref{t-gue161109I}, we have 
\begin{equation}\label{e-gue170102II}
\begin{split}
&\hat S_{-}((0,x''),(\gamma(y''),y''))\\
&\equiv\frac{1}{\abs{G}_{d\mu}}\int e^{i(y_{2n+1}+\hat\varphi_-((0,x''),(v+\gamma(y''),\mathring{y}'')))t}s_-((0,x''),(v+\gamma(y''),y''),t)m(v)dvdt.
\end{split}
\end{equation}
From \eqref{e-gue161222}, it is straightforward to see that 
\begin{equation}\label{e-gue170102III}
{\rm det\,}\Bigr(\left(\frac{\pr^2\hat\varphi_-}{\pr v_k\pr v_j}(p,p)\right)^d_{j,k=1}\Bigr)=(2i)^d\abs{\mu_1}\cdots\abs{\mu_d}\neq0.
\end{equation}
We pause and introduce some notations. Let $W$ be an open set of $\Real^N$, $N\in\mathbb N$. From now on, we write $W^\Complex$ to denote an open set in $\Complex^N$ with $W^\Complex\bigcap\Real^N=W$ and for $f\in C^\infty(W)$, from now on, we write $\Td f\in C^\infty(W^\Complex)$ to denote an almost analytic extension of $f$ (see Section 2 in~\cite{MS74}). Let $h(x'',y'')\in C^\infty(U\times U,\Complex^d)$ be the solution of the system 
\begin{equation}\label{e-gue170105a}
\frac{\pr\Td{\hat\varphi_-}}{\pr\Td y_j}((0,x''),(h(x'',y'')+\gamma(y''),\mathring{y}''))=0,\ \ j=1,2,\ldots,d,
\end{equation}
and let 
\begin{equation}\label{e-gue170105aI}
\Phi_-(x'',y''):=y_{2n+1}+\Td{\hat\varphi_-}((0,x''),(h(x'',y'')+\gamma(y''),\mathring{y}'')). 
\end{equation}
It is well-known that (see page 147 in~\cite{MS74}) ${\rm Im\,}\Phi_-(x'',y'')\geq0$. Note that 
\[\frac{\pr\hat\varphi_-}{\pr v_j}|_{\hat x''=\hat y''=0, \Td x''=\Td y'', x'=v+\gamma(y'')=0}=-\langle\,\omega_0(x)\,,\,\frac{\pr}{\pr x_j}\,\rangle=0,\] 
where $x=(0,(0,\Td x''))$. We deduce that for $\hat x''=\hat y''=0$, $\Td x''=\Td y''$, $v=-\gamma(y'')$ are real critical points. From this observation, we can calculate that 
\begin{equation}\label{e-gue170109}
d_x\Phi_-|_{x''=y'', \hat x''=0}=-f(x'')\omega_0(x),\ \ d_y\Phi_-|_{x''=y'', \hat x''=0}=f(x'')\omega_0(x),
\end{equation}
where $x=(0,\Td x'')$ and $f\in C^\infty$ is a  positive function with $f(p)=1$.
By using stationary phase formula of Melin-Sj\"ostrand~\cite{MS74}, we can carry out the $v$ integral in \eqref{e-gue170102II} and get 
\begin{equation}\label{e-gue170102cw}
\hat S_{-}((0,x''),(\gamma(y''),y''))\equiv\int e^{i\Phi_-(x'',y'')t}a_-(x'',y'',t)dt\ \ \ \mbox{on $U$},
\end{equation}
where $a_-(x'',y'',t)\sim\sum^\infty_{j=0}t^{n-\frac{d}{2}-j}a^0_-(x'',y'')$ in $S^{n-\frac{d}{2}}_{1,0}(U\times U\times\Real_+, T^{*0,q}X\boxtimes(T^{*0,q}X)^*)$, 
\[a^j_-(x'',y'')\in C^\infty(U\times U,T^{*0,q}X\boxtimes(T^{*0,q}X)^*),\ \ j=0,1,\ldots,\]
\begin{equation}\label{e-gue170102cwI}
a^0_-(p,p)=\frac{m(0)}{2\abs{G}_{d\mu}}\pi^{-n-1+\frac{d}{2}}\abs{\mu_1}^{\frac{1}{2}}\cdots\abs{\mu_d}^{\frac{1}{2}}\abs{\mu_{d+1}}\cdots\abs{\mu_n}\tau_{p,n_-}.
\end{equation}
We now study the property of the phase $\Phi_-(x'',y'')$. We need the following which is well-known (see Section 2 in~\cite{MS74})

\begin{theorem}\label{t-gue170105}
There exist a constant $c>0$ and an open set $\Omega\in\Real^d$ such that  
\begin{equation}\label{e-gue160105r}
{\rm Im\,}\Phi_-(x'',y'')\geq c\inf_{v\in\Omega}\set{{\rm Im\,}\hat\varphi_-((0,x''),(v+\gamma(y''),\mathring{y}''))+\abs{d_v\hat\varphi_-((0,x''),(v+\gamma(y''),\mathring{y}''))}^2},
\end{equation}
for all $((0,x''),(0,y''))\in U\times U$.
\end{theorem}

We can now prove 

\begin{theorem}\label{t-gue170105I}
If $U$ is small enough, then there is a constant $c>0$ such that 
\begin{equation}\label{e-gue170106m}
{\rm Im\,}\Phi_-(x'',y'')\geq c\Bigr(\abs{\hat x''}^2+\abs{\hat y''}^2+\abs{\mathring{x}''-\mathring{y}''}^2\Bigr),\ \ \forall ((0,x''),(0,y''))\in U\times U.
\end{equation}
\end{theorem}

\begin{proof}
From \eqref{e-gue140205VI}, we see that there is a constant $c_1>0$ such that 
\begin{equation}\label{e-gue170106}
{\rm Im\,}\hat\varphi_-((0,x''),(v+\gamma(y''),\mathring{y}''))\geq c_1(\abs{v+\gamma(y'')}^2+\abs{\mathring{x}''-\mathring{y}''}^2),\ \ \forall v\in\Omega,
\end{equation}
where $\Omega$ is any open set of $0\in\Real^d$. From \eqref{e-gue170106} and \eqref{e-gue160105r}, we conclude that there is a constant $c_2>0$ such that 
\begin{equation}\label{e-gue170106I}
{\rm Im\,}\Phi_-(x'',y'')\geq c_2(\abs{\mathring{x}''-\mathring{y}''}^2+\abs{d_{y'}\hat\varphi_-((0,x''),(0,\mathring{x}''))}^2).
\end{equation}
From \eqref{e-gue161222}, we see that the matrix 
\[\left(\frac{\pr^2\hat\varphi_-}{\pr x_j\pr x_k}(p,p)+\frac{\pr^2\hat\varphi_-}{\pr y_j\pr y_k}(p,p)\right)_{1\leq k\leq d, d+1\leq j\leq 2d}\]
is non-singular. From this observation and notice that $d_{y'}\hat\varphi_-((0,x''),(0,\mathring{x}''))|_{\hat x''}=0$, we deduce that if $U$ is small enough then there is a constant 
$c_3>0$ such that 
\begin{equation}\label{e-gue170106II}
\abs{d_{y'}\hat\varphi_-((0,x''),(0,x''))}\geq c_3\abs{\hat x''}.
\end{equation}
From \eqref{e-gue170106II} and \eqref{e-gue170106I}, the theorem follows. 
\end{proof}

From now on, we assume that $U$ is small enough so that \eqref{e-gue170106m} holds. 

We now determine the Hessian of $\Phi_-(x'',y'')$ at $(p,p)$. Let $\hat h(x'',y''):=h(x'',y'')+\gamma(y'')$. From \eqref{e-gue170105a}, we have 
\begin{equation}\label{e-gue170106c}
\frac{\pr^2\hat\varphi_-}{\pr x_{d+1}\pr y_1}(p,p)+\sum^d_{j=1}\frac{\pr^2\hat\varphi_-}{\pr y_1\pr y_j}(p,p)\frac{\pr\hat h_j}{\pr x_{d+1}}(p,p)=0.
\end{equation}
From \eqref{e-gue161222}, we can calculate that 
\begin{equation}\label{e-gue170106cI}
\frac{\pr^2\hat\varphi_-}{\pr x_{d+1}\pr y_1}(p,p)=2\mu_1,\ \ \frac{\pr^2\hat\varphi_-}{\pr y_1\pr y_j}(p,p)=2i\abs{\mu_1}\delta_{1,j},\ \ j=1,2,\ldots,d.
\end{equation}
From \eqref{e-gue170106cI} and \eqref{e-gue170106c}, we obtain $\frac{\pr\hat h_1}{\pr x_{d+1}}(p,p)=i\frac{\mu_1}{\abs{\mu_1}}$. We can repeat the procedure above several times and deduce that 
\begin{equation}\label{e-gue170106cII}
\frac{\pr\hat h_j}{\pr x_{d+k}}(p,p)=\frac{\pr\hat h_j}{\pr y_{d+k}}(p,p)=i\frac{\mu_j}{\abs{\mu_j}}\delta_{j,k},\ \ j,k=1,2,\ldots,d. 
\end{equation}

From \eqref{e-gue170106cII}, \eqref{e-gue161222}, \eqref{e-gue170105aI} and by some straightforward computation (we omit the details), we get 

\begin{theorem}\label{t-gue170107}
With the notations used above, we have 
\begin{equation}\label{e-gue170107}
\begin{split}
\Phi_-(x'', y'')&=-x_{2n+1}+y_{2n+1}+2i\sum^d_{j=1}\abs{\mu_j}y^2_{d+j}+2i\sum^d_{j=1}\abs{\mu_j}x^2_{d+j}\\
&+i\sum^{n}_{j=d+1}\abs{\mu_j}\abs{z_j-w_j}^2 +\sum^{n}_{j=d+1}i\mu_j(\ol z_jw_j-z_j\ol w_j)\\
&+\sum^d_{j=1}(-b_{d+j}x_{d+j}x_{2n+1}+b_{d+j}y_{d+j}y_{2n+1})\\
&+\sum^n_{j=d+1}\frac{1}{2}(b_{2j-1}-ib_{2j})(-z_jx_{2n+1}+w_jy_{2n+1})\\
&+\sum^n_{j=d+1}\frac{1}{2}(b_{2j-1}+ib_{2j})(-\ol z_jx_{2n+1}+\ol w_jy_{2n+1})\\
&+(x_{2n+1}-y_{2n+1})f(x, y) +O(\abs{(x, y)}^3),
\end{split}
\end{equation}
where $z_j=x_{2j-1}+ix_{2j}$, $j=2d+1,\ldots,2n$, $\mu_j$, $j=1,\ldots,n$, and $b_{d+1}\in\Real,\ldots,b_{2n}\in\Real$ are as in \eqref{e-gue161219} and $f$ is smooth and satisfies $f(0,0)=0$, $f(x, y)=\ol f(y, x)$. 
\end{theorem}

We can change $\Phi_-(x'',y'')$ be $\Phi_-(x'',y'')\frac{1}{f(x'')}$, where $f(x'')$ is as in \eqref{e-gue170109}. Thus, 
\begin{equation}\label{e-gue170109I}
d_x\Phi_-|_{x''=y'', \hat x''=0}=-\omega_0(x),\ \ d_y\Phi_-|_{x''=y'', \hat x''=0}=\omega_0(x),
\end{equation}
where $x=(0,\Td x'')$. It is clear that $\Phi_-(x'',y'')$ still satisfies \eqref{e-gue170106m} and \eqref{e-gue170107}. 

We now determine the leading term $a_-^0(p,p)$. In view of \eqref{e-gue170102cwI}, we only need to calculate $\frac{m(0)}{\abs{G}_{d\mu}}$. It is clear that the number
$\abs{G}_{d\mu}$ is independent of the choice of Haar measure. Put $Y_p=\set{g\circ p;\, g\in G}$. $Y_p$ is a $d$-dimensional submanifold of $X$. The $G$-invariant Hermitian metric $\langle\,\cdot\,|\,\cdot\,\rangle$ induces a volume form $dv_{Y_p}$ on $Y_p$. Put 
\begin{equation}\label{e-gue170108e}
V_{{\rm eff\,}}(p):=\int_{Y_p}dv_{Y_p}.
\end{equation}
For $f(g)\in C^\infty(G)$, let $\hat f(g\circ p):=f(g)$, $\forall g\in G$. Then, $\hat f\in C^\infty(Y_p)$. Let $d\mu$ be the measure on $G$ given by $\int_Gfd\mu:=\int_{Y_p}\hat fdv_{Y_p}$, for all $f\in C^\infty(G)$. It is not difficult to see that $d\mu$ is a Haar measure and 
\begin{equation}\label{e-gue170108}
\int_Gd\mu=V_{{\rm eff\,}}(p). 
\end{equation}
Since $\set{\frac{1}{\sqrt{2}}\frac{\pr}{\pr x_1},\ldots,\frac{1}{\sqrt{2}}\frac{\pr}{\pr x_d}}$ is an orthonormal basis for $\underline{\mathfrak{g}}_p$, we have $m(0)=2^{\frac{d}{2}}$. From this observation, \eqref{e-gue170108} and \eqref{e-gue170102cwI}, we get 
\begin{equation}\label{e-gue170108I}
a^0_-(p,p)=2^{\frac{d}{2}-1}\frac{1}{V_{{\rm eff\,}}(p)}\pi^{-n-1+\frac{d}{2}}\abs{\mu_1}^{\frac{1}{2}}\cdots\abs{\mu_d}^{\frac{1}{2}}\abs{\mu_{d+1}}\cdots\abs{\mu_n}\tau_{p,n_-}.
\end{equation}

Similarly, we can repeat the procedure above and deduce that 
\begin{equation}\label{e-gue170102wc}
\hat S_{+}((0,x''),(\gamma(y''),y''))\equiv\int e^{i\Phi_+(x'',y'')t}a_-(x'',y'',t)dt\ \ \ \mbox{on $U$},
\end{equation}
where $a_+(x'',y'',t)\sim\sum^\infty_{j=0}t^{n-\frac{d}{2}-j}a^0_-(x'',y'')$ in $S^{n-\frac{d}{2}}_{1,0}(U\times U\times\Real_+, T^{*0,q}X\boxtimes(T^{*0,q}X)^*)$, 
\[a^j_+(x'',y'')\in C^\infty(U\times U,T^{*0,q}X\boxtimes(T^{*0,q}X)^*),\ \ j=0,1,\ldots,\]
\begin{equation}\label{e-gue170102wcI}
a^0_+(p,p)=2^{\frac{d}{2}-1}\frac{1}{V_{{\rm eff\,}}(p)}\pi^{-n-1+\frac{d}{2}}\abs{\mu_1}^{\frac{1}{2}}\cdots\abs{\mu_d}^{\frac{1}{2}}\abs{\mu_{d+1}}\cdots\abs{\mu_n}\tau_{p,n_+},
\end{equation}
and $\Phi_+(x'',y'')\in C^\infty(U\times U)$, ${\rm Im\,}\Phi_+(x'',y'')\geq0$, $-\ol\Phi_+(x'',y'')$ satisfies \eqref{e-gue170106m}, \eqref{e-gue170107} and \eqref{e-gue170109I}. 

Summing up, we get one of the main result of this work

\begin{theorem}\label{t-gue170108wr}
We recall that we work with the assumption that the Levi form is non-degenerate of constant signature $(n_-,n_+)$ on $X$. 
Let $q=n_-$ or $n_+$. Suppose that $\Box^{(q)}_b$ has $L^2$ closed range. Let $p\in\mu^{-1}(0)$ and let $x=(x_1,\ldots,x_{2n+1})$ be the local coordinates defined in an open set $U$ of $p$ such that $x(p)=0$ and \eqref{e-gue161202}, \eqref{e-gue161206}, \eqref{e-gue161202I}, \eqref{e-gue161219} hold. Write $x''=(x_{d+1},\ldots,x_{2n+1})$.
Then, there exist continuous operators
\[S^G_-, S^G_+:\Omega^{0,q}_0(U)\To\Omega^{0,q}(U)\]
such that 
\begin{equation}\label{e-gue170108wr}
S^{(q)}_G\equiv S^G_-+S^G_+\ \ \mbox{on $U$},
\end{equation}
\begin{equation}\label{e-gue170108wrI}
\begin{split}
&S^G_-=0\ \ \mbox{if $q\neq n_-$},\\
&S^G_+=0\ \ \mbox{if $q\neq n_+$},\\
\end{split}
\end{equation}
and if $q=n_-$, $S^G_-(x,y)$ satisfies
\begin{equation}\label{e-gue170108wrII}
S^G_-(x, y)\equiv\int^{\infty}_{0}e^{i\Phi_-(x'', y'')t}a_-(x'', y'', t)dt\ \ \mbox{on $U$}
\end{equation}
with 
\begin{equation}  \label{e-gue170108wrIII}\begin{split}
&a_-(x, y, t)\in S^{n-\frac{d}{2}}_{1,0}(U\times U\times\mathbb{R}_+,T^{*0,q}X\boxtimes(T^{*0,q}X)^*), \\
&a_-(x, y, t)\sim\sum^\infty_{j=0}a^j_-(x, y)t^{n-\frac{d}{2}-j}\quad\text{ in }S^{n-\frac{d}{2}}_{1, 0}(U\times U\times\mathbb{R}_+,T^{*0,q}X\boxtimes(T^{*0,q}X)^*),\\
&a^j_-(x, y)\in C^\infty(U\times U,T^{*0,q}X\boxtimes(T^{*0,q}X)^*),\ \ j=0,1,2,3,\ldots,\\
&a^0_-(x,x)\neq0,\ \ \forall x\in U,
\end{split}\end{equation}
$a^0_-(p,p)$ is given by \eqref{e-gue170108I} and $\Phi_-(x'',y'')\in C^\infty(U\times U)$ satisfies \eqref{e-gue170109I}, \eqref{e-gue170106m} and \eqref{e-gue170107}. 

If $q=n_+$, then $S^G_+(x,y)$ satisfies 
\begin{equation}\label{e-gue170108waI}
S^G_+(x, y)\equiv\int^{\infty}_{0}e^{i\Phi_+(x'', y'')t}a_+(x'', y'', t)dt\ \ \mbox{on $U$}
\end{equation}
with
\begin{equation}  \label{e-gue161110rIq}
\begin{split}
&a_+(x, y, t)\in S^{n-\frac{d}{2}}_{1,0}(U\times U\times\mathbb{R}_+,T^{*0,q}X\boxtimes(T^{*0,q}X)^*), \\
&a_+(x, y, t)\sim\sum^\infty_{j=0}a^j_+(x, y)t^{n-\frac{d}{2}-j}\quad\text{ in }S^{n-\frac{d}{2}}_{1, 0}(U\times U\times\mathbb{R}_+,T^{*0,q}X\boxtimes(T^{*0,q}X)^*),\\
&a^j_+(x, y)\in C^\infty(U\times U,T^{*0,q}X\boxtimes(T^{*0,q}X)^*),\ \ j=0,1,2,3,\ldots,\\
&a^0_+(x,x)\neq0,\ \ \forall x\in U,
\end{split}\end{equation}
$a^0_+(p,p)$ is given by \eqref{e-gue170102wcI} and $\Phi_+(x'',y'')\in C^\infty(U\times U)$, $-\ol\Phi_+(x'',y'')$ satisfies \eqref{e-gue170109I}, \eqref{e-gue170106m} and \eqref{e-gue170107}. 
\end{theorem}

\subsection{$G$-invariant Szeg\"o kernel asymptotics away $\mu^{-1}(0)$}\label{s-gue170110}

The goal of this section is to prove the following 

\begin{theorem}\label{t-gue170110}
Let $D$ be an open set of $X$ with $D\bigcap\mu^{-1}(0)=\emptyset$. Then, 
\[S^{(q)}_G\equiv0\ \ \mbox{on $D$}.\]
\end{theorem}

Fix a Haar measure $d\mu$ on $G$. We first need 

\begin{lemma}\label{l-gue170110}
Let $p\notin\mu^{-1}(0)$. Then, there are open sets $U$ of $p$ and $V$ of $e\in G$ such that for any $\chi\in C^\infty_0(V)$, we have 
\begin{equation}\label{e-gue170110}
\int_GS^{(q)}(x,g\circ y)\chi(g)d\mu(g)\equiv0\ \ \mbox{on $U$}.
\end{equation}
\end{lemma}

\begin{proof}
If $q\notin\set{n_-,n_+}$. By Theorem~\ref{t-gue161109}, we get \eqref{e-gue170110}. We may assume that $q=n_-$.  Take local coordinates $v=(v_1,\ldots,v_d)$ of $G$ defined in a neighborhood $V$ of $e_0$ with $v(e_0)=(0,\ldots,0)$, local coordinates $x=(x_1,\ldots,x_{2n+1})$ of $X$ defined in a neighborhood $U=U_1\times U_2$ of $p$ with $0\leftrightarrow p$, where $U_1\subset\Real^d$ is an open set of $0\in\Real^d$,  $U_2\subset\Real^{2n+1-d}$ is an open set of $0\in\Real^{2n+1-d}$, such that  
\[
\begin{split}
&(v_1,\ldots,v_d)\circ (\gamma(x_{d+1},\ldots,x_{2n+1}),x_{d+1},\ldots,x_{2n+1})\\
&=(v_1+\gamma_1(x_{d+1},\ldots,x_{2n+1}),\ldots,v_d+\gamma_d(x_{d+1},\ldots,x_{2n+1}),x_{d+1},\ldots,x_{2n+1}),\\
&\forall (v_1,\ldots,v_d)\in V,\ \ \forall (x_{d+1},\ldots,x_{2n+1})\in U_2,
\end{split}\]
and
\[
\underline{\mathfrak{g}}={\rm span\,}\set{\frac{\pr}{\pr x_1},\ldots,\frac{\pr}{\pr x_d}},
\]
where  $\gamma=(\gamma_1,\ldots,\gamma_d)\in C^\infty(U_2,U_1)$ with $\gamma(0)=0\in\Real^d$. From Theorem~\ref{t-gue161109I}, we have 
\begin{equation}\label{e-gue170102Ic}
\mbox{$\int_GS^{(q)}(x,g\circ y)\chi(g)d\mu(g)\equiv \int_GS_-(x,g\circ y)\chi(g)d\mu(g)+\int_GS_+(x,g\circ y)\chi(g)d\mu(g)$ on $U$}.
\end{equation}
From Theorem~\ref{t-gue161109I}, we have 
\begin{equation}\label{e-gue170102IIc}
\int_GS_-(x,g\circ y)\chi(g)d\mu(g)\equiv\int e^{i(\varphi_-(x,(v+\gamma(y''),y''))t}s_-(x,(v+\gamma(y''),y''),t)\chi(v)m(v)dvdt,
\end{equation}
where $y''=(y_{d+1},\ldots,y_{2n+1})$, $m(v)dv=d\mu|_V$. Since $p\notin\mu^{-1}(0)$ and notice that $d_y\varphi_-(x,x)=\omega_0(x,x)$, we deduce that if $V$ and $U$ are small then
$d_v(\varphi_-(x,(v+\gamma(y''),y'')))\neq0$, for every $v\in V$, $(x,y)\in U\times U$. Hence, by using integration by parts with respect to $v$, we get
\begin{equation}\label{e-gue170110r}
 \int_GS_-(x,g\circ y)\chi(g)d\mu(g)\equiv0\ \ \mbox{on $U$}.
 \end{equation}
 Similarly, we have 
 \begin{equation}\label{e-gue170110rI}
\int_GS_+(x,g\circ y)\chi(g)d\mu(g)\equiv0\ \ \mbox{on $U$}.
\end{equation}
From \eqref{e-gue170102Ic}, \eqref{e-gue170110r} and \eqref{e-gue170110rI}, the lemma follows. 
\end{proof}

\begin{lemma}\label{l-gue170111}
Let $p\notin\mu^{-1}(0)$ and let $h\in G$. We can find open sets $U$ of $p$ and $V$ of $h$ such that for every $\chi\in C^\infty_0(V)$, we have 
\[\int_GS^{(q)}(x,g\circ y)\chi(g)d\mu(g)\equiv0\ \ \mbox{on $U$}.\]
\end{lemma}

\begin{proof}
Let $U$ and $V$ be open sets as in Lemma~\ref{l-gue170110}. Let $\hat V=hV$. Then, $\hat V$ is an open set of $G$. Let $\hat\chi\in C^\infty_0(\hat V)$. We have 
\begin{equation}\label{e-gue170111}
\int_GS^{(q)}(x,g\circ y)\hat\chi(g)d\mu(g)=\int_GS^{(q)}(x,h\circ g\circ y)\hat\chi(h\circ g)d\mu(g)=\int_GS^{(q)}(x,g\circ y)\chi(g)d\mu(g),
\end{equation}
where $\chi(g):=\hat\chi(h\circ g)\in C^\infty_0(V)$. From \eqref{e-gue170111} and Lemma~\ref{l-gue170110}, we deduce that 
\[\int_GS^{(q)}(x,g\circ y)\hat\chi(g)d\mu(g)\equiv0\ \ \mbox{on $U$}.\]
The lemma follows. 
\end{proof}

\begin{proof}[Proof of Theorem~\ref{t-gue170110}]
Fix $p\in D$. We need to show that $S^{(q)}_G$ is smoothing near $p$. 
Let $h\in G$. By Lemma~\ref{l-gue170111}, we can find open sets $U_h$ of $p$ and $V_h$ of $h$ such that for every $\chi\in C^\infty_0(V_h)$, we have 
\begin{equation}\label{e-gue170111c}
\int_GS^{(q)}(x,g\circ y)\chi(g)d\mu(g)\equiv0\ \ \mbox{on $U_h$}.
\end{equation}
Since $G$ is compact, we can find open sets $U_{h_j}$ and $V_{h_j}$, $j=1,\ldots,N$, such that $G=\bigcup^N_{j=1}V_{h_j}$. Let $U=D\bigcap\Bigr(\bigcap^N_{j=1}U_{h_j}\Bigr)$ and let $\chi_j\in C^\infty_0(V_{h_j})$, $j=1,\ldots,N$, with $\sum^N_{j=1}\chi_j=1$ on $G$. From \eqref{e-gue170111c}, we have 
\begin{equation}\label{e-gue170111cI}
S^{(q)}_G(x,y)=\frac{1}{\abs{G}_{d\mu}}\int_GS^{(q)}(x,g\circ y)d\mu(g)=\frac{1}{\abs{G}_{d\mu}}\sum^N_{j=1}\int_GS^{(q)}(x,g\circ y)\chi_j(g)d\mu(g)\equiv0\ \ \mbox{on $U$}.
\end{equation}
The theorem follows. 
\end{proof}

From Theorem~\ref{t-gue170112}, Theorem~\ref{t-gue170108wr} and Theorem~\ref{t-gue170110}, we get Theorem~\ref{t-gue170124}. 

\section{$G$-invariant Szeg\"o kernel asymptotics on CR manifolds wit $S^1$ action}\label{s-gue170111}

Let $(X, T^{1,0}X)$ be a compact CR manifold of dimension $2n+1$, $n\geq 1$. We assume that $X$ admits an $S^1$ action: $S^1\times X\rightarrow X$. We write $e^{i\theta}$ to denote the $S^1$ action. Let $T\in C^\infty(X, TX)$ be the global real vector field induced by the $S^1$ action given by
$(Tu)(x)=\frac{\partial}{\partial\theta}\left(u(e^{i\theta}\circ x)\right)|_{\theta=0}$, $u\in C^\infty(X)$. We recall

\begin{definition}\label{d-gue160502}
We say that the $S^1$ action $e^{i\theta}$ is CR if
$[T, C^\infty(X, T^{1,0}X)]\subset C^\infty(X, T^{1,0}X)$ and the $S^1$ action is transversal if for each $x\in X$,
$\Complex T(x)\oplus T_x^{1,0}X\oplus T_x^{0,1}X=\mathbb CT_xX$. Moreover, we say that the $S^1$ action is locally free if $T\neq0$ everywhere. It should be mentioned that transversality implies locally free.
\end{definition}

We assume now that $(X, T^{1,0}X)$ is a compact connected CR manifold with a transversal CR locally free $S^1$ action $e^{i\theta}$ and we let $T$ be the global vector field induced by the $S^1$ action. Let $\omega_0\in C^\infty(X,T^*X)$ be the global real one form determined by $\langle\,\omega_0\,,\,u\,\rangle=0$, for every $u\in T^{1,0}X\oplus T^{0,1}X$, and $\langle\,\omega_0\,,\,T\,\rangle=-1$. Note that $\omega_0$ and $T$ satisfy \eqref{e-gue170111ry}. Assume that $X$ admits a compact connected Lie group action $G$ and the Lie group $G$ acts on $X$ preserving $\omega_0$ and $J$. We recall that we work with Assumption~\ref{a-gue170128}. 

Fix $\theta_0\in]-\pi, \pi[$, $\theta_0$ small. Let
$$d e^{i\theta_0}: \mathbb CT_x X\rightarrow \mathbb CT_{e^{i\theta_0}x}X$$
denote the differential map of $e^{i\theta_0}: X\rightarrow X$. By the CR property of the $S^1$ action, we can check that
\begin{equation}\label{e-gue150508fa}
\begin{split}
de^{i\theta_0}:T_x^{1,0}X\rightarrow T^{1,0}_{e^{i\theta_0}x}X,\\
de^{i\theta_0}:T_x^{0,1}X\rightarrow T^{0,1}_{e^{i\theta_0}x}X,\\
de^{i\theta_0}(T(x))=T(e^{i\theta_0}x).
\end{split}
\end{equation}
Let $(e^{i\theta_0})^*:\Lambda^r(\Complex T^*X)\To\Lambda^r(\Complex T^*X)$ be the pull-back map by $e^{i\theta_0}$, $r=0,1,\ldots,2n+1$. From \eqref{e-gue150508fa}, it is easy to see that, for every $q=0,1,\ldots,n$, one has
\begin{equation}\label{e-gue150508faI}
(e^{i\theta_0})^*:T^{*0,q}_{e^{i\theta_0}x}X\To T^{*0,q}_{x}X.
\end{equation}
Let $u\in\Omega^{0,q}(X)$ be arbitrary. Define
\begin{equation}\label{e-gue150508faII}
Tu:=\frac{\pr}{\pr\theta}\bigr((e^{i\theta})^*u\bigr)|_{\theta=0}\in\Omega^{0,q}(X).
\end{equation}
From \eqref{e-gue170111ryII}, it is clear that
\begin{equation}\label{e-gue170111f}
Tg^*u=g^*Tu,\ \ \forall g\in G,\ \ \forall u\in\Omega^{0,q}(X).
\end{equation}
For every $\theta\in\Real$ and every $u\in C^\infty(X,\Lambda^r(\Complex T^*X))$, we write $u(e^{i\theta}\circ x):=(e^{i\theta})^*u(x)$. It is clear that, for every $u\in C^\infty(X,\Lambda^r(\Complex T^*X))$, we have
\begin{equation}\label{e-gue150510f}
u(x)=\sum_{m\in\mathbb Z}\frac{1}{2\pi}\int^{\pi}_{-\pi}u(e^{i\theta}\circ x)e^{-im\theta}d\theta.
\end{equation}

For every $m\in\mathbb Z$, let
\begin{equation}\label{e-gue150508dI}
\begin{split}
&\Omega^{0,q}_m(X):=\set{u\in\Omega^{0,q}(X);\, Tu=imu},\ \ q=0,1,2,\ldots,n,\\
&\Omega^{0,q}_{m}(X)^G:=\set{u\in\Omega^{0,q}(X)^G;\, Tu=imu},\ \ q=0,1,2,\ldots,n.
\end{split}
\end{equation}
We denote $C^\infty_m(X):=\Omega^{0,0}_m(X)$, $C^\infty_m(X)^G:=\Omega^{0,0}_m(X)^G$. From the CR property of the $S^1$ action and \eqref{e-gue170111ryII}, it is not difficult to see that 
\[Tg^*\ddbar_b=g^*T\ddbar_b=\ddbar_bg^*T=\ddbar_bTg^*\ \ \mbox{on $\Omega^{0,q}(X)$},\ \ \forall g\in G.\]
Hence,
\begin{equation}\label{e-gue160527}
\ddbar_b:\Omega^{0,q}_m(X)^G\To\Omega^{0,q+1}_m(X)^G,\ \ \forall m\in\mathbb Z.
\end{equation}
We now assume that the Hermitian metric $\langle\,\cdot\,|\,\cdot\,\rangle$ on $\Complex TX$ is $G\times S^1$ invariant.  Then the $L^2$ inner product $(\,\cdot\,|\,\cdot\,)$ on $\Omega^{0,q}(X)$ 
induced by $\langle\,\cdot\,|\,\cdot\,\rangle$ is $G\times S^1$-invariant. We then have 
\[\begin{split}
&Tg^*\ol{\pr}^*_b=g^*T\ol{\pr}^*_b=\ol{\pr}^*_bg^*T=\ol{\pr}^*_bTg^*\ \ \mbox{on $\Omega^{0,q}(X)$},\ \ \forall g\in G,\\
&Tg^*\Box^{(q)}_b=g^*T\Box^{(q)}_b=\Box^{(q)}_bg^*T=\Box^{(q)}_bTg^*\ \ \mbox{on $\Omega^{0,q}(X)$},\ \ \forall g\in G.
\end{split}\]

Let $L^2_{(0,q),m}(X)^G$ be
the completion of $\Omega^{0,q}_m(X)^G$ with respect to $(\,\cdot\,|\,\cdot\,)$. 
We write $L^2_m(X)^G:=L^2_{(0,0),m}(X)^G$. Put 
\[({\rm Ker\,}\Box^{(q)}_b)^G_m:=({\rm Ker\,}\Box^{(q)}_b)^G\bigcap L^2_{(0,q),m}(X)^G.\]
It is not difficult to see that, for every $m\in\mathbb Z$, $({\rm Ker\,}\Box^{(q)}_b)^G_m\subset\Omega^{0,q}_m(X)^G$ and ${\rm dim\,}({\rm Ker\,}\Box^{(q)}_b)^G_m<\infty$.
The $m$-th $G$-invariant Szeg\"o projection is the orthogonal projection 
\[S^{(q)}_{G,m}:L^2_{(0,q)}(X)\To ({\rm Ker\,}\Box^{(q)}_b)^G_m\]
with respect to $(\,\cdot\,|\,\cdot\,)$. Let $S^{(q)}_{G,m}(x,y)\in C^\infty(X\times X,T^{*0,q}X\boxtimes(T^{*0,q}X)^*)$ be the distribution kernel of $S^{(q)}_{G,m}$. We can check that 
\begin{equation}\label{e-gue170111cr}
S^{(q)}_{G,m}(x,y)=\frac{1}{2\pi}\int^{\pi}_{-\pi}S^{(q)}_G(x,e^{i\theta}\circ y)e^{im\theta}d\theta.
\end{equation}
The goal of this section is to study the asymptotics of $S^{(q)}_{G,m}$ as $m\To+\infty$. 

From Theorem~\ref{t-gue170110}, \eqref{e-gue170111cr} and by using integration by parts several times, we get 

\begin{theorem}\label{t-gue170111w}
Let $D\subset X$ be an open set with $D\bigcap\mu^{-1}(0)=\emptyset$. Then, 
\[S^{(q)}_{G,m}=O(m^{-\infty})\ \ \mbox{on $D$}. \]
\end{theorem}

We now study $S^{(q)}_{G,m}$ near $\mu^{-1}(0)$. We can repeat the proof of Theorem~\ref{t-gue161202} with minor change and get 

\begin{theorem}\label{t-gue161202z}
Let $p\in\mu^{-1}(0)$. There exist local coordinates $v=(v_1,\ldots,v_d)$ of $G$ defined in  a neighborhood $V$ of $e_0$ with $v(e_0)=(0,\ldots,0)$, local coordinates $x=(x_1,\ldots,x_{2n+1})$ of $X$ defined in a neighborhood $U=U_1\times(\hat U_2\times]-2\delta,2\delta[)$ of $p$ with $0\leftrightarrow p$, where $U_1\subset\Real^d$ is an open set of $0\in\Real^d$,  $\hat U_2\subset\Real^{2n-d}$ is an open set of $0\in\Real^{2n-d} $, $\delta>0$, and a smooth function $\gamma=(\gamma_1,\ldots,\gamma_d)\in C^\infty(\hat U_2\times]-2\delta,2\delta[,U_1)$ with $\gamma(0)=0\in\Real^d$  such that
\begin{equation}\label{e-gue161202z}
\begin{split}
&(v_1,\ldots,v_d)\circ (\gamma(x_{d+1},\ldots,x_{2n+1}),x_{d+1},\ldots,x_{2n+1})\\
&=(v_1+\gamma_1(x_{d+1},\ldots,x_{2n+1}),\ldots,v_d+\gamma_d(x_{d+1},\ldots,x_{2n+1}),x_{d+1},\ldots,x_{2n+1}),\\
&\forall (v_1,\ldots,v_d)\in V,\ \ \forall (x_{d+1},\ldots,x_{2n+1})\in\hat U_2\times]-2\delta,2\delta[,
\end{split}
\end{equation}
\begin{equation}\label{e-gue161206z}
\begin{split}
&T=-\frac{\pr}{\pr x_{2n+1}},\\
&\underline{\mathfrak{g}}={\rm span\,}\set{\frac{\pr}{\pr x_1},\ldots,\frac{\pr}{\pr x_d}},\\
&\mu^{-1}(0)\bigcap U=\set{x_{d+1}=\cdots=x_{2d}=0},\\
&\mbox{On $\mu^{-1}(0)\bigcap U$, we have $J(\frac{\pr}{\pr x_j})=\frac{\pr}{\pr x_{d+j}}+a_j(x)\frac{\pr}{\pr x_{2n+1}}$, $j=1,2,\ldots,d$}, 
\end{split}
\end{equation}
where $a_j(x)$ is a smooth function on $\mu^{-1}(0)\bigcap U$, independent of $x_1,\ldots,x_{2d}$, $x_{2n+1}$ and $a_j(0)=0$, $j=1,\ldots,d$, 
\begin{equation}\label{e-gue161202Iz}
\begin{split}
&T^{1,0}_pX={\rm span\,}\set{Z_1,\ldots,Z_n},\\
&Z_j=\frac{1}{2}(\frac{\pr}{\pr x_j}-i\frac{\pr}{\pr x_{d+j}})(p),\ \ j=1,\ldots,d,\\
&Z_j=\frac{1}{2}(\frac{\pr}{\pr x_{2j-1}}-i\frac{\pr}{\pr x_{2j}})(p),\ \ j=d+1,\ldots,n,\\
&\langle\,Z_j\,|\,Z_k\,\rangle=\delta_{j,k},\ \ j,k=1,2,\ldots,n,\\
&\mathcal{L}_p(Z_j, \ol Z_k)=\mu_j\delta_{j,k},\ \ j,k=1,2,\ldots,n
\end{split}
\end{equation}
and 
\begin{equation}\label{e-gue161219z}
\begin{split}
\omega_0(x)&=(1+O(\abs{x}))dx_{2n+1}+\sum^d_{j=1}4\mu_jx_{d+j}dx_j\\
&\quad+\sum^n_{j=d+1}2\mu_jx_{2j}dx_{2j-1}-\sum^n_{j=d+1}2\mu_jx_{2j-1}dx_{2j}+O(\abs{x}^2).
\end{split}
\end{equation}
\end{theorem}

\begin{remark}\label{r-gue170309}
Let $p\in\mu^{-1}(0)$ and let $x=(x_1,\ldots,x_{2n+1})$ be the local coordinates as in Theorem~\ref{t-gue161202z}. We can change $x_{2n+1}$ be $x_{2n+1}-\sum^d_{j=1}a_j(x)x_{d+j}$, where $a_j(x)$, $j=1,\ldots,d$, are as in \eqref{e-gue161206z}. With this new local coordinates $x=(x_1,\ldots,x_{2n+1})$, on $\mu^{-1}(0)\bigcap U$, we have 
\[J(\frac{\pr}{\pr x_j})=\frac{\pr}{\pr x_{d+j}},\ \ j=1,2,\ldots,d.\]
Moreover, it is clear that $\Phi_-(x,y)+\sum^d_{j=1}a_j(x)x_{d+j}-\sum^{d=1}_{j=1}a_j(y)y_{d+j}$ satisfies \eqref{e-gue170126I}. Note that $a_j(x)$ is a smooth function on $\mu^{-1}(0)\bigcap U$, independent of $x_1,\ldots,x_{2d}$, $x_{2n+1}$ and $a_j(0)=0$, $j=1,\ldots,d$. 
\end{remark}

We now work with local coordinates as in Theorem~\ref{t-gue161202z}. 
From \eqref{e-gue170107}, we see that near $(p,p)\in U\times U$, we have $\frac{\pr\Phi_-}{\pr y_{2n+1}}\neq0$.
Using the Malgrange preparation theorem \cite[Th.\,7.5.7]{Hor03}, we have
\begin{equation} \label{e-gue170105z}
\Phi_-(x,y)=g(x,y)(y_{2n+1}+\hat\Phi_-(x'',\mathring{y}''))
\end{equation}
in some neighborhood of $(p,p)$, where $\mathring{y}''=(y_{d+1},\ldots,y_{2n})$, $g, \hat\Phi_-\in C^\infty$. Since ${\rm
Im\,}\Phi_-\geq0$, it is not difficult to see that
${\rm Im\,}\hat\Phi_-\geq0$ in some neighborhood of $(p,p)$. We may take $U$
small enough so that \eqref{e-gue170105z} holds and ${\rm
Im\,}\hat\Phi_-\geq0$ on $U\times U$.  From the global theory of Fourier integral
operators \cite[Th.\,4.2]{MS74}, we see that
$\Phi_-(x,y)t$ and $(y_{2n+1}+\hat\Phi_-(x'',\mathring{y}''))t$ are equivalent in the
sense of Melin-Sj\"{o}strand. We can replace the phase $\Phi_-$
by $y_{2n+1}+\hat\Phi_-(x,\mathring{y''})$. From now on, we assume that 
\begin{equation}\label{e-gue170117eI}
\Phi_-(x,y)=y_{2n+1}+\hat\Phi_-(x'',\mathring{y}'').
\end{equation}
It is easy to check that $\Phi_-(x,y)$ satisfies \eqref{e-gue170106m} and \eqref{e-gue170107} with $f(x,y)=0$. Similarly, from now on, we assume that 
\begin{equation}\label{e-gue170117e}
\Phi_+(x,y)=-y_{2n+1}+\hat\Phi_+(x'',\mathring{y}'').
\end{equation}

We now study $S^{(q)}_{G,m}(x,y)$. From Theorem~\ref{t-gue170112}, we get 

\begin{theorem}\label{t-gue170112I}
Assume that $q\notin\set{n_-,n_+}$. Then, $S^{(q)}_{G,m}=O(m^{-\infty})$ on $X$. 
\end{theorem}

Assume that $q=n_-$. It is well-known that when $X$ admits a transversal $S^1$ action, then  $\Box^{(q)}_b$ has $L^2$ closed range (see Theorem 1.12 in~\cite{HM14}). Fix $p\in\mu^{-1}(0)$ and let $v=(v_1,\ldots,v_d)$ and $x=(x_1,\ldots,x_{2n+1})$ be the local coordinates of $G$ and $X$ as in Theorem~\ref{t-gue161202z} and let $U$ and $V$ be open sets as in Theorem~\ref{t-gue161202z}. We take $U$ small enough so that there is a constant $c>0$ such that 
\begin{equation}\label{e-gue170117tI}
d(e^{i\theta}\circ g\circ x,y)\geq c,\ \ \forall (x,y)\in U\times U,\ \ \forall g\in G, \theta\in[-\pi,-\delta]\bigcup[\delta,\pi],
\end{equation}
where $\delta>0$ is as in Theorem~\ref{t-gue161202z}. 
We will study $S^{(q)}_{G,m}(x,y)$ in $U$. From \eqref{e-gue170111cr}, we have 
\begin{equation}\label{e-gue170117}
\begin{split}
S^{(q)}_{G,m}(x,y)&=\frac{1}{2\pi}\int^{\pi}_{-\pi}S^{(q)}_G(x,e^{i\theta}\circ y)e^{im\theta}d\theta=\frac{1}{2\pi}\int^{\pi}_{-\pi}e^{-imx_{2n+1}+imy_{2n+1}}S^{(q)}_G(\mathring{x},e^{i\theta}\circ\mathring{y})e^{im\theta}d\theta\\
&=I+II,\\
&I=\frac{1}{2\pi}\int^{\pi}_{-\pi}e^{-imx_{2n+1}+imy_{2n+1}}\chi(\theta)S^{(q)}_G(\mathring{x},e^{i\theta}\circ\mathring{y})e^{im\theta}d\theta,\\
&II=\frac{1}{2\pi}\int^{\pi}_{-\pi}e^{-imx_{2n+1}+imy_{2n+1}}(1-\chi(\theta))S^{(q)}_G(\mathring{x},e^{i\theta}\circ\mathring{y})e^{im\theta}d\theta,
\end{split}
\end{equation}
where $\mathring{x}=(x_1,\ldots,x_{2n},0)\in U$, $\mathring{y}=(y_1,\ldots,y_{2n},0)\in U$, $\chi\in C^\infty_0(]-2\delta,2\delta[)$, $\chi=1$ on $[-\delta, \delta]$. We first study $II$. We have 
\begin{equation}\label{e-gue170117p}
II=\frac{1}{2\pi\abs{G}_{d\mu}}\int^{\pi}_{-\pi}\int_Ge^{-imx_{2n+1}+imy_{2n+1}}(1-\chi(\theta))S^{(q)}(\mathring{x},e^{i\theta}\circ g\circ\mathring{y})e^{im\theta}d\mu(g)d\theta,
\end{equation}
where $d\mu$ is a Haar measure on $G$. From \eqref{e-gue170117p}, \eqref{e-gue170117tI} and notice that $S^{(q)}$ is smoothing away diagonal, we deduce that 
\begin{equation}\label{e-gue170117pI}
II=O(m^{-\infty}). 
\end{equation}
We now study $I$. From Theorem~\ref{t-gue170108wr}, \eqref{e-gue170111cr}, \eqref{e-gue170117eI} and \eqref{e-gue170117e}, we have 
\begin{equation}\label{e-gue170117pII}
\begin{split}
&I=I_0+I_1,\\
&I_0=\frac{1}{2\pi}\int^\infty_0\int^{\pi}_{-\pi}e^{-imx_{2n+1}+imy_{2n+1}}\chi(\theta)e^{i(-\theta+\hat\Phi_-(\mathring{x}'',\mathring{y}''))t+im\theta}a_-(\mathring{x}'',
(\mathring{y}'',-\theta),t)dtd\theta,\\
&I_1=\frac{1}{2\pi}\int^\infty_0\int^{\pi}_{-\pi}e^{-imx_{2n+1}+imy_{2n+1}}\chi(\theta)e^{i(\theta+\hat\Phi_+(\mathring{x}'',\mathring{y}''))t+im\theta}a_+(\mathring{x}'',
(\mathring{y}'',-\theta),t)dtd\theta.
\end{split}
\end{equation}
We first study $I_1$. From $\frac{\pr}{\pr\theta}\Bigr(i(\theta+\hat\Phi_+(\mathring{x}'',\mathring{y}''))t+im\theta\Bigr)\neq0$, we can integrate by parts with respect to $\theta$ several times and deduce that 
\begin{equation}\label{e-gue170117pIII}
I_1=O(m^{-\infty}).
\end{equation}
We now study $I_0$. We have 
\begin{equation}\label{e-gue170117pIV}
I_0=\frac{1}{2\pi}\int^\infty_0\int^{\pi}_{-\pi}e^{-imx_{2n+1}+imy_{2n+1}}\chi(\theta)e^{im(-\theta t+\hat\Phi_-(\mathring{x}'',\mathring{y}'')t+\theta)}ma_-(\mathring{x}'',
(\mathring{y}'',-\theta),mt)dtd\theta.
\end{equation}
We can use the complex stationary phase formula of Melin-Sj\"ostrand~\cite{MS74} to carry the $dtd\theta$ integration in \eqref{e-gue170117pIV} and get (the calculation is similar as in the proof of Theorem 3.17 in~\cite{HM14a}, we omit the details)
\begin{equation}\label{e-gue170117pV}
\begin{split}
&I_0=e^{im\Psi(x,y)}b(\mathring{x}'',\mathring{y}'',m)+O(m^{-\infty}),\\
&\Psi(x,y)=\hat\Phi_-(\mathring{x}'',\mathring{y}'')-x_{2n+1}+y_{2n+1},\\
&b(\mathring{x}'',\mathring{y}'',m)\in S^{n-\frac{d}{2}}_{{\rm loc\,}}(1; U\times U, T^{*0,q}X\boxtimes(T^{*0,q}X)^*),\\
&\mbox{$b(\mathring{x}'',\mathring{y}'',m)\sim\sum^\infty_{j=0}m^{n-\frac{d}{2}-j}b_j(\mathring{x}'',\mathring{y}'')$ in $S^{n-\frac{d}{2}}_{{\rm loc\,}}(1; U\times U, T^{*0,q}X\boxtimes(T^{*0,q}X)^*)$},\\
&b_j(\mathring{x}'',\mathring{y}'')\in C^\infty(U\times U, T^{*0,q}X\boxtimes(T^{*0,q}X)^*),\ \ j=0,1,2,\ldots,
\end{split}
\end{equation}
and 
\begin{equation}\label{e-gue170117pVI}
b_0(p,p)=a^0_-(p,p)=2^{\frac{d}{2}-1}\frac{1}{V_{{\rm eff\,}}(p)}\pi^{-n-1+\frac{d}{2}}\abs{\mu_1}^{\frac{1}{2}}\cdots\abs{\mu_d}^{\frac{1}{2}}\abs{\mu_{d+1}}\cdots\abs{\mu_n}\tau_{p,n_-}.
\end{equation}

Assume that $q=n_+\neq n_-$. We can repeat the method above with minor change and deduce that $S^{(q)}_{G,m}(x,y)=O(m^{-\infty})$ on $X$. 
Summing up, we get one of the main result of this work 

\begin{theorem}\label{t-gue170118}
Recall that we work with the assumptions that the Levi form is non-degenerate of constant signature $(n_-,n_+)$ on $X$ and $G\times S^1$ acts globally free near $\mu^{-1}(0)$. If $q\neq n_-$, then, as $m\To+\infty$, $S^{(q)}_{G,m}(x,y)=O(m^{-\infty})$ on $X$. 

Suppose that $q=n_-$. Let $D\subset X$ be an open set with $D\bigcap\mu^{-1}(0)=\emptyset$. Then, as $m\To+\infty$, 
\[S^{(q)}_{G,m}=O(m^{-\infty})\ \ \mbox{on $D$}. \]
Let $p\in\mu^{-1}(0)$ and let $U$ be a small open set of $p$ with local coordinates $x=(x_1,\ldots,x_{2n+1})$. Then, 
\begin{equation}\label{e-gue170117pVIII}
\begin{split}
&S^{(q)}_{G,m}(x,y)=e^{im\Psi(x,y)}b(x,y,m)+O(m^{-\infty}),\\
&b(x,y,m)\in S^{n-\frac{d}{2}}_{{\rm loc\,}}(1; U\times U, T^{*0,q}X\boxtimes(T^{*0,q}X)^*),\\
&\mbox{$b(x,y,m)\sim\sum^\infty_{j=0}m^{n-\frac{d}{2}-j}b_j(x,y)$ in $S^{n-\frac{d}{2}}_{{\rm loc\,}}(1; U\times U, T^{*0,q}X\boxtimes(T^{*0,q}X)^*)$},\\
&b_j(x,y)\in C^\infty(U\times U, T^{*0,q}X\boxtimes(T^{*0,q}X)^*),\ \ j=0,1,2,\ldots,\\
\end{split}
\end{equation}
\begin{equation}\label{e-gue170117pVIIIa}
b_0(p,p)=2^{\frac{d}{2}-1}\frac{1}{V_{{\rm eff\,}}(p)}\pi^{-n-1+\frac{d}{2}}\abs{\mu_1}^{\frac{1}{2}}\cdots\abs{\mu_d}^{\frac{1}{2}}\abs{\mu_{d+1}}\cdots\abs{\mu_n}\tau_{p,n_-},
\end{equation}
$\Psi(x,y)\in C^\infty(U\times U)$ is as in \eqref{e-gue170117pV}, $d_x\Psi(x,x)=-d_y\Psi(x,x)=-\omega_0(x)$, for every $x\in\mu^{-1}(0)$, $\Psi(x,x)=0$ if $x\in\mu^{-1}(0)$ and there is a constant $c>0$ such that for all $(x,y)\in U\times U$, 
\begin{equation}\label{e-gue170117pVIIIb}
{\rm Im\,}\Psi(x,y)\geq c\Bigr(d(x,\mu^{-1}(0))^2+d(y,\mu^{-1}(0))^2+\inf_{g\in G,\theta\in S^1}d(e^{i\theta}\circ g\circ x,y)^2\Bigr).
\end{equation}
\end{theorem}

\section{Equivalent of the phase function $\Phi_-(x,y)$}\label{s-gue170120}

Let $p\in\mu^{-1}(0)$ and let $U$ be a small open set of $p$. We need 

\begin{definition}\label{d-gue140305}
With the assumptions and notations used in Theorem~\ref{t-gue170108wr}, 
let $\Phi_1, \Phi_2\in C^\infty(U\times U)$. We assume that $\Phi_1$ and $\Phi_2$ satisfy \eqref{e-gue170109I}, \eqref{e-gue170107} 
and \eqref{e-gue170106m}. 
We say that $\Phi_1$ and $\Phi_2$ are equivalent on $U$ if for any
$b_1(x,y,t)\in  S^{n-\frac{d}{2}}_{{\rm cl\,}}\big(U\times U\times\mathbb{R}_+,T^{*0,q}X\boxtimes(T^{*0,q}X)^*\big)$
we can find
$b_2(x,y,t)\in  S^{n-\frac{d}{2}}_{{\rm cl\,}}\big(U\times U\times\mathbb{R}_+,T^{*0,q}X\boxtimes(T^{*0,q}X)^*\big)$
such that
\[\int^\infty_0e^{i\Phi_1(x,y)t}b_1(x,y,t)dt\equiv\int^\infty_0e^{i\Phi_2(x,y)t}b_2(x,y,t)dt\ \ \mbox{on $U$}\]
and vise versa.
\end{definition}
We characterize now the phase $\Phi_-$. 
\begin{theorem} \label{t-gue140305II}
Let $\Phi_-(x,y)\in C^\infty(U\times U)$ be as in Theorem~\ref{t-gue170108wr}. Let $\Phi\in C^\infty(U\times U)$.
We assume that $\Phi$ satisfies \eqref{e-gue170109I}, \eqref{e-gue170107} 
and \eqref{e-gue170106m}. The functions
$\Phi$ and $\Phi_-$ are equivalent on $U$ in the sense of Definition~\ref{d-gue140305}
if and only if there is a function $f\in C^\infty(U\times U)$ such that $\Phi(x,y)-f(x,y)\Phi_-(x,y)$
vanishes to infinite order at ${\rm diag\,}\Bigr((\mu^{-1}(0)\bigcap U)\times(\mu^{-1}(0)\bigcap U)\Bigr)$.
\end{theorem}

\begin{proof}
The "$\Leftarrow$" part follows from global theory of complex Fourier integral operator of Melin-Sj\"ostrand~\cite{MS74}. We only need to prove the "$\Rightarrow$" part. 
Take local coordinates $x=(x_1,\ldots,x_{2n+1})$
defined in some small neighbourhood of $p$ such that $x(p)=0$ and $\omega_0(p)=dx_{2n+1}$. 
Since $d_y\Phi(x, y)|_{x=y\in\mu^{-1}(0)}=d_y\Phi_-(x, y)|_{x=y\in\mu^{-1}(0)}=\omega_0(x)$, we have 
$\frac{\pr\Phi}{\pr y_{2n+1}}(p,p)=\frac{\pr\Phi_-}{\pr y_{2n+1}}(p,p)=1$. 
From this observation and the Malgrange preparation theorem \cite[Theorem 7.5.7]{Hor03}, 
we conclude that in some small neighborhood of $(p,p)$, we can find $f(x,y), f_1(x,y)\in C^\infty$
such that
\begin{equation}\label{e-gue140214I}
\begin{split}
&\Phi_-(x,y)=f(x,y)(y_{2n+1}+h(x,\mathring{y})),\\
&\Phi(x,y)=f_1(x,y)(y_{2n+1}+h_1(x,\mathring{y}))
\end{split}
\end{equation}
in some small neighborhood of $(p,p)$, where $\mathring{y}=(y_1,\ldots,y_{2n})$. 
For simplicity, we assume that \eqref{e-gue140214I} hold on $U\times U$. 
It is clear that $\Phi_-(x,y)$ and $y_{2n-1}+h(x,\mathring{y})$ are equivalent 
in the sense of Definition~\ref{d-gue140305}, $\Phi(x,y)$ and $y_{2n+1}+h_1(x,\mathring{y})$ 
are equivalent in the sense of Definition~\ref{d-gue140305}, we may assume that $\Phi_-(x,y)=y_{2n+1}+h(x,\mathring{y})$ and $\Phi(x,y)=y_{2n+1}+h_1(x,\mathring{y})$. Fix $x_0\in\mu^{-1}(0)\bigcap U$. 
We are going to prove that $h(x,\mathring{y})-h_1(x,\mathring{y})$ vanishes to infinite order at $(x_0,x_0)\in(\mu^{-1}(0)\bigcap U)\times(\mu^{-1}(0)\bigcap U)$. Take 
\[b(x,y,t)\sim\sum^\infty_{j=0}b_j(x,y)t^{n-\frac{d}{2}-j}\in  S^{n-\frac{d}{2}}_{{\rm cl\,}}\big(U\times U\times\mathbb{R}_+,T^{*0,q}X\boxtimes(T^{*0,q}X)^*\big)\]
with $b_0(x,x)\neq0$ at each $x\in U\bigcap\mu^{-1}(0)$. Since $\Phi$ and $\Phi_-$ are equivalent on $U$ in the sense of Definition~\ref{d-gue140305}, 
we can find
$a(x,y,t)\in  S^{n-\frac{d}{2}}_{{\rm cl\,}}\big(U\times U\times\mathbb{R}_+,T^{*0,q}X\boxtimes(T^{*0,q}X)^*\big)$
such that
\[\int^\infty_0e^{i\Phi_-(x,y)t}b(x,y,t)dt\equiv\int^\infty_0e^{i\Phi(x,y)t}a(x,y,t)dt\ \ \mbox{on $U$}.\]
Put
\[x_0=(x^1_0,x^2_0,\ldots,x^{2n+1}_0),\ \ \mathring{x}_0=(x^1_0,\ldots,x^{2n}_0).\]
Take $\tau\in C^\infty_0(\Real^{2n+1})$, $\tau_1\in C^\infty_0(\Real^{2n})$, $\chi\in C^\infty_0(\Real)$ 
so that $\tau=1$ near $x_0$, $\tau_1=1$ near $\mathring{x}_0$,  $\chi=1$ near $x^{2n+1}_0$ 
and ${\rm Supp\,}\tau\Subset U$, ${\rm Supp\,}\tau_1\times{\rm Supp\,}\chi\Subset U'\times{\rm Supp\,}\chi\Subset U$, 
where $U'$ is an open neighborhood of  $\mathring{x}_0$ in $\Real^{2n}$. For each $k>0$, we consider the distributions
\begin{equation}\label{e-gue140215fII}
\begin{split}
&A_k:u\mapsto\int^\infty_0e^{i(y_{2n-1}+h(x,\mathring{y}))t-iky_{2n+1}}\tau(x)b(x,y,t)\\
&\quad\quad\quad\quad\quad\times\tau_1(\mathring{y})\chi(y_{2n+1})u(\mathring{y})dydt,\\
&B_k:u\mapsto\int^\infty_0e^{i(y_{2n+1}+h_1(x,\mathring{y}))t-iky_{2n+1}}\tau(x)a(x,y,t)\\
&\quad\quad\quad\quad\quad\times\tau_1(\mathring{y})\chi(y_{2n+1})u(\mathring{y})dydt,
\end{split}
\end{equation}
for $u\in C^\infty_0(U',T^{*0,q}X)$.
By using the stationary phase formula of Melin-Sj\"ostrand~\cite{MS74}, we can show 
that (cf.\ the proof of~\cite[Theorem 3.12]{HM14a}) $A_k$ and $B_k$ are smoothing operators and
\begin{equation}\label{e-gue140215fIV}
\begin{split}
&A_k(x,\mathring{y})\equiv e^{ikh(x,\mathring{y})}g(x,\mathring{y},k)+O(k^{-\infty}),\\
&B_k(x,\mathring{y})\equiv e^{ikh_1(x,\mathring{y})}p(x,\mathring{y},k)+O(k^{-\infty}),\\
&g(x,\mathring{y},k), p(x,\mathring{y},k)\in S^{n-\frac{d}{2}}_{{\rm loc\,}}(1;U\times U',T^{*0,q}X\boxtimes(T^{*0,q}X)^*),\\
&g(x,\mathring{y},k)\mbox{$\sim\sum^\infty_{j=0}g_j(x,\mathring{y})k^{n-\frac{d}{2}-j}$ in $S^{n-\frac{d}{2}}_{{\rm loc\,}}(1;U\times U',T^{*0,q}X\boxtimes(T^{*0,q}X)^*)$},\\
&p(x,\mathring{y},k)\mbox{$\sim\sum^\infty_{j=0}p_j(x,y')k^{n-\frac{d}{2}-j}$ in $S^{n-\frac{d}{2}}_{{\rm loc\,}}(1;U\times U',T^{*0,q}X\boxtimes(T^{*0,q}X)^*)$},\\
&g_j(x,\mathring{y}), p_j(x,\mathring{y})\in C^\infty(U\times U',T^{*0,q}X\boxtimes(T^{*0,q}X)^*),\ \ j=0,1,\ldots,\\
&g_0(x_0,\mathring{x}_0)\neq0.
\end{split}
\end{equation}
Since
\[\int^\infty_0e^{i(y_{2n+1}+h(x,\mathring{y}))t}b(x,y,t)dt-\int^\infty_0e^{i(y_{2n+1}+h_1(x,\mathring{y}))t}a(x,y,t)dt\]
is smoothing, by using integration by parts with respect to $y_{2n+1}$, it is easy to see that
$A_k-B_k=O(k^{-\infty})$ (see~\cite[Section 3]{HM14a}). Thus,
\begin{equation}\label{e-gue140215fV}
\begin{split}
&e^{ikh(x,\mathring{y})}g(x,\mathring{y},k)=e^{ikh_1(x,\mathring{y})}p(x,\mathring{y},k)+F_k(x,\mathring{y}),\\
&F_k(x, \mathring{y}')=O(k^{-\infty}).
\end{split}
\end{equation}
Now, we are ready to prove that $h(x,\mathring{y})-h_1(x,\mathring{y})$ vanishes to infinite order at $(x_0,\mathring{x}_0)$. We assume that there exist $\alpha_0\in\mathbb N^{2n+1}_0$, $\beta_0\in\mathbb N^{2n}_0$, $\abs{\alpha_0}+\abs{\beta_0}\geq1$ such that
\[\abs{\pr^{\alpha_0}_x\pr^{\beta_0}_{\mathring{y}}(ih(x,\mathring{y})-ih_1(x,\mathring{y}))}_{(x_0,\mathring{x}_0)}=C_{\alpha_0,\beta_0}\neq0\]
and
\[\abs{\pr^{\alpha}_x\pr^{\beta}_{\mathring{y}}(ih(x,\mathring{y})-ih_1(x,\mathring{y}))}_{(x_0,\mathring{x}_0)}=0\ \ \mbox{if $\abs{\alpha}+\abs{\beta}<\abs{\alpha_0}+\abs{\beta_0}$}.\]
From \eqref{e-gue140215fV}, we have
\begin{equation}\label{e-gue140215fVI}
\begin{split}
&\abs{\pr^{\alpha_0}_x\pr^{\beta_0}_{\mathring{y}}\Bigr(e^{ikh(x,\mathring{y})-ikh_1(x,\mathring{y})}g(x,\mathring{y},k)-p(x,\mathring{y},k)\Bigr)}_{(x_0,\mathring{x}_0)}\\
&=-\abs{\pr^{\alpha_0}_x\pr^{\beta_0}_{\mathring{y}}\Bigr(e^{-ikh_1(x,\mathring{y})}F_k(x,\mathring{y})\Bigr)}_{(x_0,\mathring{x}_0)}.
\end{split}
\end{equation}
Since $h_1(x_0,\mathring{x}_0)=-x^{2n+1}_0$ and $F_k(x,\mathring{y})=O(k^{-\infty})$, we have
\begin{equation} \label{e-gue140215fVII}
\lim_{k\To\infty}k^{-n+\frac{d}{2}-1}\abs{\pr^{\alpha_0}_x\pr^{\beta_0}_{\mathring{y}}\Bigr(e^{-ikh_1(x,\mathring{y})}F_k(x,\mathring{y})\Bigr)}_{(x_0,\mathring{x}_0)}=0.
\end{equation}
On the other hand, we can check that
\begin{equation} \label{e-gue140215fVIII}
\begin{split}
&\lim_{k\To\infty}k^{-n+\frac{d}{2}-1}
\abs{\pr^{\alpha_0}_x\pr^{\beta_0}_{\mathring{y}}\Bigr(e^{ikh(x,\mathring{y})-ikh_1(x,\mathring{y})}g(x,\mathring{y},k)- p(x,\mathring{y},k)\Bigr)}_{(x_0,\mathring{x}_0)}\\
&=C_{\alpha_0,\beta_0}g_0(x_0,\mathring{x}_0)\neq0
\end{split}
\end{equation}
since $g_0(x_0,\mathring{x}_0)\neq0$. From \eqref{e-gue140215fVI}, \eqref{e-gue140215fVII} and 
\eqref{e-gue140215fVIII}, we get a contradiction. Thus, $h(x,\mathring{y})-h_1(x,\mathring{y})$ vanishes to infinite order at $(x_0,\mathring{x}_0)$. 
Since $x_0$ is arbitrary, the theorem follows.
\end{proof}

\section{The proof of Theorem~\ref{t-gue170122}}\label{s-gue170225}

\subsection{Preparation}\label{s-gue170226}

Fix $p\in\mu^{-1}(0)$ and let $x=(x_1,\ldots,x_{2n+1})$ be the local coordinates as in Remark~\ref{r-gue170309} defined in an open set $U$ of $p$. We may assume that $U=\Omega_1\times\Omega_2\times\Omega_3\times\Omega_4$, where $\Omega_1\subset\Real^d$, $\Omega_2\subset\Real^d$ are open sets of $0\in\Real^d$, $\Omega_3\subset\Real^{2n-2d}$ is an open set of $0\in\Real^{2n-2d}$ and $\Omega_4$ is an open set of $0\in\Real$. From now on, we identify $\Omega_2$ with 
\[\set{(0,\ldots,0,x_{d+1},\ldots,x_{2d},0,\ldots,0)\in U;\, (x_{d+1},\ldots,x_{2d})\in\Omega_2},\] 
$\Omega_3$ with $\set{(0,\ldots,0,x_{2d+1},\ldots,x_{2n},0)\in U;\, (x_{d+1},\ldots,x_{2n})\in\Omega_3}$,  $\Omega_2\times\Omega_3$ with 
\[\set{(0,\ldots,0,x_{d+1},\ldots,x_{2n},0)\in U;\, (x_{d+1},\ldots,x_{2n})\in\Omega_2\times\Omega_3}.\] 
For $x=(x_1,\ldots,x_{2n+1})$, we write $x''=(x_{d+1},\ldots,x_{2n+1})$, $\mathring{x}''=(x_{d+1},\ldots,x_{2n})$,
$\hat x''=(x_{d+1},\ldots,x_{2d})$, 
\[\Td x''=(x_{2d+1},\ldots,x_{2n+1}),\ \ \Td{\mathring{x}}''=(x_{2d+1},\ldots,x_{2n}).\] 
From now on, we identify $x''$ with $(0,\ldots,0,x_{d+1},\ldots,x_{2n+1})\in U$, $\mathring{x}''=(x_{d+1},\ldots,x_{2n})$ with $(0,\ldots,0,x_{d+1},\ldots,x_{2n},0)\in U$, $\hat x''$ with 
\[(0,\ldots,0,x_{d+1},\ldots,x_{2d},0,\ldots,0)\in U,\] 
$\Td x''$ with $(0,\ldots,0,x_{2d+1},\ldots,x_{2n+1})\in U$, $\Td{\mathring{x}}''$ with $(0,\ldots,0,x_{2d+1},\ldots,x_{2n},0)$. Since $G\times S^1$ acts globally free on $\mu^{-1}(0)$, we take $\Omega_2$ and $\Omega_3$ small enough so that if $x, x_1\in\Omega_2\times\Omega_3$ and $x\neq x_1$, then 
\begin{equation}\label{e-gue170227c}
g\circ e^{i\theta}\circ x\neq g_1\circ e^{i\theta_1}\circ x_1,\ \ \forall (g,e^{i\theta})\in G\times S^1, \ \ \forall (g_1,e^{i\theta_1})\in G\times S^1.
\end{equation}

We now assume that $q=n_-$ and let $\Psi(x,y)\in C^\infty(U\times U)$ be as in Theorem~\ref{t-gue170128I}. From $S^{(q)}_{G,m}=(S^{(q)}_{G,m})^*$, we get 
\begin{equation}\label{e-gue170225}
e^{im\Psi(x,y)}b(x,y,m)=e^{-im\ol\Psi(y,x)}b^*(x,y,m)+O(m^{-\infty}),
\end{equation}
where $(S^{(q)}_{G,m})^*:L^2_{(0,q)}(X)\To L^2_{(0,q)}(X)$ is the adjoint of $S^{(q)}_{G,m}:L^2_{(0,q)}(X)\To L^2_{(0,q)}(X)$ with respect to $(\,\cdot\,|\,\cdot\,)$ and $b^*(x,y,m):T^{*0,q}_xX\To T^{*0,q}_yX$ is the adjoint of $b(x,y,m):T^{*0,q}_yX\To T^{*0,q}_xX$ with respect to $\langle\,\cdot\,|\,\cdot\,\rangle$. From \eqref{e-gue170225}, we can repeat the proof of Theorem~\ref{t-gue140305II} with minor change and deduce that 
\begin{equation}\label{e-gue170225I}
\mbox{$\Psi(x,y)+\ol\Psi(y,x)$ vanishes to infinite order at ${\rm diag\,}\Bigr((\mu^{-1}(0)\bigcap U)\times(\mu^{-1}(0)\bigcap U)\Bigr)$}.
\end{equation}
From $\ddbar_bS^{(q)}_{G,m}=0$, we can check that 
\begin{equation}\label{e-gue170225II}
\mbox{$\ddbar_b\Psi(x,y)$ vanishes to infinite order at ${\rm diag\,}\Bigr((\mu^{-1}(0)\bigcap U)\times(\mu^{-1}(0)\bigcap U)\Bigr)$}.
\end{equation}
From \eqref{e-gue170225I}, \eqref{e-gue170225II} and notice that $\frac{\pr}{\pr x_j}-i\frac{\pr}{\pr x_{d+j}}\in T^{0,1}_xX$, $j=1,\ldots,d$, where $x\in\mu^{-1}(0)$ (see Remark~\ref{r-gue170309}), and 
$\frac{\pr}{\pr x_j}\Psi(x,y)=\frac{\pr}{\pr y_j}\Psi(x,y)=0$, $j=1,\ldots,d$, we conclude that 
\begin{equation}\label{e-gue170225III}
\begin{split}
&\mbox{$\frac{\pr}{\pr x_{d+j}}\Psi(x,y)|_{x_{d+1}=\cdots=x_{2d}=0}$ and $\frac{\pr}{\pr y_{d+j}}\Psi(x,y)|_{y_{d+1}=\cdots=y_{2d}=0}$ vanish to infinite order at}\\
&{\rm diag\,}\Bigr((\mu^{-1}(0)\bigcap U)\times(\mu^{-1}(0)\bigcap U)\Bigr).
\end{split}
\end{equation}
Let $G_j(x,y):=\frac{\pr}{\pr y_{d+j}}\Psi(x,y)|_{y_{d+1}=\cdots=y_{2d}=0}$, $H_j(x,y):=\frac{\pr}{\pr x_{d+j}}\Psi(x,y)|_{x_{d+1}=\cdots=x_{2d}=0}$. Put
\begin{equation}\label{e-gue170226}
\begin{split}
&\Psi_1(x,y):=\Psi(x,y)-\sum^d_{j=1}y_{d+j}G_j(x,y),\\
&\Psi_2(x,y):=\Psi(x,y)-\sum^d_{j=1}x_{d+j}H_j(x,y).
\end{split}
\end{equation}
Then, 
\begin{equation}\label{e-gue170226Ip}
\begin{split}
&\frac{\pr}{\pr y_{d+j}}\Psi_1(x,y)|_{y_{d+1}=\cdots=y_{2d}=0}=0,\ \ j=1,2,\ldots,d,\\
&\frac{\pr}{\pr x_{d+j}}\Psi_2(x,y)|_{x_{d+1}=\cdots=x_{2d}=0}=0,\ \ j=1,2,\ldots,d,
\end{split}
\end{equation}
and 
\begin{equation}\label{e-gue170226II}
\begin{split}
&\mbox{$\Psi(x,y)-\Psi_1(x,y)$ vanishes to infinite order at ${\rm diag\,}\Bigr((\mu^{-1}(0)\bigcap U)\times(\mu^{-1}(0)\bigcap U)\Bigr)$},\\
&\mbox{$\Psi(x,y)-\Psi_2(x,y)$ vanishes to infinite order at ${\rm diag\,}\Bigr((\mu^{-1}(0)\bigcap U)\times(\mu^{-1}(0)\bigcap U)\Bigr)$}.
\end{split}
\end{equation}

We also write $u=(u_1,\ldots,u_{2n+1})$ to denote the local coordinates of $U$. Recall that for any smooth function $f\in C^\infty(U)$, we write $\Td f\in C^\infty(U^{\Complex})$ to denote an almost analytic extension of $f$ (see the discussion after \eqref{e-gue170102III}). We consider the following two systems 
\begin{equation}\label{e-gue170226III}
\begin{split}
\frac{\pr\Td\Psi_1}{\pr\Td u_{2d+j}}(\Td x,\Td{\Td u''})+\frac{\pr\Td\Psi_2}{\pr\Td x_{2d+j}}(\Td{\Td u''},\Td y)=0,\ \ j=1,2,\ldots,2n-2d,
\end{split}
\end{equation}
and 
\begin{equation}\label{e-gue170226IIIa}
\frac{\pr\Td\Psi_1}{\pr\Td u_{d+j}}(\Td x,\Td{u''})+\frac{\pr\Td\Psi_2}{\pr\Td x_{d+j}}(\Td{u''},\Td y)=0,\ \ j=1,2,\ldots,2n-d,
\end{equation}
where $\Td{\Td u''}=(0,\ldots,0,\Td u_{2d+1},\ldots,\Td u_{2n+1})$, $\Td{u''}=(0,\ldots,0,\Td u_{d+1},\ldots,\Td u_{2n+1})$. From \eqref{e-gue170226Ip} and Theorem~\ref{t-gue170128a}, we can take $\Td\Psi_1$ and $\Td\Psi_2$ so that for every $j=1,2,\ldots,d$, 
\begin{equation}\label{e-gue170226I}
\begin{split}
&\frac{\pr\Td\Psi_1}{\pr\Td u_{d+j}}(\Td x,\Td{u''})=0\ \ \mbox{if $\Td u_{d+1}=\cdots=\Td u_{2d}=0$},\\
&\frac{\pr\Td\Psi_2}{\pr\Td x_{d+j}}(\Td{u''},\Td y)=0\ \ \mbox{if $\Td u_{d+1}=\cdots=\Td u_{2d}=0$},
\end{split}
\end{equation}
and 
\begin{equation}\label{e-gue170227}
\begin{split}
&\Td\Psi_1(\Td x, \Td y)=-\Td x_{2n+1}+\Td y_{2n+1}+\Td{\hat\Psi_1}(\Td{\mathring{x}''},\Td{\mathring{y}''}),\ \ \Td{\hat\Psi_1}\in C^\infty(U^\Complex\times U^\Complex),\\
&\Td\Psi_2(\Td x, \Td y)=-\Td x_{2n+1}+\Td y_{2n+1}+\Td{\hat\Psi_2}(\Td{\mathring{x}''},\Td{\mathring{y}''}),\ \ \Td{\hat\Psi_2}\in C^\infty(U^\Complex\times U^\Complex),
\end{split}
\end{equation}
where $\Td{\mathring{x}''}=(0,\ldots,0,\Td x_{d+1},\ldots,\Td x_{2n},0)$, $\Td{\mathring{y}''}=(0,\ldots,0,\Td y_{d+1},\ldots,\Td y_{2n},0)$.

From Theorem~\ref{t-gue170128a}, \eqref{e-gue170126I} and 
\[d_x\Psi(x,x)=-d_y\Psi(x,x)=-\omega_0(x),\ \ \forall x\in\mu^{-1}(0),\]
it is not difficult to see that 
\[\frac{\pr\Td\Psi_1}{\pr\Td u_{d+j}}(\Td x'',\Td x'')+\frac{\pr\Td\Psi_2}{\pr\Td x_{d+j}}(\Td x'',\Td x'')=0,\ \ j=1,2,\ldots,2n-d,\]
and the matrices 
\[\left(\frac{\pr^2\Psi}{\pr u_{2d+j}\pr u_{2d+k}}(p,p)+\frac{\pr^2\Psi}{\pr x_{2d+j}\pr x_{2d+k}}(p,p)\right)^{2n-2d}_{j,k=1},\ \ 
\left(\frac{\pr^2\Psi}{\pr u_{d+j}\pr u_{d+k}}(p,p)+\frac{\pr^2\Psi}{\pr x_{d+j}\pr x_{d+k}}(p,p)\right)^{2n-d}_{j,k=1}\] 
are non-singular. Moreover, 
\begin{equation}\label{e-gue170226a}
\begin{split}
&\det\left(\frac{\pr^2\Psi}{\pr u_{2d+j}\pr u_{2d+k}}(p,p)+\frac{\pr^2\Psi}{\pr x_{2d+j}\pr x_{2d+k}}(p,p)\right)^{2n-2d}_{j,k=1}=(4i\abs{\mu_{d+1}}\cdots 4i\abs{\mu_n})^2,\\
&\det\left(\frac{\pr^2\Psi}{\pr u_{d+j}\pr u_{d+k}}(p,p)+\frac{\pr^2\Psi}{\pr x_{d+j}\pr x_{d+k}}(p,p)\right)^{2n-d}_{j,k=1}=(8i\abs{\mu_1}\cdots 8i\abs{\mu_d})(4i\abs{\mu_{d+1}}\cdots 4i\abs{\mu_n})^2.
\end{split}
\end{equation}
Hence, near $(p,p)$, we can solve \eqref{e-gue170226III} and \eqref{e-gue170226IIIa} and the solutions are unique. Let $\alpha(x,y)=(\alpha_{2d+1}(x,y),\ldots,\alpha_{2n}(x,y))\in C^\infty(U\times U,\Complex^{2n-2d})$ and $\beta(x,y)=(\beta_{d+1}(x,y),\ldots,\beta_{2n}(x,y))\in C^\infty(U\times U,\Complex^{2n-d})$ be the solutions of \eqref{e-gue170226III} and \eqref{e-gue170226IIIa}, respectively. From \eqref{e-gue170226I}, it is easy to see that 
\begin{equation}\label{e-gue170226b}
\beta(x,y)=(\beta_{d+1}(x,y),\ldots,\beta_{2n}(x,y))=(0,\ldots,0,\alpha_{2d+1}(x,y),\ldots,\alpha_{2n}(x,y)).
\end{equation}
From \eqref{e-gue170226b}, we see that the value of $\Td\Psi_1(x,\Td{\Td u''})+\Td\Psi_2(\Td{\Td u''},y)$ at critical point $\Td{\Td u''}=\alpha(x,y)$ is equal to the value of 
$\Td\Psi_1(x,\Td{u''})+\Td\Psi_2(\Td{u''},y)$ at critical point $\Td{u''}=\beta(x,y)$. Put 
\begin{equation}\label{e-gue170226bI}
\Psi_3(x,y):=\Td\Psi_1(x,\alpha(x,y))+\Td\Psi_2(\alpha(x,y),y)=\Td\Psi_1(x,\beta(x,y))+\Td\Psi_2(\beta(x,y),y).
\end{equation}
$\Psi_3(x,y)$ is a complex phase function. From \eqref{e-gue170227}, we have
\begin{equation}\label{e-gue170227I}
\Psi_3(x,y)=-x_{2n+1}+y_{2n+1}+\hat\Psi_3(\mathring{x}'',\mathring{y}''),\ \ \hat\Psi_3(\mathring{x}'',\mathring{y}'')\in C^\infty(U\times U).
\end{equation}

Moreover, we have the following

\begin{theorem}\label{t-gue170226cw}
The function $\Psi_3(x,y)-\Psi(x,y)$ vanishes to infinite order at ${\rm diag\,}\Bigr((\mu^{-1}(0)\bigcap U)\times(\mu^{-1}(0)\bigcap U)\Bigr)$. 
\end{theorem}

\begin{proof}
We consider the kernel $S^{(q)}_{G,m}\circ S^{(q)}_{G,m}$ on $U$. Let $V\Subset U$ be an open set of $p$. Let $\chi(\mathring{x}'')\in C^\infty_0(\Omega_2\times\Omega_3)$. From \eqref{e-gue170227c}, we can extend $\chi(\mathring{x}'')$ to 
\[W:=\set{g\circ e^{i\theta}\circ x;\, (g,e^{i\theta})\in G\times S^1, x\in\Omega_2\times\Omega_3}\]
by $\chi(g\circ e^{i\theta}\circ\mathring{x}''):=\chi(\mathring{x}'')$, for every $(g,e^{i\theta})\in G\times S^1$. Assume that  $\chi=1$ on some neighborhood of $V$. Let $\chi_1\in C^\infty_0(U)$ with $\chi_1=1$ on some neighborhood 
of $V$ and ${\rm Supp\,}\chi_1\subset\set{x\in X;\, \chi(x)=1}$. We have 
\begin{equation}\label{e-gue170227b}
\chi_1S^{(q)}_{G,m}\circ S^{(q)}_{G,m}=\chi_1S^{(q)}_{G,m}\chi\circ S^{(q)}_{G,m}+\chi_1S^{(q)}_{G,m}(1-\chi)\circ S^{(q)}_{G,m}.
\end{equation}
Let's first consider $\chi_1S^{(q)}_{G,m}(1-\chi)\circ S^{(q)}_{G,m}$. We have 
\begin{equation}\label{e-gue170227bI}
(\chi_1S^{(q)}_{G,m}(1-\chi))(x,u)
=\frac{1}{\abs{G}_{d\mu}2\pi}\chi_1(x)\int^{\pi}_{-\pi}\int_GS^{(q)}(x,g\circ e^{i\theta}\circ u)(1-\chi(u))e^{im\theta}d\mu(g)d\theta.
\end{equation}
If $u\notin\set{x\in X;\, \chi(x)=1}$. Since ${\rm Supp\,}\chi_1\subset\set{x\in X;\, \chi(x)=1}$ and $\chi(x)=\chi(g\circ e^{i\theta}\circ x)$, for every $(g,e^{i\theta})\in G\times S^1$, for every $x\in X$, we conclude that $g\circ e^{i\theta}\circ u\notin{\rm Supp\,}\chi_1$, for every $(g,e^{i\theta})\in G\times S^1$. From this observation and notice that $S^{(q)}$ 
is smoothing away the diagonal, we can integrate by parts with respect to $\theta$ in \eqref{e-gue170227bI} and deduce that $\chi_1S^{(q)}_{G,m}\circ(1-\chi)=O(m^{-\infty})$ and hence 
\begin{equation}\label{e-gue170227bII}
\chi_1S^{(q)}_{G,m}(1-\chi)\circ S^{(q)}_{G,m}=O(m^{-\infty}). 
\end{equation}
From \eqref{e-gue170227b} and \eqref{e-gue170227bII}, we get 
\begin{equation}\label{e-gue170227bIII}
\chi_1S^{(q)}_{G,m}\circ S^{(q)}_{G,m}=\chi_1S^{(q)}_{G,m}\chi\circ S^{(q)}_{G,m}+O(m^{-\infty}). 
\end{equation}
We can check that on $U$, 
\begin{equation}\label{e-gue170227y}
\begin{split}
&(\chi_1S^{(q)}_{G,m}\chi\circ S^{(q)}_{G,m})(x,y)\\
&=(2\pi)\int e^{im\Psi(x,u'')+im\Psi(u'',y)}\chi_1(x)b(x,\mathring{u}'',m)\chi(\mathring{u}'')b(\mathring{u}'',y,m)dv(\mathring{u}'')+O(m^{-\infty})\\
&=(2\pi)\int e^{im\Psi_1(x,u'')+im\Psi_2(u'',y)}\chi_1(x)b(x,\mathring{u}'',m)\chi(\mathring{u}'')b(\mathring{u}'',y,m)dv(\mathring{u}'')+O(m^{-\infty})\\ 
&\mbox{(here we use \eqref{e-gue170226II})},
\end{split}
\end{equation}
where $d\mu(g)d\theta dv(\mathring{u}'')=dv(x)$ on $U$. We use complex stationary phase formula of Melin-Sj\"ostrand~\cite{MS74} to carry out the integral \eqref{e-gue170227y} and get 
\begin{equation}\label{e-gue170227yI}
\begin{split}
&(\chi_1S^{(q)}_{G,m}\chi\circ S^{(q)}_{G,m})(x,y)=e^{im\Psi_3(x,y)}a(x,y,m)+O(m^{-\infty})\ \ \mbox{on $U$}, \\
&a(x,y,m)\in S^{n-\frac{d}{2}}_{{\rm loc\,}}(1; U\times U, T^{*0,q}X\boxtimes(T^{*0,q}X)^*),\\
&\mbox{$a(x,y,m)\sim\sum^\infty_{j=0}m^{n-\frac{d}{2}-j}a_j(x,y)$ in $S^{n-\frac{d}{2}}_{{\rm loc\,}}(1; U\times U, T^{*0,q}X\boxtimes(T^{*0,q}X)^*)$},\\
&a_j(x,y)\in C^\infty(U\times U, T^{*0,q}X\boxtimes(T^{*0,q}X)^*),\ \ j=0,1,2,\ldots,\\
&a_0(p,p)\neq0.
\end{split}
\end{equation}
From \eqref{e-gue170227bIII}, \eqref{e-gue170227yI} and notice that $(\chi_1S^{(q)}_{G,m}\circ S^{(q)}_{G,m})(x,y)=(\chi_1S^{(q)}_{G,m})(x,y)$, we deduce that 
\begin{equation}\label{e-gue170227yII}
e^{im\Psi_3(x,y)}a(x,y,m)=e^{im\Psi(x,y)}\chi_1(x)b(x,y,m)+O(m^{-\infty})\ \ \mbox{on $U$}. 
\end{equation}
From \eqref{e-gue170227yII}, we can repeat the proof of Theorem~\ref{t-gue140305II} with minor change and deduce that $\Psi_3(x,y)-\Psi(x,y)$ vanishes to infinite order at ${\rm diag\,}\Bigr((\mu^{-1}(0)\bigcap U)\times(\mu^{-1}(0)\bigcap U)\Bigr)$.
\end{proof}

The following two theorems follow from \eqref{e-gue170226II}, \eqref{e-gue170226bI}, Theorem~\ref{t-gue170226cw}, complex stationary phase formula of Melin-Sj\"ostrand~\cite{MS74} and some straightforward computation. We omit the details. 

\begin{theorem}\label{t-gue170301w}
With the notations used above, let 
\[\begin{split}
&A_m(x,y)=e^{im\Psi(x,y)}a(x,y,m),\ \ B_m(x,y)=e^{im\Psi(x,y)}b(x,y,m),\\
&a(x,y,m)\in S^{k}_{{\rm loc\,}}(1; U\times U, H\boxtimes F^*),\\
&b(x,y,m)\in S^{\ell}_{{\rm loc\,}}(1; U\times U, F\boxtimes E^*),\\
&\mbox{$a(x,y,m)\sim\sum^\infty_{j=0}m^{k-j}a_j(x,y)$ in $S^{k}_{{\rm loc\,}}(1; U\times U, H\boxtimes F^*)$},\\
&\mbox{$b(x,y,m)\sim\sum^\infty_{j=0}m^{\ell-j}b_j(x,y)$ in $S^{\ell}_{{\rm loc\,}}(1; U\times U, F\boxtimes E^*)$},\\
&a_j(x,y)\in C^\infty(U\times U, H\boxtimes F^*),\ \ j=0,1,2,\ldots, \\
&b_j(x,y)\in C^\infty(U\times U, F\boxtimes E^*),\ \ j=0,1,2,\ldots, 
\end{split}\]
where $E$, $F$ and $H$ are vector bundles over $X$. 
Let $\chi(\mathring{x}'')\in C^\infty_0(\Omega_2\times\Omega_3)$. Then, we have 
\[\begin{split}
&\int A_m(x,u)\chi(\mathring{u}'')B_m(u,y)dv(\mathring{u}'')=e^{im\Psi(x,y)}c(x,y,m)+O(m^{-\infty}),\\
&c(x,y,m)\in S^{k+\ell-(n-\frac{d}{2})}_{{\rm loc\,}}(1; U\times U, H\boxtimes E^*),\\
&\mbox{$c(x,y,m)\sim\sum^\infty_{j=0}m^{k+\ell-(n-\frac{d}{2})-j}c_j(x,y)$ in $S^{k+\ell-(n-\frac{d}{2})}_{{\rm loc\,}}(1; U\times U, H\boxtimes E^*)$},\\
\end{split}\]
\begin{equation}\label{e-gue170301u}
c_0(x,x)=2^{-n-\frac{d}{2}}\pi^{n-\frac{d}{2}}\abs{\det\mathcal{L}_{x}}^{-1}\abs{\det R_x}^{\frac{1}{2}}a_0(x,x)b_0(x,x)\chi(\mathring{x}''),\ \ \forall x\in\mu^{-1}(0)\bigcap U,
\end{equation}
where $\abs{\det R_x}$ is in the discussion before Theorem~\ref{t-gue170128}. 

Moreover, if there are $N_1, N_2\in\mathbb N$, such that $\abs{a_0(x,y)}\leq C\abs{(x,y)-(x_0,x_0)}^{N_1}$,  $\abs{b_0(x,y)}\leq C\abs{(x,y)-(x_0,x_0)}^{N_2}$, 
for all $x_0\in\mu^{-1}(0)\bigcap U$, where $C>0$ is a constant, then, 
\begin{equation}\label{e-gue170301uI}
\abs{c_0(x,y)}\leq\hat C\abs{(x,y)-(x_0,x_0)}^{N_1+N_2}, 
\end{equation}
for all $x_0\in\mu^{-1}(0)\bigcap U$, where $\hat C>0$ is a constant.
\end{theorem}

\begin{theorem}\label{t-gue170301wI}
With the notations used above, let 
\[\begin{split}
&\mathcal{A}_m(x,\Td y'')=e^{im\Psi(x,\Td y'')}\alpha(x,\Td y'',m),\ \ \mathcal{B}_m(\Td x'',y)=e^{im\Psi(\Td x'',y)}\beta(\Td x'',y,m),\\
&\alpha(x,\Td y'',m)\in S^{k}_{{\rm loc\,}}(1; U\times(\Omega_3\times\Omega_4), H\boxtimes F^*),\\
&\beta(\Td x'',y,m)\in S^{\ell}_{{\rm loc\,}}(1; (\Omega_3\times\Omega_4)\times U, F\boxtimes E^*),\\
&\mbox{$\alpha(x,\Td y'',m)\sim\sum^\infty_{j=0}m^{k-j}\alpha_j(x,\Td y'')$ in $S^{k}_{{\rm loc\,}}(1; U\times(\Omega_3\times\Omega_4), H\boxtimes F^*)$},\\
&\mbox{$\beta(\Td x'',y,m)\sim\sum^\infty_{j=0}m^{\ell-j}\beta_j(\Td x'',y)$ in $S^{\ell}_{{\rm loc\,}}(1; (\Omega_3\times\Omega_4)\times U, F\boxtimes E^*)$},\\
&\alpha_j(x,\Td y'')\in C^\infty(U\times(\Omega_3\times\Omega_4), H\boxtimes F^*),\ \ j=0,1,2,\ldots, \\
&\beta_j(\Td x'',y)\in C^\infty((\Omega_3\times\Omega_4)\times U, F\boxtimes E^*),\ \ j=0,1,2,\ldots, 
\end{split}\]
where $E$, $F$ and $H$ are vector bundles over $X$. 
Let $\chi_1(\Td{\mathring{x}}'')\in C^\infty_0(\Omega_3)$. Then, we have 
\[\begin{split}
&\int\mathcal{A}_m(x,\Td u'')\chi_1(\Td{\mathring{u}}'')\mathcal{B}_m(\Td u'',y)dv(\Td{\mathring{u}})=e^{im\Psi(x,y)}\gamma(x,y,m)+O(m^{-\infty}),\\
&\gamma(x,y,m)\in S^{k+\ell-(n-d)}_{{\rm loc\,}}(1; U\times U, H\boxtimes E^*),\\
&\mbox{$\gamma(x,y,m)\sim\sum^\infty_{j=0}m^{k+\ell-(n-d)-j}\gamma_j(x,y)$ in $S^{k+\ell-(n-d)}_{{\rm loc\,}}(1; U\times U, H\boxtimes E^*)$},\\
\end{split}\]
\begin{equation}\label{e-gue170301ua}
\gamma_0(x,x)=2^{-n}\pi^{n-d}\abs{\det\mathcal{L}_{x}}^{-1}\abs{\det R_x}\alpha_0(x,\Td x'')\beta_0(\Td x'',x)\chi_1(\Td{\mathring{x}}''),\ \ \forall x\in\mu^{-1}(0)\bigcap U,
\end{equation}
where $\abs{\det R_x}$ is in the discussion before Theorem~\ref{t-gue170128}. 

Moreover, if there are $N_1, N_2\in\mathbb N$, such that $\abs{\alpha_0(x,\Td y'')}\leq C\abs{(x,\Td y'')-(x_0,x_0)}^{N_1}$,  $\abs{\beta_0(x,\Td y'')}\leq C\abs{(x,\Td y'')-(x_0,x_0)}^{N_2}$, 
for all $x_0\in\mu^{-1}(0)\bigcap U$, where $C>0$ is a constant, then, 
\begin{equation}\label{e-gue170301uaI}
\abs{\gamma_0(x,y)}\leq\hat C\abs{(x,y)-(x_0,x_0)}^{N_1+N_2}, 
\end{equation}
for all $x_0\in\mu^{-1}(0)\bigcap U$, where $\hat C>0$ is a constant.
\end{theorem}

\subsection{The proof of Theorem~\ref{t-gue170122}}\label{s-gue170303}

Since $\underline{\mathfrak{g}}_x$ is orthogonal to $H_xY\bigcap JH_xY$ and $H_xY\bigcap JH_xY\subset\underline{\mathfrak{g}}^{\perp_b}_x$, for every $x\in Y$, we can find a $G$-invariant orthonormal basis $\set{Z_1,\ldots,Z_n}$ of $T^{1,0}X$ on $Y$ such that 
\[\mathcal{L}_x(Z_j(x),\ol Z_k(x))=\delta_{j,k}\lambda_j(x),\ \ j,k=1,\ldots,n,\ \ x\in Y,\]
and 
\[\begin{split}
&Z_j(x)\in\underline{\mathfrak{g}}_x+iJ\underline{\mathfrak{g}}_x,\ \ \forall x\in Y,\ \ j=1,2,\ldots,d,\\
&Z_j(x)\in\Complex H_xY\bigcap J(\Complex H_xY),\ \ \forall x\in Y,\ \ j=d+1,\ldots,n.
\end{split}\]
Let $\set{e_1,\ldots,e_n}$ denote the orthonormal basis of $T^{*0,1}X$ on $Y$, dual to $\set{\ol Z_1,\ldots,\ol Z_n}$.
Fix $s=0,1,2,\ldots,n-d$. For $x\in Y$, put 
\begin{equation}\label{e-gue170303}
B^{*0,s}_xX=\set{\sum_{d+1\leq j_1<\cdots<j_s\leq n}a_{j_1,\ldots,j_s}e_{j_1}\wedge\cdots\wedge e_{j_s};\, a_{j_1,\ldots,j_s}\in\Complex,\ \forall d+1\leq j_1<\cdots<j_s\leq n}
\end{equation}
and let $B^{*0,s}X$ be the vector bundle of $Y$ with fiber $B^{*0,s}_xX$, $x\in Y$. Let $C^\infty(Y,B^{*0,s}X)^G$ denote the set of all $G$-invariant sections of $Y$ with values in $B^{*0,s}X$. Let 
\begin{equation}\label{e-gue170303cw}
\iota_G:C^\infty(Y,B^{*0,s}X)^G\To\Omega^{0,s}(Y_G)
\end{equation}
be the natural identification. 

Assume that $\lambda_1<0,\ldots,\lambda_r<0$, and $\lambda_{d+1}<0,\ldots,\lambda_{n_--r+d}<0$. For $x\in Y$, put 
\begin{equation}\label{e-gue170303I}
\hat{\mathcal{N}}(x,n_-)=\set{ce_{d+1}\wedge\cdots\wedge e_{n_--r+d};\, c\in\Complex},
\end{equation}
and let 
\begin{equation}\label{e-gue170303cwI}
\begin{split}
&\hat p=\hat p_x:\mathcal{N}(x,n_-)\To \hat{\mathcal{N}}(x,n_-),\\
&u=ce_1\wedge\cdots\wedge e_r\wedge e_{d+1}\wedge\cdots\wedge e_{n_--r+d}\To ce_{d+1}\wedge\cdots\wedge e_{n_--r+d}.
\end{split}
\end{equation}
Let $\iota:Y\To X$ be the natural inclusion and let $\iota^*:\Omega^{0,q}(X)\To\Omega^{0,q}(Y)$ be the pull-back of $\iota$. Recall that we work with the assumption that $q=n_-$. 
Let $\Box^{(q-r)}_{b,Y_G}$ be the Kohn Laplacian for $(0,q-r)$ forms on $Y_G$. Fix $m\in\mathbb N$. Let 
\[H^{q-r}_{b,m}(Y_G):=\set{u\in\Omega^{0,q-r}(Y_G);\, \Box^{(q-r)}_{b,Y_G}u=0,\ \ Tu=imu}.\]
Let $S^{(q-r)}_{Y_G,m}:L^2_{(0,q-r)}(Y_G)\To H^{q-r}_{b,m}(Y_G)$ be the orthogonal projection and let $S^{(q-r)}_{Y_G,m}(x,y)$ be the distribution kernel of $S^{(q-r)}_{Y_G,m}$. Let 
\begin{equation}\label{e-gue170303cwa}
f(x)=\sqrt{V_{{\rm eff\,}}(x)}\abs{\det\,R_x}^{-\frac{1}{4}}\in C^\infty(Y)^G.
\end{equation}
Let
\begin{equation}\label{e-gue170303cwII}
\begin{split}
\sigma_m:\Omega^{0,q}(X)&\To H^{q-r}_{b,m}(Y_G),\\
u&\To m^{-\frac{d}{4}}S^{(q-r)}_{Y_G,m}\circ\iota_G\circ \hat p\circ\tau_{x,n_-}\circ f\circ \iota^*\circ S^{(q)}_{G,m}u.
\end{split}
\end{equation}
Recall that $\tau_{x,n_-}$ is given by \eqref{tau140530}. Let $\sigma^*_m:\Omega^{0,q-r}(Y_G)\To \Omega^{0,q}(X)$ be the adjoints of $\sigma_m$. It is easy to see that 
\[\sigma^*_mu\in H^q_{b,m}(X)^G:=({\rm Ker\,}\Box^{(q)}_b)^G_m,\ \ \forall u\in\Omega^{0,q-r}(Y_G).\]
Let $\sigma_m(x,y)$ and $\sigma^*_m(x,y)$ denote the distribution kernels of $\sigma_m$ and $\sigma^*_m$, respectively.

Let's pause and recall some well-known results for $S^{(q-r)}_{Y_G,m}$. We first introduce some notations.  Let $\mathcal{L}_{Y_G,x}$ be the Levi form on $Y_G$  at $x\in Y_G$ induced naturally from $\mathcal{L}$. The Hermitian metric $\langle\,\cdot\,|\,\cdot\,\rangle$ on $T^{1,0}X$ induces a Hermitian metric $\langle\,\cdot\,|\,\cdot\,\rangle$ on $T^{1,0}Y_G$. Let $\det\,\mathcal{L}_{Y_G,x}=\lambda_1\ldots\lambda_{n-d}$, where $\lambda_j$, $j=1,\ldots,n-d$, are the eigenvalues of $\mathcal{L}_{Y_G,x}$ with respect to the Hermitian metric $\langle\,\cdot\,|\,\cdot\,\rangle$. For $x\in Y_G$, let 
\[\hat\tau_x:T^{*0,q-r}_xY_G\To \hat{\mathcal{N}}(x,n_-)\]
be the orthogonal projection. 

Let $\pi:Y\To Y_G$ be the natural quotient. Let $S^{(q-r)}_{Y_G}:L^2_{(0,q-r)}(Y_G)\To{\rm Ker\,}\Box^{(q-r)}_{b,Y_G}$ be the Szeg\"o projection as \eqref{e-suXI-I}. Since $S^{(q-r)}_{Y_G}$ is smoothing away the diagonal (see Theorem~\ref{t-gue161109I}), it is easy to see that for  any $x, y\in Y$, if $\pi(e^{i\theta}\circ x)\neq\pi(e^{i\theta}\circ y)$, for every $\theta\in[0,2\pi[$, then there are open sets $U$ of $\pi(x)$ in $Y_G$ and $V$ of $\pi(y)$ in $Y_G$ such that for all $\hat\chi\in C^\infty_0(U)$, $\Td\chi\in C^\infty_0(V)$, we have 
\begin{equation}\label{e-gue170304}
\hat\chi S^{(q-r)}_{Y_G,m}\Td\chi=O(m^{-\infty})\ \ \mbox{on $Y_G$}. 
\end{equation}
Fix $p\in Y$ and let $x=(x_1,\ldots,x_{2n+1})$ be the local coordinates as in Remark~\ref{r-gue170309}. We will use the same notations as in the beginning of Section~\ref{s-gue170226}. From now on, we identify $\Td x''$ as local coordinates of $Y_G$ near $\pi(p)\in Y_G$ and we identify $W:=\Omega_3\times\Omega_4$ with an open set of $\pi(p)$ in $Y_G$. It is well-known that (see Theorem 4.11 in~\cite{HM14a}), as $m\To+\infty$, 
\begin{equation}\label{e-gue170304I}
\begin{split}
&S^{(q-r)}_{Y_G,m}(\Td x'',\Td y'')=e^{im\phi(\Td x'',\Td y'')}b(\Td x'',\Td y'',m)+O(m^{-\infty})\ \ \mbox{on $W$},\\
&\beta(\Td x'',\Td y'',m)\in S^{n-d}_{{\rm loc\,}}(1; W\times W, T^{*0,q-r}Y_G\boxtimes(T^{*0,q-r}Y_G)^*),\\
&\mbox{$\beta(\Td x'',\Td y'',m)\sim\sum^\infty_{j=0}m^{n-d-j}b_j(\Td x'',\Td y'')$ in $S^{n-d}_{{\rm loc\,}}(1; W\times W, T^{*0,q-r}Y_G\boxtimes(T^{*0,q-r}Y_G)^*)$},\\
&\beta_j(\Td x'',\Td y'')\in C^\infty(W\times W, T^{*0,q-r}Y_G\boxtimes(T^{*0,q-r}Y_G)^*),\ \ j=0,1,2,\ldots,
\end{split}
\end{equation}
\begin{equation}\label{e-gue170304II}
\beta_0(\Td x'',\Td x'')=\frac{1}{2}\pi^{-(n-d)-1}\abs{\det\,\mathcal{L}_{Y_G,\Td x''}}\hat\tau_{\Td x''},\ \ \forall \Td x''\in W,
\end{equation}
and 
\begin{equation}\label{e-gue170304III}
\begin{split}
&\phi(\Td x'',\Td y'')=-x_{2n+1}+y_{2n+1}+\hat\phi(\Td{\mathring{x}}'',\Td{\mathring{y}}'')\in C^\infty(W\times W),\\
&d_{\Td x''}\phi(\Td x'',\Td y'')=-d_{\Td y''}\phi(\Td x'',\Td x'')=-\omega_0(\Td x''),\\
&{\rm Im\,}\hat\phi(\Td{\mathring{x}}'',\Td{\mathring{y}}'')\geq c\abs{\Td{\mathring{x}}''-\Td{\mathring{y}}''}^2,\ \ \mbox{where $c>0$ is a constant}, \\
&\mbox{$p_0(\Td x'', d_{\Td x''}\phi(\Td x'',\Td y''))$ vanishes to infinite order at $\Td{\mathring{x}}''=\Td{\mathring{y}}''$},\\
&\phi(\Td x'', \Td y'')=-x_{2n+1}+y_{2n+1}+i\sum^{n}_{j=d+1}\abs{\mu_j}\abs{z_j-w_j}^2 \\
&\quad+\sum^{n}_{j=d+1}i\mu_j(\ol z_jw_j-z_j\ol w_j)+O(\abs{(\Td{\mathring{x}}'', \Td{\mathring{y}}'')}^3),
\end{split}
\end{equation}
where $p_0$ denotes the principal symbol of $\Box^{(q-r)}_{b,Y_G}$, $z_j=x_{2j-1}+ix_{2j}$, $j=d+1,\ldots,n$, and $\mu_{d+1},\ldots,\mu_n$ are the eigenvalues of $\mathcal{L}_{Y_G,p}$. 

It is well-known that (see Remark 3.6 in~\cite{HM14a} ) for any $\phi_1(\Td x'',\Td y'')\in C^\infty(W\times W)$, if $\phi_1$ satisfies \eqref{e-gue170304III}, then $\phi_1-\phi$ vanishes to infinite order at $\Td{\mathring{x}}''=\Td{\mathring{y}}''$. It is not difficult to see that the phase function $\Psi(\Td x'',\Td y'')$ satisfies \eqref{e-gue170304III}. Hence, we can replace the phase $\phi(\Td x'',\Td y'')$ by $\Psi(\Td x'',\Td y'')$ and we have 
\begin{equation}\label{e-gue170304ry}
S^{(q-r)}_{Y_G,m}(\Td x'',\Td y'')=e^{im\Psi(\Td x'',\Td y'')}\beta(\Td x'',\Td y'',m)+O(m^{-\infty})\ \ \mbox{on $W$}.
\end{equation}
We can now prove 

\begin{theorem}\label{t-gue170304ry}
With the notations used above, if $y\notin Y$, then for any open set $D$ of $y$ with $\ol D\bigcap Y=\emptyset$, we have
\begin{equation}\label{e-gue170304ryI}
\sigma_m=O(m^{-\infty})\ \ \mbox{on $Y_G\times D$.}
\end{equation}

Let $x, y\in Y$.  If $\pi(e^{i\theta}\circ x)\neq\pi(e^{i\theta}\circ y)$, for every $\theta\in[0,2\pi[$, then there are open sets $U_G$ of $\pi(x)$ in $Y_G$ and $V$ of $y$ in $X$ such that 
\begin{equation}\label{e-gue170304ryII}
\sigma_m=O(m^{-\infty})\ \ \mbox{on $U_G\times V$}. 
\end{equation}

Let $p\in\mu^{-1}(0)$  and let $x=(x_1,\ldots,x_{2n+1})$ be the local coordinates as in Remark~\ref{r-gue170309}. Then, 
\begin{equation}\label{e-gue170304ryIII}
\begin{split}
&\sigma_m(\Td x'',y)=e^{im\Psi(\Td x'',y'')}\alpha(\Td x'',y'',m)+O(m^{-\infty})\ \ \mbox{on $W\times U$},\\
&\alpha(\Td x'',y'',m)\in S^{n-\frac{3}{4}d}_{{\rm loc\,}}(1; W\times U, T^{*0,q-r}Y_G\boxtimes(T^{*0,q}X)^*),\\
&\mbox{$\alpha(\Td x'',y'',m)\sim\sum^\infty_{j=0}m^{n-\frac{3}{4}d-j}\alpha_j(\Td x'',y'')$ in $S^{n-\frac{3}{4}d}_{{\rm loc\,}}(1; W\times U, T^{*0,q-r}Y_G\boxtimes(T^{*0,q}X)^*)$},\\
&\alpha_j(\Td x'',y'')\in C^\infty(W\times U, T^{*0,q-r}Y_G\boxtimes (T^{*0,q}X)^*),\ \ j=0,1,2,\ldots,
\end{split}
\end{equation}
\begin{equation}\label{e-gue170304rc}
\alpha_0(\Td x'',\Td x'')=2^{-n+2d-1}\pi^{\frac{d}{2}-n-1}\frac{1}{\sqrt{V_{{\rm eff\,}}(\Td x'')}}\abs{\det\,\mathcal{L}_{\Td x''}}\abs{\det\,R_x}^{-\frac{3}{4}}\hat\tau_{\Td x''}\tau_{\Td x'',n_-},\ \ \forall \Td x''\in W,
\end{equation}
where $U$ is an open set of $p$, $W=\Omega_3\times\Omega_4$, $\Omega_3$ and $\Omega_4$ are open sets as in the beginning of Section~\ref{s-gue170226}. 
\end{theorem}

\begin{proof}
Note that $S^{(q)}_{G,m}=O(m^{-\infty})$ away $Y$. From this observation, we get \eqref{e-gue170304ryI}. Let $x, y\in Y$.  Assume that $\pi(e^{i\theta}\circ x)\neq\pi(e^{i\theta}\circ y)$, for every $\theta\in[0,2\pi[$. Since 
\[S^{(q)}_{G,m}(x,y)=\frac{1}{2\pi\abs{G}_{d\mu}}\int^{\pi}_{-\pi}\int_GS^{(q)}(x,e^{i\theta}\circ g\circ y)e^{im\theta}d\mu(g)d\theta\]
and $S^{(q)}$ is smoothing away the diagonal, we can integrate by parts with respect to $\theta$ and deduce that 
 there are open sets $U_1$ of $x$ in $X$ and $V_1$ of $y$ in $X$ such that 
\begin{equation}\label{e-gue170305}
S^{(q)}_{G,m}=O(m^{-\infty})\ \ \mbox{on $U_1\times V_1$}. 
\end{equation}
From \eqref{e-gue170304}, we see that  there are open sets $\hat U_G$ of $\pi(x)$ in $Y_G$ and $\hat V_G$ of $\pi(y)$ in $Y_G$ such that 
\begin{equation}\label{e-gue170305I}
S^{(q-r)}_{Y_G,m}=O(m^{-\infty})\ \ \mbox{on $\hat U_G\times\hat V_G$}. 
\end{equation}
From \eqref{e-gue170305} and \eqref{e-gue170305I}, we get \eqref{e-gue170304ryII}. 

Fix $u=(u_1,\ldots,u_{2n+1})\in Y\bigcap U$. From \eqref{e-gue170304ryI} and \eqref{e-gue170304ryII}, we only need to show that \eqref{e-gue170304ryIII} and \eqref{e-gue170304rc} hold near $u$ and  we may assume that $u=(0,\ldots,0,u_{2d+1},\ldots,u_{2n},0)=\Td{\mathring{u}}''$. Let $V$ be a small neighborhood of $u$. 
 Let $\chi(\Td{\mathring{x}}'')\in C^\infty_0(\Omega_3)$. From \eqref{e-gue170227c}, we can extend $\chi(\Td{\mathring{x}}'')$ to 
\[Q=\set{g\circ e^{i\theta}\circ x;\, (g,e^{i\theta})\in G\times S^1, x\in \Omega_3}\]
by $\chi(g\circ e^{i\theta}\circ\Td{\mathring{x}}''):=\chi(\Td{\mathring{x}}'')$, for every $(g,e^{i\theta})\in G\times S^1$. Assume that $\chi=1$ on some neighborhood of $V$. Let $V_G=\set{\pi(x);\, x\in V}$. 
Let $\chi_1\in C^\infty_0(Y_G)$ with $\chi_1=1$ on some neighborhood 
of $V_G$ and ${\rm Supp\,}\chi_1\subset\set{\pi(x)\in Y_G;\, x\in Y, \chi(x)=1}$. We have 
\begin{equation}\label{e-gue170227bq}
\begin{split}
\chi_1\sigma_m&=m^{-\frac{d}{4}}\chi_1S^{(q-r)}_{Y_G,m}\circ\iota_G\circ \hat p\circ\tau_{x,n_-}\circ f\circ \iota^*\circ S^{(q)}_{G,m}\\
&=m^{-\frac{d}{4}}\chi_1S^{(q-r)}_{Y_G,m}\circ\iota_G\circ \hat p\circ\tau_{x,n_-}\circ f\circ \iota^*\circ \chi S^{(q)}_{G,m}\\
&\quad+m^{-\frac{d}{4}}\chi_1S^{(q-r)}_{Y_G,m}\circ\iota_G\circ \hat p\circ\tau_{x,n_-}\circ f\circ \iota^*\circ(1-\chi)S^{(q)}_{G,m}.
\end{split}
\end{equation}
If $u\in Y$ but $u\notin\set{x\in X;\, \chi(x)=1}$. Since ${\rm Supp\,}\chi_1\subset\set{\pi(x)\in X;\, x\in Y, \chi(x)=1}$ and $\chi(x)=\chi(g\circ e^{i\theta}\circ x)$, for every $(g,e^{i\theta})\in G\times S^1$, for every $x\in X$, we conclude that $\pi(e^{i\theta}\circ u)\notin{\rm Supp\,}\chi_1$, for every $e^{i\theta}\in S^1$. From this observation and \eqref{e-gue170304}, we get 
\begin{equation}\label{e-gue170227bIIq}
m^{-\frac{d}{4}}\chi_1S^{(q-r)}_{Y_G,m}\circ\iota_G\circ \hat p\circ\tau_{x,n_-}\circ f\circ \iota^*\circ(1-\chi)S^{(q)}_{G,m}=O(m^{-\infty})\ \ \mbox{on $Y_G\times X$}. 
\end{equation}
From \eqref{e-gue170227bq} and \eqref{e-gue170227bIIq}, we get 
\begin{equation}\label{e-gue170227bIIIq}
\chi_1\sigma_m=m^{-\frac{d}{4}}\chi_1S^{(q-r)}_{Y_G,m}\circ\iota_G\circ \hat p\circ\tau_{x,n_-}\circ f\circ \iota^*\circ \chi S^{(q)}_{G,m}+O(m^{\infty})\ \ \mbox{on $Y_G\times X$}. 
\end{equation}
From \eqref{e-gue170304ry} and Theorem~\ref{t-gue170128I}, we can check that on $U$, 
\begin{equation}\label{e-gue170227yq}
\begin{split}
&\chi_1\sigma_m(\Td x'',y)\\
&=(2\pi)\int e^{im\Psi(\Td x'',\Td v'')+im\Psi(v'',y)}\chi_1(\Td x)\beta(\Td x'',\Td{\mathring{v}}'',m)\hat b(\Td{\mathring{v}}'',y,m)dv(\Td{\mathring{v}}'')+O(m^{-\infty}),
\end{split}
\end{equation}
where $\hat b(\Td{\mathring{v}}'',y,m)=\Bigr(\iota_G\circ \hat p\circ\tau_{x,n_-}\circ f\circ \iota^*\circ \chi(\Td{\mathring{v}}'')\circ b\Bigr)(\Td{\mathring{v}}'',y,m)$. From \eqref{e-gue170227yq} and Theorem~\ref{t-gue170301wI}, we see that \eqref{e-gue170304ryIII} and \eqref{e-gue170304rc} hold near $u$. The theorem follows. 
\end{proof}

Let  
\begin{equation}\label{e-gue170305a}
\begin{split}
&F_m:=\sigma_m^*\sigma_m:\Omega^{0,q}(X)\To H^{q}_{b,m}(X)^G,\\
&\hat F_m:=\sigma_m\sigma^*_m:\Omega^{0,q-r}(Y_G)\To H^{q-r}_{b,m}(Y_G).
\end{split}
\end{equation}
Let $F_m(x,y)$ and $\hat F_m(x,y)$ be the distribution kernels of $F_m$ and $\hat F_m$ respectively. From Theorem~\ref{t-gue170301w}, Theorem~\ref{t-gue170301wI}, we can repeat the proof of Theorem~\ref{t-gue170304ry} with minor change and deduce the following two theorems

\begin{theorem}\label{t-gue170305a}
With the notations used above, if $y\notin Y$, then for any open set $D$ of $y$ with $\ol D\bigcap Y=\emptyset$, we have
\begin{equation}\label{e-gue170304ryIp}
F_m=O(m^{-\infty})\ \ \mbox{on $X\times D$.}
\end{equation}

Let $x, y\in Y$.  If $\pi(e^{i\theta}\circ x)\neq\pi(e^{i\theta}\circ y)$, for every $\theta\in[0,2\pi[$, then there are open sets $D_1$ of $x$ in $X$ and $D_2$ of $y$ in $X$ such that 
\begin{equation}\label{e-gue170305b}
F_m=O(m^{-\infty})\ \ \mbox{on $D_1\times D_2$}. 
\end{equation}

Let $p\in\mu^{-1}(0)$  and let $x=(x_1,\ldots,x_{2n+1})$ be the local coordinates as in Remark~\ref{r-gue170309}. Then, 
\begin{equation}\label{e-gue170305bI}
\begin{split}
&F_m(x,y)=e^{im\Psi(x'',y'')}a(x'',y'',m)+O(m^{-\infty})\ \ \mbox{on $U\times U$},\\
&a(x'',y'',m)\in S^{n-\frac{d}{2}}_{{\rm loc\,}}(1; U\times U, T^{*0,q}X\boxtimes (T^{*0,q}X)^*),\\
&\mbox{$a(x'',y'',m)\sim\sum^\infty_{j=0}m^{n-\frac{d}{2}-j}a_j(\Td x'',y'')$ in $S^{n-\frac{d}{2}}_{{\rm loc\,}}(1; U\times U, T^{*0,q}X\boxtimes(T^{*0,q}X)^*)$},\\
&a_j(x'',y'')\in C^\infty(U\times U, T^{*0,q}X\boxtimes(T^{*0,q}X)^*),\ \ j=0,1,2,\ldots,
\end{split}
\end{equation}
\begin{equation}\label{e-gue170305bII}
a_0(\Td x'',\Td x'')=2^{-3n+4d-1}\pi^{-n-1}\frac{1}{V_{{\rm eff\,}}(\Td x'')}\abs{\det\,\mathcal{L}_{\Td x''}}\abs{\det\,R_x}^{-\frac{1}{2}}\tau_{\Td x'',n_-},\ \ \forall \Td x''\in U\bigcap Y,
\end{equation}
where $U$ is an open set of $p$. 
\end{theorem}

\begin{theorem}\label{t-gue170305aI}
Let $x, y\in Y$.  If $\pi(e^{i\theta}\circ x)\neq\pi(e^{i\theta}\circ y)$, for every $\theta\in[0,2\pi[$, then there are open sets $D_G$ of $\pi(x)$ in $Y_G$ and $V_G$ of $\pi(y)$ in $Y_G$ such that 
\begin{equation}\label{e-gue170305bp}
\hat F_m=O(m^{-\infty})\ \ \mbox{on $D_G\times V_G$}. 
\end{equation}

Let $p\in\mu^{-1}(0)$  and let $x=(x_1,\ldots,x_{2n+1})$ be the local coordinates as in Remark~\ref{r-gue170309}. Then, 
\begin{equation}\label{e-gue170305bIp}
\begin{split}
&\hat F_m(x,y)=e^{im\Psi(\Td x'',\Td y'')}\hat a(\Td x'',\Td y'',m)+O(m^{-\infty})\ \ \mbox{on $W\times W$},\\
&\hat a(\Td x'',\Td y'',m)\in S^{n-d}_{{\rm loc\,}}(1; W\times W, T^{*0,q-r}Y_G\boxtimes(T^{*0,q-r}Y_G)^*),\\
&\mbox{$\hat a(\Td x'',\Td y'',m)\sim\sum^\infty_{j=0}m^{n-d-j}\hat a_j(\Td x'',\Td y'')$ in $S^{n-d}_{{\rm loc\,}}(1; W\times W, T^{*0,q-r}Y_G\boxtimes(T^{*0,q-r}Y_G)^*)$},\\
&\hat a_j(\Td x'',\Td y'')\in C^\infty(W\times W, T^{*0,q-r}Y_G\boxtimes(T^{*0,q-r}Y_G)^*),\ \ j=0,1,2,\ldots,
\end{split}
\end{equation}
\begin{equation}\label{e-gue170305bIIp}
\hat a_0(\Td x'',\Td x'')=2^{-3n+\frac{5}{2}d-1}\pi^{-n+\frac{d}{2}-1}\abs{\det\,\mathcal{L}_{Y_G,\Td x''}}\hat\tau_{\Td x''},\ \ \forall \Td x''\in W
\end{equation}
where $W=\Omega_3\times\Omega_4$, $\Omega_3$ and $\Omega_4$ are open sets as in the beginning of Section~\ref{s-gue170226}. 
\end{theorem}

Let $R_m:=\frac{1}{C_0}F_m-S^{(q)}_{G,m}:\Omega^{0,q}(X)\To H^q_{b,m}(X)^G$, where $C_0=2^{-3d+3n}\pi^{\frac{d}{2}}$. Since $F_m=F_mS^{(q)}_{G,m}$, it is clear that 
\begin{equation}\label{e-gue170305i}
\frac{1}{C_0}F_m=S^{(q)}_{G,m}+R_m=S^{(q)}_{G,m}+R_mS^{(q)}_{G,m}=(I+R_m)S^{(q)}_{G,m}.
\end{equation}

Our next goal is to show that for $m$ large, $I+R_m:\Omega^{0,q}(X)\To \Omega^{0,q}(X)$ is injective. 
From Theorem~\ref{t-gue170305a} and Theorem~\ref{t-gue170128I}, we see that if $y\notin Y$, then for any open set $D$ of $y$ with $\ol D\bigcap Y=\emptyset$, we have
\begin{equation}\label{e-gue170305f}
R_m=O(m^{-\infty})\ \ \mbox{on $X\times D$.}
\end{equation}
Let $x, y\in Y$.  If $\pi(e^{i\theta}\circ x)\neq\pi(e^{i\theta}\circ y)$, for every $\theta\in[0,2\pi[$, then there are open sets $D_1$ of $x$ in $X$ and $D_2$ of $y$ in $X$ such that 
\begin{equation}\label{e-gue170305fI}
R_m=O(m^{-\infty})\ \ \mbox{on $D_1\times D_2$}. 
\end{equation}

Let $p\in\mu^{-1}(0)$  and let $x=(x_1,\ldots,x_{2n+1})$ be the local coordinates as in Remark~\ref{r-gue170309}. Then, 
\begin{equation}\label{e-gue170305fII}
\begin{split}
&R_m(x,y)=e^{im\Psi(x'',y'')}r(x'',y'',m)+O(m^{-\infty})\ \ \mbox{on $U\times U$},\\
&r(x'',y'',m)\in S^{n-\frac{d}{2}}_{{\rm loc\,}}(1; U\times U, T^{*0,q}X\boxtimes(T^{*0,q}X)^*),\\
&\mbox{$r(x'',y'',m)\sim\sum^\infty_{j=0}m^{n-\frac{d}{2}-j}r_j(x'',y'')$ in $S^{n-\frac{d}{2}}_{{\rm loc\,}}(1; U\times U, T^{*0,q}X\boxtimes(T^{*0,q}X)^*)$},\\
&r_j(x'',y'')\in C^\infty(U\times U, T^{*0,q}X\boxtimes(T^{*0,q}X)^*),\ \ j=0,1,2,\ldots.
\end{split}
\end{equation}
Moreover, from \eqref{e-gue170305bII} and \eqref{e-gue170117pVIIIam}, it is easy to see 
\begin{equation}\label{e-gue170305fIII}
\abs{r_0(x,y)}\leq C\abs{(x,y)-(x_0,x_0)},
\end{equation}
for all $x_0\in\mu^{-1}(0)\bigcap U$, where $C>0$ is a constant. We need 

\begin{lemma}\label{l-gue170306s}
Let $p\in\mu^{-1}(0)$  and let $x=(x_1,\ldots,x_{2n+1})$ be the local coordinates as in Remark~\ref{r-gue170309} defined in an open set $U$ of $p$. Let
\[
\begin{split}
&H_m(x,y)=e^{im\Psi(x'',y'')}h(x,y,m)\ \ \mbox{on $U\times U$},\\
&h(x,y,m)\in S^{n-1-\frac{d}{2}}_{{\rm loc\,}}(1; U\times U, T^{*0,q}X\boxtimes(T^{*0,q}X)^*),\\
&\mbox{$h(x,y,m)\sim\sum^\infty_{j=0}m^{n-1-\frac{d}{2}-j}h_j( x,y)$ in $S^{n-1-\frac{d}{2}}_{{\rm loc\,}}(1; U\times U, T^{*0,q}X\boxtimes(T^{*0,q}X)^*)$},\\
&h_j(x,y)\in C^\infty_0(U\times U, T^{*0,q}X\boxtimes(T^{*0,q}X)^*),\ \ j=0,1,2,\ldots.
\end{split}\]
Assume that $h(x,y,m)\in C^\infty_0(U\times U, T^{*0,q}X\boxtimes(T^{*0,q}X)^*)$. Then, there is a constant $\hat C>0$ independent of $m$ such that 
\begin{equation}\label{e-gue170306s}
\norm{H_mu}\leq \delta_m\norm{u},\ \ \forall u\in\Omega^{0,q}(X),\ \ \forall m\in\mathbb N,
\end{equation}
where $\delta_m$ is a sequence with $\lim_{m\To\infty}\delta_m=0$. 
\end{lemma}

\begin{proof}
Fix $N\in\mathbb N$. It is not difficult to see that 
\begin{equation}\label{e-gue170306Ix}
\norm{H_mu}\leq\norm{(H^*_mH_m)^{2^N}u}^{\frac{1}{2^{N+1}}}\norm{u}^{1-\frac{1}{2^{N+1}}},\ \ \forall u\in\Omega^{0,q}(X),
\end{equation}
where $H^*_m$ denotes the adjoint of $H_m$. From Theorem~\ref{t-gue170301w}, we can repeat the proof of Theorem~\ref{t-gue170304ry} with minor change and deduce that 
\begin{equation}\label{e-gue170306xII}
\begin{split}
&(H^*_mH_m)^{2^N}(x,y)=e^{im\Psi(x'',y'')}p (x,y,m)+O(m^{-\infty})\ \ \mbox{on $U\times U$},\\
&p(x,y,m)\in S^{n-2^{N+1}-\frac{d}{2}}_{{\rm loc\,}}(1; U\times U, T^{*0,q}X\boxtimes(T^{*0,q}X)^*),\\
&p(x,y,m)\in C^\infty_0(U\times U, T^{*0,q}X\boxtimes(T^{*0,q}X)^*).
\end{split}
\end{equation}
Hence, 
\begin{equation}\label{e-gue170306xIII}
\abs{(H^*_mH_m)^{2^N}(x,y)}\leq \hat Cm^{n-2^{N+1}-\frac{d}{2}},\ \ \forall (x,y)\in U\times U,
\end{equation}
where $\hat C>0$ is a constant independent of $m$. Take $N$ large enough so that $n-2^{N+1}-\frac{d}{2}<0$. From \eqref{e-gue170306xIII} and \eqref{e-gue170306Ix}, 
we get \eqref{e-gue170306s}. 
\end{proof}

We also need 

\begin{lemma}\label{l-gue170306}
Let $p\in\mu^{-1}(0)$  and let $x=(x_1,\ldots,x_{2n+1})$ be the local coordinates as in Remark~\ref{r-gue170309} defined in an open set $U$ of $p$. Let \[
\begin{split}
&B_m(x,y)=e^{im\Psi(x'',y'')}g(x,y,m)\ \ \mbox{on $U\times U$},\\
&g(x,y,m)\in S^{n-\frac{d}{2}}_{{\rm loc\,}}(1; U\times U, T^{*0,q}X\boxtimes(T^{*0,q}X)^*),\\
&\mbox{$g(x,y,m)\sim\sum^\infty_{j=0}m^{n-\frac{d}{2}-j}g_j( x,y)$ in $S^{n-\frac{d}{2}}_{{\rm loc\,}}(1; U\times U, T^{*0,q}X\boxtimes(T^{*0,q}X)^*)$},\\
&g_j(x,y)\in C^\infty_0(U\times U, T^{*0,q}X\boxtimes(T^{*0,q}X)^*),\ \ j=0,1,2,\ldots,\\
&g(x,y)\in C^\infty_0(U\times U, T^{*0,q}X\boxtimes(T^{*0,q}X)^*).
\end{split}\]
Suppose that 
\begin{equation}\label{e-gue170306}
\abs{g_0(x,y)}\leq C\abs{(x,y)-(x_0,x_0)},
\end{equation}
for all $x_0\in\mu^{-1}(0)\bigcap U$, where $C>0$ is a constant. Then, 
\begin{equation}\label{e-gue170306I}
\norm{B_mu}\leq \varepsilon_m\norm{u},\ \ \forall u\in\Omega^{0,q}(X),\ \ \forall m\in\mathbb N,
\end{equation}
where $\varepsilon_m$ is a sequence with $\lim_{m\To\infty}\varepsilon_m=0$. 
\end{lemma}

\begin{proof}
Fix $N\in\mathbb N$. It is not difficult to see that 
\begin{equation}\label{e-gue170306Id}
\norm{B_mu}\leq\norm{(B^*_mB_m)^{2^N}u}^{\frac{1}{2^{N+1}}}\norm{u}^{1-\frac{1}{2^{N+1}}},\ \ \forall u\in\Omega^{0,q}(X),
\end{equation}
where $B^*_m$ denotes the adjoint of $B_m$. From Theorem~\ref{t-gue170301w}, we can repeat the proof of Theorem~\ref{t-gue170304ry} with minor change and deduce that 
\begin{equation}\label{e-gue170306II}
\begin{split}
&(B^*_mB_m)^{2^N}(x,y)=e^{im\Psi(x'',y'')}\hat g(x,y,m)+O(m^{-\infty})\ \ \mbox{on $U\times U$},\\
&\hat g(x,y,m)\in S^{n-\frac{d}{2}}_{{\rm loc\,}}(1; U\times U, T^{*0,q}X\boxtimes(T^{*0,q}X)^*),\\
&\mbox{$\hat g(x,y,m)\sim\sum^\infty_{j=0}m^{n-\frac{d}{2}-j}\hat g_j(x,y)$ in $S^{n-\frac{d}{2}}_{{\rm loc\,}}(1; U\times U, T^{*0,q}X\boxtimes(T^{*0,q}X)^*)$},\\
&\hat g_j(x,y)\in C^\infty_0(U\times U, T^{*0,q}X\boxtimes(T^{*0,q}X)^*),\ \ j=0,1,2,\ldots, \\
&\hat g(x,y,m)\in C^\infty_0(U\times U, T^{*0,q}X\boxtimes(T^{*0,q}X)^*),
\end{split}
\end{equation}
and 
\begin{equation}\label{e-gue170306III}
\abs{\hat g_0(x,y)}\leq C\abs{(x,y)-(x_0,x_0)}^{2^{N+1}},
\end{equation}
for all $x_0\in\mu^{-1}(0)\bigcap U$, where $C>0$ is a constant. Let 
\[\begin{split}
(B^*_mB_m)^{2^N}_0(x,y)=e^{im\Psi(x'',y'')}\hat g_0(x,y,m),\\
(B^*_mB_m)^{2^N}_1(x,y)=e^{im\Psi(x'',y'')}h(x,y,m),
\end{split}\]
where $h(x,y,m)=\hat g(x,y,m)-\hat g_0(x,y,m)$. It is clear that 
\[h(x,y,m)\in S^{n-1-\frac{d}{2}}_{{\rm loc\,}}(1; U\times U, T^{*0,q}X\boxtimes(T^{*0,q}X)^*).\]
From Lemma~\ref{l-gue170306s}, we see that 
\begin{equation}\label{e-gue170306e}
\norm{(B^*_mB_m)^{2^N}_1u}\leq \delta_m\norm{u},\ \ \forall u\in\Omega^{0,q}(X),\ \ \forall m\in\mathbb N,
\end{equation}
where $\delta_m$ is a sequence with $\lim_{m\To\infty}\delta_m=0$. 

From \eqref{e-gue170306III}, we see that 
\begin{equation}\label{e-gue170306r}
\abs{\hat g_0(x,y)}\leq C_1\Bigr(\abs{\hat x''}+\abs{\hat y''}+\abs{\Td{\mathring{x}}''-\Td{\mathring{y}}''}\Bigr)^{2^{N+1}},
\end{equation}
where $C_1>0$ is a constant. From \eqref{e-gue170106m}, we see that 
\begin{equation}\label{e-gue170306rI}
\abs{{\rm Im\,}\Psi(x,y)}\geq c\Bigr(\abs{\hat x''}^2+\abs{\hat y''}^2+\abs{\Td{\mathring{x}}''-\Td{\mathring{y}}''}^2\Bigr),
\end{equation}
where $c>0$ is a constant. From \eqref{e-gue170306r} and \eqref{e-gue170306rI}, we conclude that 
\begin{equation}\label{e-gue170306rII}
\abs{(B^*_mB_m)^{2^N}_0(x,y)}\leq \hat Cm^{-2^N+n-\frac{d}{2}},\ \ \forall (x,y)\in U\times U,
\end{equation}
where $\hat C>0$ is a constant independent of $m$. From \eqref{e-gue170306rII}, we see that if $N$ large enough, then
\begin{equation}\label{e-gue170306ryIII}
\norm{(B^*_mB_m)^{2^N}_0u}\leq\hat\delta_m\norm{u},\ \ \forall u\in\Omega^{0,q}(X),\ \ \forall m\in\mathbb N,
\end{equation}
where $\hat\delta_m$ is a sequence with $\lim_{m\To\infty}\hat\delta_m=0$. 

From \eqref{e-gue170306Id}, \eqref{e-gue170306e} and \eqref{e-gue170306ryIII}, we get \eqref{e-gue170306I}. 
\end{proof}

From \eqref{e-gue170305f}, \eqref{e-gue170305fI}, \eqref{e-gue170305fII}, \eqref{e-gue170305fIII} and Lemma~\ref{l-gue170306}, we get 

\begin{theorem}\label{t-gue170306j}
With the notations above, we have 
\[
\norm{R_mu}\leq \varepsilon_m\norm{u},\ \ \forall u\in\Omega^{0,q}(X),\ \ \forall m\in\mathbb N,\]
where $\varepsilon_m$ is a sequence with $\lim_{m\To\infty}\varepsilon_m=0$. 

In particular, if $m$ is large enough, then the map 
\[I+R_m:\Omega^{0,q}(X)\To \Omega^{0,q}(X)\]
is injective. 
\end{theorem}

\begin{proof} [Proof of Theorem~\ref{t-gue170122}]
From \eqref{e-gue170305i} and Theorem~\ref{t-gue170306j}, we see that if $m$ is large enough, then the map
\[F_m=\sigma^*_m\sigma_m: H^q_{b,m}(X)^G\To H^q_{b,m}(X)^G\]
is injective. Hence, if $m$ is large enough, then the map
\[\sigma_m:H^q_{b,m}(X)^G\To H^{q-r}_{b,m}(Y_G)\]
is injective and ${\rm dim\,}H^q_{b,m}(X)^G\leq{\rm dim\,}H^{q-r}_{b,m}(Y_G)$. Similarly, we can repeat the proof of Theorem~\ref{t-gue170306j} with minor change and deduce that, if $m$ is large enough, then the map
\[\hat F_m=\sigma_m\sigma^*_m: H^{q-r}_{b,m}(Y_G)\To H^{q-r}_{b,m}(Y_G)\]
is injective. Hence, if $m$ is large enough, then the map
\[\sigma^*_m:H^{q-r}_{b,m}(Y_G)\To H^{q}_{b,m}(X)^G\]
is injective. Thus,  ${\rm dim\,}H^q_{b,m}(X)^G={\rm dim\,}H^{q-r}_{b,m}(Y_G)$ and $\sigma_m$ is an isomorphism if $m$ large enough. 
\end{proof}

\bibliographystyle{plain}

\end{document}